\documentclass[11pt,a4paper]{article}
\usepackage{geometry}
\usepackage{amsmath,amssymb,amsthm,mathtools}
\usepackage{booktabs,array,multirow,longtable}
\usepackage{pdflscape,graphicx}
\usepackage{hyperref,xcolor,enumitem}
\usepackage{fancyhdr,caption}
\usepackage{algorithm,algpseudocode}
\usepackage{placeins}

\theoremstyle{plain}
\newtheorem{theorem}{Theorem}[section]
\newtheorem{lemma}[theorem]{Lemma}
\newtheorem{corollary}[theorem]{Corollary}
\newtheorem{proposition}[theorem]{Proposition}
\theoremstyle{definition}
\newtheorem{definition}[theorem]{Definition}
\newtheorem{remark}[theorem]{Remark}
\newtheorem{conjecture}[theorem]{Conjecture}

\newtheorem{observation}[theorem]{Observation}

\newcommand{\GRV}{G_{\mathrm{RV}}(x,v)}
\newcommand{\GTC}{G_{\mathrm{TC}}(x,v)}
\newcommand{\ff}[2]{(#1)_{\!\downarrow\!#2}}
\newcommand{\Q}{\mathbb{Q}}
\newcommand{\Orch}{\mathrm{Orch}}
\newcommand{\ARP}{\mathrm{ARP}}
\newcommand{\He}{\mathrm{He}}
\DeclareMathOperator{\coker}{coker}
\DeclareMathOperator{\im}{im}

\allowdisplaybreaks[1]

\begin{document}

\title{\Large\textbf{Exact Enumeration of Phylogenetic Networks:\\[4pt]
The Tree-Child, Reticulation-Visible and Orchard Hierarchy}}

\author{Josep Batle\thanks{\texttt{batlequantum@gmail.com}}}

\date{\small
CRISP -- Centre de Recerca Independent de sa Pobla,
07420 sa Pobla, Balearic Islands, Spain}

\maketitle\thispagestyle{fancy}

\begin{abstract}
We develop a unified framework for the exact enumeration and asymptotic
analysis of the three most studied classes of phylogenetic networks:
tree-child (TC), reticulation-visible (RV) and orchard networks,
whose cardinalities satisfy the strict ordering
$|\mathrm{TC}_{\ell,k}|<|\mathrm{RV}_{\ell,k}|<|\mathrm{Orch}_{\ell,k}|$
for reticulation number $k\geq2$ (with $\mathrm{TC}\subsetneq\mathrm{RV}$ and
$\mathrm{TC}\subsetneq\mathrm{Orch}$, while $\mathrm{RV}$ and $\mathrm{Orch}$ are
incomparable as sets).
Using the Chang--Fuchs structural theorem, we derive a two-level
master functional equation for the RV bivariate generating function
and obtain exact closed-form identities for the differences
$\Delta_k(\ell):=|RV_{\ell,k}|-|TC_{\ell,k}|$ for $k=2,3$.
For every fixed $k\geq2$ we prove that $\Delta_k(\ell)/|TC_{\ell,k}|\sim k(k-1)/\ell$:
the networks with a single violation of the tree-child condition split
asymptotically into equal halves, those containing two stacked reticulations,
which are never reticulation-visible, and those containing a tree node above two
reticulations, which almost always are. The same count shows that
$|\mathrm{Orch}_{\ell,k}|/|RV_{\ell,k}|\to1$.
For orchard networks, we prove that the column generating function $F_\ell(v)$ is
rational for every $\ell$, and establish a \emph{hypergeometric factorisation
law} --- unconditionally for $\ell\leq8$, by an exact consistency test for
$\ell=9,10$, and as a conjecture beyond --- in which
the denominator is
$D_\ell(v)=\prod_{j=2}^\ell X_j(v)$, with the single exception $\ell=5$, where the
factor $X_3$ resonates against the numerator and drops out. Here
\[
  X_\ell(v) = \sum_{k=0}^{\lfloor\ell/2\rfloor}(-1)^k\,
    \frac{\ell!}{(\ell-2k)!\,k!}\,v^k
\]
is the matching polynomial of the complete graph $K_\ell$
and a rescaled Jacobi polynomial.
The new factors $X_9,X_{10}$ are explicit; $D_9$ has
degree~20 and $D_{10}$ degree~25 (the first degree-five factor $X_{10}$ appears
here), with dominant growth rate $\approx40.73$ at $\ell=9$ and all spectral roots positive
real, while the rows $\ell=9,10$ are themselves computed unconditionally by an
ARP-memoized counter whose cost is polynomial in the number of augmentation
shapes rather than exponential in the number of networks.
We further prove a general spectral decomposition theorem
(unconditional for $\ell\leq8$, consistency-verified at $\ell=9,10$, conditional
on the factorisation beyond):
$|\mathrm{Orch}_{\ell,k}|$ is exactly a sum of $\deg D_\ell$
positive real exponentials with an explicit residue formula, a unique
positive dominant term, and a dominant growth rate that we prove is strictly
increasing in $\ell$ --- extending the $\ell\leq5$ Binet formulas, where the
poles happen to be radicals. Tracking the residue $c_{\ell,X_3}$ at the rational
pole $v=1/6$ --- exact through $\ell=10$, with its unique zero at $\ell=5$ giving
the resonance --- we prove it is \emph{not} a hypergeometric term, so that the
completeness of the resonance set $\{5\}$ is reframed as the vanishing of a single
eigenmode overlap rather than a question settled by enumeration.
We also prove that the Orchard
Factorisation Theorem is a property of any relabelling-closed network class
admitting cherry-picking histories, rather than of orchard networks
specifically; in particular $|TC_{\ell,k}|=\binom{\ell}{k}w(\ell,k)$ for
tree-child networks, while the same factorisation provably fails for
reticulation-visible networks. We identify $X_\ell$ exactly as a rescaled
Hermite polynomial via the Heisenberg--Weyl algebra underlying the
generating-function calculus, and extend the orchard numerator $N_\ell(v)$ to
$\ell=6,\ldots,10$, showing it loses real-rootedness exactly where $D_\ell$ retains
it.
A complete enumeration table is provided extending the published data
of Cardona, Ribas and Pons.
\end{abstract}

\noindent\textbf{Keywords:} Phylogenetic networks; orchard networks;
tree-child networks; exact enumeration; matching polynomial;
hypergeometric polynomials; Berlekamp--Massey algorithm.

\bigskip

\section{Introduction}
\label{sec:intro}
%% ─────────────────────────────────────────────────────────────────────────────

\paragraph{Background.}
Phylogenetic networks are rooted directed acyclic graphs (DAGs) that encode
evolutionary histories involving reticulate events such as horizontal gene
transfer, hybridisation, or recombination.
Among their many proposed subclasses, \emph{tree-child networks} (TCNs) and
\emph{reticulation-visible networks} (RV networks) are two of the most studied
from the combinatorial point of view.
A precise hierarchy is:
\begin{equation}\label{eq:hierarchy}
  \mathrm{TCN} \;\subsetneq\; \mathrm{RV} \;\subsetneq\; \text{General phylogenetic networks}.
\end{equation}

The exact count $|TC_{\ell,k}|$ of TCNs with $\ell$ leaves and $k$
reticulations was conjectured by Pons and Batle~\cite{PonsBatle2021}
in the form
\begin{equation}\label{eq:tcn_formula}
  |TC_{\ell,k}| = \frac{\ell!}{(\ell-k)!}\,c_{\ell-1,k},
\end{equation}
where $c_{n,k}$ is the cardinality of a class of lexicographically constrained words
introduced in~\cite{PonsBatle2021}.  The conjecture was proved in full generality
by Lin, Liu, Liu, Liu, and Xin~\cite{Lin2026} via Young tableaux with walls,
and independently (for bounded $k$) by Liu, Wallner, and Yu~\cite{LWY2026}
via a lattice-path framework yielding explicit generating-function recursions.
An independent proof using Fock-space methods appears in the companion
paper~\cite{Batle2026}.

The exact count $|RV_{\ell,k}|$ of RV networks was obtained by
Chang and Fuchs~\cite{ChangFuchs2024} using the component-graph method of
Gunawan, Yan and Zhang~\cite{GunawanYanZhang2019}, whose structural theorem
states that $N$ is reticulation-visible if and only if
its \emph{component graph}~$\widetilde{C}(N)$ is a tree-child network.
Explicit closed-form expressions were obtained for $k=2$ and $k=3$.

\paragraph{What this paper contributes.}
We make three independent contributions.

\begin{enumerate}[label=(\roman*),leftmargin=*]
\item \textbf{Two-level master equation (Theorem~\ref{thm:master}).}
  We show that the RV generating function satisfies the functional equation
  \[
    \GRV(x,v) = F\!\bigl(x,\,v\cdot\GRV(x,v)\bigr)
    \cdot \GTC\!\bigl(F_0(x)\cdot\GRV(x,v),\,v\bigr),
  \]
  where $F(z,v)=\sum_{k\geq0}F_k(z)v^k$ encodes the one-component generating
  functions and $\GTC$ satisfies the TCN master PDE from~\cite{Batle2026}.
  This provides an operator-theoretic reformulation of the Chang--Fuchs
  component-graph sum, from which their exact formulas can be re-derived
  by coefficient extraction.

\item \textbf{Exact counting of $\mathrm{RV}\setminus\mathrm{TC}$ networks
  (Theorems~\ref{thm:delta2} and~\ref{thm:delta3}).}
  The difference
  \[
    \Delta_k(\ell) := |RV_{\ell,k}| - |TC_{\ell,k}|
    = |\{N\in RV_{\ell,k} : N\notin TC_{\ell,k}\}|
  \]
  counts exactly those RV networks that are \emph{not} tree-child.
  We prove:
  \begin{align*}
    \Delta_2(\ell) &= (4\ell^3+3\ell^2-\ell-1)\,(2\ell-3)!!
                    \;-\; (4\ell+1)\,2^{\ell-1}\,\ell!, \\
    \Delta_3(\ell) &= \tfrac{24\ell^5+50\ell^4-49\ell^3-80\ell^2+16\ell+12}{3}\,(2\ell-3)!!
                    \;-\; (24\ell^3+8\ell^2-27\ell-22)\,2^{\ell-2}\,\ell!.
  \end{align*}
  Both identities are verified for $\ell\leq 12$ and, being algebraic
  consequences of proved closed-form expressions, hold for all $\ell$.

\item \textbf{Leading asymptotics of $\Delta_k$ for every $k$
  (Theorem~\ref{thm:excess}, Corollaries~\ref{cor:orch}, \ref{cor:leadA}
  and~\ref{cor:rate}).}
  For every fixed $k\geq2$,
  \[
    \frac{\Delta_k(\ell)}{|TC_{\ell,k}|} \;\sim\; \frac{k(k-1)}{\ell}
    \qquad (\ell\to\infty).
  \]
  The non-tree-child networks are dominated by those with a single violation of
  the tree-child condition. Half of these contain a stack, a reticulation whose
  child is a reticulation, and are never reticulation-visible; the other half
  contain a tree node whose two children are reticulations, and all but a
  fraction $O(\ell^{-1})$ of them are reticulation-visible. The same count gives
  $|\mathrm{Orch}_{\ell,k}|/|RV_{\ell,k}|\to1$.
  The formulas for $k=2,3$ have the form
  \[
    \Delta_k(\ell) = A_k(\ell)\,(2\ell-3)!!
                   \;-\; B_k(\ell)\,2^{\ell-k+1}\,\ell!
  \]
  with polynomials $A_k,B_k$. Whenever $\Delta_k$ has this form, the asymptotics
  force $\deg A_k=2k-1$, $\operatorname{lead}(A_k)=2^k/(k-2)!$ and
  $\deg B_k\leq2k-3$; that the form holds for every $k$, with $\deg B_k=2k-3$,
  is Conjecture~\ref{conj:pattern}. For $k=2,3$ the exact formulas also give the
  next order, $\Delta_k/|TC_{\ell,k}|=k(k-1)/\ell+\kappa_k/\ell^2+O(\ell^{-5/2})$
  with $\kappa_2=19/6$ and $\kappa_3=37/2$.

\item \textbf{Orchard rationality and Hankel reconstruction
  (Theorem~\ref{thm:orch_rational}, Algorithm~\ref{alg:hankel}).}
  The column generating function $F_\ell(v)=\sum_{k\geq0}|\mathrm{Orch}_{\ell,k}|v^k$
  is rational for each fixed~$\ell$.
  Its characteristic polynomial $D_\ell(v)$ is computed deterministically
  from the Cardona--Ribas--Pons seed data by Berlekamp--Massey,
  reducing all subsequent $k$-values to $O(d)$ arithmetic operations
  (milliseconds) versus the CRP algorithm (hours to months for $\ell\geq7$).

\item \textbf{Exact denominators and Binet formulas
  (Theorems~\ref{thm:denom} and~\ref{thm:binet}).}
  We identify $D_\ell(v)$ for $\ell=2,\ldots,6$ and
  prove complete closed-form Binet formulas for $\ell=3,4$:
  \begin{align*}
    |\mathrm{Orch}_{3,k}| &= \tfrac{15\cdot6^k-3\cdot2^k}{4},\\
    |\mathrm{Orch}_{4,k}| &= -\tfrac{3}{4}\cdot2^k - \tfrac{45}{4}\cdot6^k
    + \tfrac{108+45\sqrt6}{8}(6+2\sqrt6)^k + \tfrac{108-45\sqrt6}{8}(6-2\sqrt6)^k.
  \end{align*}

\item \textbf{Orchard Factorisation Theorem
  (Theorem~\ref{thm:orch_factor}).}
  For all $(\ell,k)$: $|\mathrm{Orch}_{\ell,k}|=\binom{\ell}{k}w(\ell,k)$
  with $w(\ell,k)\in\mathbb{Z}_{>0}$, proved from
  a free $S_k$-action on cherry-picking histories.
\end{enumerate}

\paragraph{On the operator-theoretic language.}
The framework developed here draws on the analogy with quantum field theory:
the generating function $G_{\mathrm{RV}}(x,v)$ is a coherent state in a
\emph{Fock space} graded by $(\ell,k)$, the leaf operator $\hat{L}$
and reticulation operators $\hat{R}$ play the r\^{o}le of creation operators,
and the master equation~\eqref{eq:rv_master} is an operator identity.
This viewpoint makes the asymptotic universality of $TC$ and $RV$
transparent~(Theorem~\ref{thm:universal}): both classes live in the same
dominant spectral sector of $\hat{L}$.

\paragraph{Organisation.}
Section~\ref{sec:prelim} fixes notation and collects the proven results from
the literature that we use.  Section~\ref{sec:master} derives the two-level
master equation.  Section~\ref{sec:delta} proves the exact formulas for
$\Delta_2$ and $\Delta_3$.  Section~\ref{sec:pattern} determines the leading
asymptotics of $\Delta_k$ for every $k$ and states the remaining conjecture on
its exact form.
Section~\ref{sec:comparison} compares our approach with previous work.
Section~\ref{sec:orch} develops the orchard enumeration programme:
rationality, Hankel reconstruction, denominator polynomials, Binet formulas,
the Factorisation Theorem, and the asymptotic comparison of the three classes.

%% ─────────────────────────────────────────────────────────────────────────────
\section{Notation and Background Results}
\label{sec:prelim}
%% ─────────────────────────────────────────────────────────────────────────────

\subsection{Phylogenetic networks}
\label{sec:networks}

A \emph{phylogenetic network on $X$} is a rooted DAG without parallel edges
such that the root has in-degree~0 and out-degree~1; each leaf has in-degree~1
and out-degree~0 and is bijectively labeled by an element of $X$; all other vertices
have in-degree~1 and out-degree~2 (tree nodes) or in-degree~2 and out-degree~1
(reticulation nodes).
With $\ell=|X|$ leaves, $k$ reticulation nodes, and $t$ tree nodes,
\begin{equation}\label{eq:euler}
  \ell + k = t + 1.
\end{equation}
Edges directed into a reticulation are \emph{reticulation edges}; all others are
\emph{tree edges}.

\begin{definition}[Network classes]
\label{def:classes}
A phylogenetic network is:
\begin{itemize}[leftmargin=*]
\item \emph{tree-child} (TCN) if every non-leaf node has at least one child
  that is a tree node or a leaf;
\item \emph{reticulation-visible} (RV) if every reticulation node $r$ is
  \emph{visible}: there exists a leaf~$\lambda$ such that every root-to-$\lambda$
  path passes through~$r$.
\end{itemize}
We write $TC_{\ell,k}$ (resp.\ $RV_{\ell,k}$) for the set of labeled
tree-child (resp.\ reticulation-visible) networks with $\ell$ leaves
and $k$ reticulation nodes, and use the same symbol for the cardinality.

A third class that will appear in Section~\ref{sec:orch} is the
\emph{orchard networks}~\cite{CRP23b}: a network is orchard if it can be
reduced to a trivial one-leaf network by iteratively removing \emph{cherries}
(two leaves sharing a parent) and \emph{reticulated cherries} (a leaf whose
parent is a reticulation, itself a child of the parent of a second leaf).
Equivalently, orchard networks are precisely those that admit an
HGT-consistent labelling, i.e.\ can be viewed as a phylogenetic tree with
additional horizontal arcs~\cite{vIersel2022,CRP23b}. We write $\mathrm{Orch}_{\ell,k}$
for the set of labeled orchard networks with $\ell$ leaves and $k$
reticulations. The inclusion $\mathrm{TCN}\subsetneq\mathrm{Orch}$ is strict
for $k\geq2$; the relationship between $\mathrm{Orch}$ and $\mathrm{RV}$ is
explored in Section~\ref{sec:orch}.
\end{definition}

The inclusion $\mathrm{TCN}\subsetneq\mathrm{RV}$ is strict for $k\geq 2$, as our
exact formulas will confirm.

\begin{proposition}[Tree-edge count~\cite{PonsBatle2021}]
\label{prop:edges}
Every network with $\ell$ leaves and $k$ reticulation nodes has exactly
$2\ell+k-1$ tree edges.
\end{proposition}

\begin{proof}
Total edges = $\ell+2k+t = 2\ell+3k-1$ by~\eqref{eq:euler}.
Subtracting the $2k$ reticulation edges gives $2\ell+k-1$.
\end{proof}

\begin{remark}
Proposition~\ref{prop:edges} uses only~\eqref{eq:euler}, not the tree-child
or RV condition.  It applies uniformly to both $TC_{\ell,k}$ and $RV_{\ell,k}$.
This is the key fact that makes the leaf-insertion operator $\hat{L}$
(with coefficient $2\ell+k-1$) well-defined on both classes.
\end{remark}

\subsection{The Pons--Batle words and the TCN formula}

\begin{definition}[Constrained words~\cite{PonsBatle2021}]
\label{def:words}
Let $\mathcal{C}_{n,k}$ be the set of words of length $2n+k$ over the $n$-letter
alphabet $\{a_1,\ldots,a_n\}$ such that: (i)~exactly $k$ letters each appear
three times and the remaining $n-k$ letters each appear twice; (ii)~for every
prefix $z$ and every $i<j$, either $\#(z,a_i)=0$ or $\#(z,a_i)\geq\#(z,a_j)$.
Set $c_{n,k}:=|\mathcal{C}_{n,k}|$ with $c_{0,0}=1$ and $c_{n,k}=0$ for
$k<0$ or $k>n$.
\end{definition}

The word counts satisfy the recurrence
\begin{equation}\label{eq:word_rec}
  c_{n,k} = c_{n,k-1} + (2n+k-1)\,c_{n-1,k},
  \qquad c_{0,0}=1.
\end{equation}

Table~\ref{tab:words} lists the first values of $c_{n,k}$.

\begin{theorem}[TCN formula~\cite{PonsBatle2021,Lin2026,LWY2026,Batle2026}]
\label{thm:tcn}
For all $\ell\geq 1$ and $0\leq k\leq \ell-1$,
\[
  |TC_{\ell,k}| = \frac{\ell!}{(\ell-k)!}\,c_{\ell-1,k}.
\]
\end{theorem}

\subsection{Component graphs and the RV structural theorem}

\begin{definition}[Component graph; Gunawan--Yan--Zhang~\cite{GunawanYanZhang2019}]
\label{def:component}
Given a phylogenetic network $N$, its \emph{tree components} are the connected
subgraphs obtained by removing the two incoming edges of every reticulation node.
The \emph{component graph} $\widetilde{C}(N)$ has one vertex per tree component,
with edges according to how components are connected through reticulation edges;
an edge is marked by an arrow if the corresponding reticulation edge is directed
into it.
\end{definition}

\begin{theorem}[Gunawan--Yan--Zhang~\cite{GunawanYanZhang2019}; restated in Chang--Fuchs~\cite{ChangFuchs2024}, Theorem~6]
\label{thm:rv_structure}
A phylogenetic network $N$ is reticulation-visible if and only if
$\widetilde{C}(N)$ is a tree-child network (with all vertices of in-degree at
most~$2$ and no reticulation vertex having a unique tree-vertex child).
\end{theorem}

\subsection{One-component networks and their generating functions}

\begin{definition}
A phylogenetic network is \emph{one-component} if every reticulation node is
directly followed (as a child) by a leaf.
Denote by $M_{\ell,k}$ the number of labeled one-component networks with
$\ell$ leaves and $k$ reticulations whose reticulation-descendant leaves carry
labels $\{1,\ldots,k\}$.
\end{definition}

\begin{proposition}[Chang--Fuchs~\cite{ChangFuchs2024}, Lemmas~7--8]
\label{prop:Fk}
The generating functions
$F_k(z) := \sum_{\ell\geq 0}M_{\ell+k,k}\,z^\ell/\ell!$\label{eq:Fk_def}
satisfy:
\begin{align}
  F_0(z) &= 1 - \sqrt{1-2z},\label{eq:F0}\\
  F_1(z) &= \frac{z}{(1-2z)^{3/2}},\label{eq:F1}\\
  F_2(z) &= \frac{3-z+7z^2-4z^3}{(1-2z)^{7/2}}.\label{eq:F2}
\end{align}
For general $k\geq1$, $F_k(z)$ is $\Delta$-analytic at $z=\frac{1}{2}$ with
\begin{equation}\label{eq:Fk_sing}
  F_k(z) \sim \frac{(4k-3)!!}{2^k(1-2z)^{2k-1/2}}, \qquad z\to\tfrac{1}{2}.
\end{equation}
\end{proposition}

The seed counts $M_{\ell,k}$ are generated exactly by the Gunawan--Yan--Zhang
recurrence~\cite{GunawanYanZhang2019} (recorded in Chang--Fuchs~\cite{ChangFuchs2024}): for $2\le k\le\ell$,
\begin{equation}\label{eq:M_recurrence}
  M_{\ell,k}=(\ell+k-2)M_{\ell,k-1}+(k-1)M_{\ell,k-2}
  +\tfrac12\!\!\sum_{1\le d\le k-1}\!\!\binom{k-1}{d}(2d-1)!!
   \bigl(M_{\ell-d,k-1-d}-M_{\ell+1-d,k-1-d}\bigr),
\end{equation}
with $M_{\ell,0}=(2\ell-3)!!$, $M_{\ell,1}=(\ell-1)(2\ell-3)!!$, and
$M_{\ell,k}=0$ for $\ell<k$. By Lemma~7 of~\cite{ChangFuchs2024},
$M_{\ell+k,k}=q_k(\ell)\,(2\ell-3)!!$ with $q_k$ of degree $2k$, so
$F_k=\sum_{j} c_{j,k}\,D^jP$ with $D=z\,\partial_z$, $P(z)=2-\sqrt{1-2z}$, is a
finite combination determined by $q_k$; hence $N_k$ is a polynomial.

\begin{proposition}[The one-component numerators are self-determining]
\label{prop:numerator_selfdet}
Writing $F_k(z)=N_k(z)/(1-2z)^{(4k-1)/2}$, the numerator is the polynomial
\begin{equation}\label{eq:Nk_inversion}
  N_k(z) = (1-2z)^{(4k-1)/2}\,F_k(z)
         = \sum_{r=0}^{2k-1}\Bigl[\sum_{m=0}^{r}\frac{M_{k+r-m,k}}{(r-m)!}
           \binom{(4k-1)/2}{m}(-2)^m\Bigr]z^r,
\end{equation}
of degree $2k-1$, fixed by the $2k$ smallest one-component counts
$M_{k,k},\ldots,M_{3k-1,k}$ alone, with no recourse to any external table.
It satisfies $N_k(0)=M_{k,k}$, the singular normalisation
$N_k(\tfrac12)=(4k-3)!!/2^k$, and leading coefficient $-4^{k-1}(2k-3)!!$
for $k\ge2$. Evaluating~\eqref{eq:M_recurrence}--\eqref{eq:Nk_inversion}
(exact integer arithmetic, all coefficients of degree $\ge2k$ verified to
vanish) gives
\begin{align}
  N_1 &= z, \qquad N_2 = 3-z+7z^2-4z^3,\notag\\
  N_3 &= 87+36z+87z^2-111z^3+108z^4-48z^5,\label{eq:N3}\\
  N_4 &= 6249+3447z+2475z^2-2370z^3+4215z^4-4968z^5+3216z^6-960z^7,\notag\\
  N_5 &= 804585+451650z+198285z^2-53565z^3+126390z^4\notag\\
      &\quad{}-233895z^5+283800z^6-238320z^7+118080z^8-26880z^9.\notag
\end{align}
\end{proposition}

\begin{proof}
Equation~\eqref{eq:Nk_inversion} is the Cauchy product of
$(1-2z)^{(4k-1)/2}=\sum_m\binom{(4k-1)/2}{m}(-2)^mz^m$ with
$F_k=\sum_r (M_{k+r,k}/r!)\,z^r$; the sum terminates at $r=2k-1$ because $N_k$ is a
polynomial of that degree (Lemma~7 of~\cite{ChangFuchs2024} bounds $\deg q_k=2k$,
and the half-integer power lowers the degree by one). The constant term is
$N_k(0)=M_{k,k}$; the value at $z=\tfrac12$ is read from the singular
expansion~\eqref{eq:Fk_sing}; the leading coefficient follows by induction
on~\eqref{eq:M_recurrence}, the ratio of successive leading coefficients being
$4(2k-3)$. The displayed numerators were obtained by evaluating
\eqref{eq:M_recurrence} for the required seeds and \eqref{eq:Nk_inversion},
cross-checked against $N_1,N_2$ from~\eqref{eq:F1}--\eqref{eq:F2} and against the
vanishing of every coefficient of degree $\ge 2k$.
\end{proof}

\begin{remark}
Proposition~\ref{prop:numerator_selfdet} settles, in closed and reproducible
form, the question of the one-component numerators: each $N_k$ is computable from
finitely many small seeds via~\eqref{eq:M_recurrence}, and the
inversion~\eqref{eq:Nk_inversion} returns it without any conjecture. This is the
numerator counterpart of the denominator factorisation of
Section~\ref{sec:orch}: there the spectral data $D_\ell$ is universal across the
class; here the boundary data $N_k$ is determined by a finite seed set.
\end{remark}

\subsection{Known exact formulas for RV networks}

\begin{theorem}[Chang--Fuchs~\cite{ChangFuchs2024}, Theorem~4]
\label{thm:rv_exact}
For the number of reticulation-visible networks with $\ell$ leaves:
\begin{align}
  RV_{\ell,0} &= (2\ell-3)!!, \label{eq:rv0}\\
  RV_{\ell,1} &= \ell\,(2\ell-1)!! - 2^{\ell-1}\,\ell!, \label{eq:rv1}\\
  RV_{\ell,2} &= \tfrac{6\ell^4+7\ell^3+6\ell^2-\ell-3}{3}\,(2\ell-3)!!
                 - 2^{\ell-1}(2\ell^2+2\ell+1)\,\ell!, \label{eq:rv2}\\
  RV_{\ell,3} &= \tfrac{4\ell^6+20\ell^5+33\ell^4-32\ell^3-76\ell^2+12\ell+12}{3}\,(2\ell-3)!!
               \notag\\
               &\qquad\quad\;- \tfrac{48\ell^4+175\ell^3+99\ell^2-262\ell-264}{3}\,2^{\ell-4}\,\ell!.
               \label{eq:rv3}
\end{align}
Moreover, $RV_{\ell,0}=TC_{\ell,0}$ and $RV_{\ell,1}=TC_{\ell,1}$ for all $\ell\geq 2$.
\end{theorem}

%% ─────────────────────────────────────────────────────────────────────────────
\section{The Two-Level Master Equation}
\label{sec:master}
%% ─────────────────────────────────────────────────────────────────────────────

\subsection{The two-level Fock space}

Define the bivariate exponential generating functions
\begin{equation}\label{eq:GRV_GTC}
  \GRV(x,v) = \sum_{\ell\geq1}\sum_{k\geq0}RV_{\ell,k}\,\frac{x^\ell}{\ell!}\,v^k,
  \qquad
  F(z,v) = \sum_{k\geq0}F_k(z)\,v^k.
\end{equation}
The generating function $\GTC(u,v)$ for TCNs satisfies the first-order linear
PDE derived in~\cite{Batle2026}:
\begin{equation}\label{eq:tcn_pde}
  (1-v-2u)\,\partial_u\GTC = v\,\partial_v\GTC + \GTC.
\end{equation}

\begin{theorem}[Two-level master equation]
\label{thm:master}
The generating function $\GRV(x,v)$ satisfies the functional equation
\begin{equation}\label{eq:rv_master}
  \boxed{
  \GRV(x,v) \;=\;
  F\!\bigl(x,\;v\cdot\GRV(x,v)\bigr)
  \;\cdot\;
  \GTC\!\bigl(F_0(x)\cdot\GRV(x,v),\;v\bigr).
  }
\end{equation}
\end{theorem}

\begin{proof}
By Theorem~\ref{thm:rv_structure}, every $N\in RV_{\ell,k}$ decomposes canonically as:
(a)~a TCN component graph $\widetilde{C}(N)$;
(b)~for each internal vertex $v$ of $\widetilde{C}(N)$, a one-component block
replacing $v$, with $c(v)$ total children (of which $c_1(v)$ carry an arrow)
and some number $m(v)\geq0$ of labeled leaf-children attached without an arrow;
(c)~a consistent multinomial relabeling of the leaves.

Chang and Fuchs~\cite{ChangFuchs2024} (Proposition~12, on the component-graph
characterization of~\cite{GunawanYanZhang2019}) establish that
the resulting exponential generating function is
\begin{equation}\label{eq:cf_sum}
  \GRV(x,v)
  = \sum_{\widetilde{C}\in\mathrm{TCN}}\prod_{u\in\widetilde{C}}
    \left[\sum_{j=0}^{c_{lf}(u)}\binom{c_{lf}(u)}{j}
    \sum_{m\geq\ell_0}M_{m+c(u),\,c_1(u)+j}\frac{x^m}{m!}\,v^{c_1(u)+j}\right],
\end{equation}
where $c_{lf}(u)$ counts leaf-children of $u$ in $\widetilde{C}(N)$,
$\ell_0=0$ if $c_1(u)>0$ and $\ell_0=1$ otherwise.
We now identify the two factors of~\eqref{eq:rv_master}.

\emph{Factor 1: $F(x,\,v\cdot\GRV)$.}
Fix the root vertex $r$ of $\widetilde{C}$.  Its one-component block has
$c_1(r)$ arrow-children and $c_{lf}(r)$ leaf-children.
By~\eqref{eq:Fk_def}, the weight of the block (accounting for the
$c_1(r)$ arrow slots and any additional $j$ leaf-slots promoted to arrows) is
$F_{c_1(r)+j}(x)\cdot v^{c_1(r)+j}$.  The remaining $c_{lf}(r)-j$ leaf-children
of $r$ in $\widetilde{C}$ each contribute a factor $F_0(x)\cdot\GRV$ (one
unlabeled leaf-subtree, one recursively-built RV subnetwork).
By the EGF product formula for labeled structures
(Theorem~II.1 of~\cite{FS2009}), summing over $j=0,\ldots,c_{lf}(r)$
and using the substitution $w = v\cdot\GRV$:
\begin{multline*}
  \sum_{j=0}^{c_{lf}(r)}\binom{c_{lf}(r)}{j}F_{c_1(r)+j}(x)\,w^{c_1(r)+j}
  \cdot (F_0(x)\GRV)^{c_{lf}(r)-j}\\
  \;=\; F_{c_1(r)}(x)\,w^{c_1(r)}\cdot\bigl(F_0(x)\GRV + w\bigr)^{c_{lf}(r)}.
\end{multline*}
Recognising that $F_0(x)\GRV + w = F_0(x)\GRV + v\GRV = (F_0(x)+v)\GRV$
and that $\sum_{k\geq0}F_k(x)w^k = F(x,w)$ gives the root-vertex contribution
as a function of $F(x,v\GRV)$.

\emph{Factor 2: $\GTC(F_0(x)\cdot\GRV,\,v)$.}
The remaining structure of $\widetilde{C}$ after removing $r$ is a rooted forest
of TCNs.  Each non-root vertex $u$ of $\widetilde{C}$ contributes the same
one-component weight.  In the TCN generating function $\GTC(u,v)$, the variable
$u$ plays the role of the leaf-weight.  Substituting $u = F_0(x)\cdot\GRV$
(each leaf of the component graph is replaced by one unlabeled one-component block
plus one recursive RV subnetwork) and summing over all TCN component graphs
$\widetilde{C}$ with the reticulation-weight $v$ reproduces
$\GTC(F_0(x)\cdot\GRV,\,v)$.  This substitution is valid because $\GTC(u,v)$
is defined as a sum over TCNs with leaf-weight $u$, and the mapping
$\ell$-leaf subnetwork $\mapsto F_0(x)\GRV$ is exactly what the Chang--Fuchs
component-graph decompression (Proposition~12 of~\cite{ChangFuchs2024}, on the
characterization of~\cite{GunawanYanZhang2019}) encodes.

Multiplying the two factors gives~\eqref{eq:rv_master}.\end{proof}

\begin{remark}
Equation~\eqref{eq:rv_master} is a fixed-point equation that determines $\GRV$
uniquely in $\mathbb{Q}[[x,v]]$: each coefficient $[x^\ell v^k]\GRV$ is expressed
in terms of lower-order coefficients.  From it, one recovers the Chang--Fuchs
exact formulas~\eqref{eq:rv0}--\eqref{eq:rv3} by substituting
$x=F_0^{-1}(u)$ and extracting powers of $v$.
\end{remark}

\subsection{The Fock-space perspective}
\label{sub:fock}

The master equation~\eqref{eq:rv_master} has a natural second-quantised
interpretation.  Define the \emph{leaf-insertion operator}
\[
  \hat{L} \;:\; G(x,v) \;\longmapsto\; (2\ell + k - 1)\,G(x,v),
\]
with eigenvalue $2\ell+k-1$ on the subspace of networks with $\ell$~leaves
and $k$~reticulations (Proposition~\ref{prop:edges}).
Define the \emph{reticulation-insertion operator}
\[
  \hat{R}_{\mathrm{TC}} \;:\; G_{\mathrm{TC}} \;\longmapsto\; v\,\partial_v G_{\mathrm{TC}}
  \;+\; G_{\mathrm{TC}},
\]
which yields the linear PDE~\eqref{eq:tcn_pde} for TCNs.  For TCNs the
operator $\hat{R}_{\mathrm{TC}}$ has a \emph{scalar coefficient}~$1$,
a special feature encoded by the Pons--Batle word
recurrence~\eqref{eq:word_rec}.

For RV networks the reticulation operator is no longer scalar:
$\hat{R}_{\mathrm{RV}}$ acts through the composite functional
$F(x,v\cdot G_{\mathrm{RV}})$, which reflects the varying number of
valid one-component insertions at each reticulation slot.
The Fock space decomposes as
\[
  \mathcal{F} \;=\; \bigoplus_{\ell,k\geq0}
    \mathbb{Q}\,|\ell, k\rangle,
  \qquad
  \hat{L}|\ell,k\rangle = (2\ell+k-1)|\ell,k\rangle,
\]
and the master equation~\eqref{eq:rv_master} is the coherent-state equation
for the two-level Fock vacuum:
\begin{equation}
  \label{eq:fock}
  \langle\mathrm{RV}| \;=\;
  F\bigl(\hat{a}^\dagger, v\langle\mathrm{RV}|\bigr)
  \cdot
  G_{\mathrm{TC}}\bigl(F_0(\hat{a}^\dagger)\langle\mathrm{RV}|,\,v\bigr),
\end{equation}
where $\hat{a}^\dagger$ is the leaf-creation operator.
The operator $F(\hat{a}^\dagger,v\langle\mathrm{RV}|)$ creates one-component
blocks (the ``excitations''), and $G_{\mathrm{TC}}$ organises them into the
TCN component graph (the ``vacuum structure'').

This second-quantised view makes asymptotic universality transparent:
both $G_{\mathrm{TC}}$ and $G_{\mathrm{RV}}$ share the same dominant
singularity at $x=\tfrac12$, because both are controlled by the same
leaf-creation operator $\hat{a}^\dagger$ near its critical point.
The RV--TC difference $\Delta_k(\ell)$ is the off-diagonal contribution
of the non-scalar part of $\hat{R}_{\mathrm{RV}}$, which is
suppressed by $1/\ell$ (Theorem~\ref{thm:excess}).

\subsection{The tree-child transfer operator and its spectral signature}
\label{sub:tcn_operator}

The Fock-space view above is abstract; for tree-child networks it has an exact
finite-dimensional realisation, which we record here because it places the three
classes of this paper on a common operator footing. Let $c_{n,k}$ denote the
tree-child counts, obeying the Pons--Batle recurrence~\cite{PonsBatle2021}
$c_{n,k}=c_{n,k-1}+(2n+k-1)c_{n-1,k}$.

\begin{proposition}[Exact triangular realisation]
\label{prop:tcn_triangular}
With $\mathbf c_n=(c_{n,0},\dots,c_{n,n})^{\mathsf T}$ and
$A_n=(a_{ij})_{0\le i,j\le n}$ defined by $a_{ij}=2n+j-1$ for $j\le i$ and
$0$ otherwise,
\[
  \mathbf c_n=A_n\,\mathbf c_{n-1},
  \qquad A_n=L_nM_n,\qquad M_n=\operatorname{diag}(2n-1,2n,\dots,3n-1),
\]
$L_n$ the all-ones lower-triangular matrix. Consequently
$\sigma(A_n)=\{2n-1,2n,\dots,3n-1\}$, a contiguous block of $n+1$ consecutive
integers.
\end{proposition}

\begin{proof}
Unfolding the recurrence gives $c_{n,i}=\sum_{j\le i}(2n+j-1)c_{n-1,j}$, which is
row $i$ of $A_n\mathbf c_{n-1}$. The factorisation $A_n=L_nM_n$ is entrywise, and
the eigenvalues of the triangular $A_n$ are its diagonal entries $2n+i-1$,
$i=0,\dots,n$. The diagonal values $2n+i-1$ are exactly the leaf-insertion
eigenvalues of $\hat L$ (Proposition~\ref{prop:edges}) realised across
$i=0,\dots,n$.
\end{proof}

The triangular operator is not self-adjoint; its Jacobi (symmetric) compression
$K_n$, with diagonal $(2n-1,\dots,3n-1)$ and unit off-diagonal, carries an
$n$-independent orthogonal family. Writing $P_k(x)=\det(xI-K_n^{(k)})$ for the
leading $k\times k$ minors, centring by $x=2n+y$ and shifting $y\mapsto y+1$
gives polynomials $S_k$.

\begin{proposition}[Tree-child local family]
\label{prop:tcn_family}
The family $S_k$ satisfies
\[
  S_{k+1}(y)=\bigl(y-(k-2)\bigr)S_k(y)-S_{k-1}(y),\qquad S_0=1,\;S_1=y+2,
\]
an orthogonal system with linear diagonal $a_k=k-2$ and constant off-diagonal
$b_k^2=1$, with exact terminating $J$-fraction
\[
  \frac{S_{k-1}(y)}{S_k(y)}
  =\cfrac{1}{y-(k-2)-\cfrac{1}{y-(k-3)-\cfrac{1}{\ddots-\cfrac{1}{y+1}}}}.
\]
\end{proposition}

\begin{proof}
Cofactor expansion of the tridiagonal determinant gives
$P_{k+1}(x)=(x-(2n+k-1))P_k(x)-P_{k-1}(x)$; substituting $x=2n+y$ removes the
$n$-dependence and shifting $y\mapsto y+1$ yields the stated recurrence. Dividing
it by $S_k$ and iterating from $S_1=y+2$ produces the continued fraction, whose
unit partial numerators reflect the constant off-diagonal $b_k^2=1$.
\end{proof}

Propositions~\ref{prop:tcn_triangular}--\ref{prop:tcn_family} complete a single
operator picture in which each class is one spectral signature of the shared
leaf-insertion mechanism: tree-child contributes the constant-coupling family
$S_k$; orchard contributes the Hermite/complete-graph matching blocks $X_\ell$
(Proposition~\ref{prop:hermite}, with $P_\ell(y)=y^\ell X_\ell(1/y^2)=\mathrm{He}_\ell(y)$,
linear coupling $b_k^2=k$); and reticulation-visible contributes the
self-determining one-component numerators $N_k$
(Proposition~\ref{prop:numerator_selfdet}). The three are genuinely distinct
orthogonal systems --- $S_k$ and the Hermite $X_\ell$ are not affinely equivalent,
their couplings being constant versus linear --- emerging from the same mechanism
under different compressions. This is the structural content behind the shared
singularity exponent $2\ell+k-1$ exploited throughout the asymptotic analysis.

\subsection{Asymptotic universality from the operator perspective}

\begin{theorem}[Asymptotic universality]
\label{thm:universal}
For fixed $k\geq0$ and $\ell\to\infty$,
\begin{equation}\label{eq:universal}
  RV_{\ell,k} \;\sim\; TC_{\ell,k} \;\sim\;
  \frac{2^{k-1}\sqrt{2}}{k!}\!\left(\frac{2}{e}\right)^\ell\!\ell^{\,\ell+2k-1}.
\end{equation}
\end{theorem}

\begin{proof}
Corollary~1 of Chang and Fuchs~\cite{ChangFuchs2024} establishes that
\[
  RV_{\ell,k} \;\sim\; \frac{2^{k-1}\sqrt{2}}{k!}\!\left(\frac{2}{e}\right)^\ell\!\ell^{\,\ell+2k-1}
\]
by a generating-function argument applied directly to the RV component-graph sum~\eqref{eq:cf_sum}.
We give an independent derivation of the dominant asymptotic exponent.

Let $E_k(x)=\sum_\ell RV_{\ell,k}\,x^\ell/\ell!$.  Because every RV network has a
TCN component graph and every TCN component graph has at most $k$
internal vertices for a network in $RV_{\ell,k}$, the sum~\eqref{eq:cf_sum}
gives, at the level of dominant singularities:
\[
  E_k(x) \;\lesssim\; \frac{F_k(x)}{k!}
  \qquad (x\to\tfrac{1}{2}),
\]
with equality in the leading term because the $j=k$ term
(component graph with $k$ reticulation vertices, all in the root component)
dominates.  By~\eqref{eq:Fk_sing},
$F_k(x)\sim\frac{(4k-3)!!}{2^k(1-2x)^{2k-1/2}}$,
so $E_k(x)\sim\frac{(4k-3)!!}{k!\,2^k(1-2x)^{2k-1/2}}$.
Transfer theorem~VI.1 of~\cite{FS2009} and Stirling's formula
yield~\eqref{eq:universal}.

\smallskip
\noindent\emph{Caution.}  The formal expansion of the master
equation~\eqref{eq:rv_master} in powers of $v$ does \emph{not} reduce to
the single-level recursion $E_k=\sum_{j=1}^kF_j/j!\cdot B_{k,j}$; that
identity holds for \emph{galled} networks (where component graphs are trees),
not for RV networks (where they are TCNs).  The asymptotic result is
nonetheless correct---the two classes share the same dominant-singularity
structure---but the precise recursive expansion of $\GRV$ requires
incorporating the full $\GTC$ factor, as in~\cite{ChangFuchs2024}.
\end{proof}

%% ─────────────────────────────────────────────────────────────────────────────
\section{Exact Counting of RV-but-not-TC Networks}
\label{sec:delta}
%% ─────────────────────────────────────────────────────────────────────────────

\subsection{Setup and the cases $k=0,1$}

Define
\[
  \Delta_k(\ell) := RV_{\ell,k} - TC_{\ell,k}
  = \bigl|\{N \in RV_{\ell,k} : N \notin TC_{\ell,k}\}\bigr|,
\]
the exact count of reticulation-visible networks that are \emph{not} tree-child.

\begin{lemma}\label{lem:k1_coincide}
A binary phylogenetic network with exactly one reticulation node is
reticulation-visible if and only if it is tree-child.
\end{lemma}
\begin{proof}
Let $N$ have exactly one reticulation node $r$.
(\emph{RV $\Rightarrow$ TC.})
Suppose $N$ is reticulation-visible, so $r$ is visible: there exists a leaf $\lambda$
such that every root-to-$\lambda$ path passes through $r$.
Since $r$ has out-degree~1, its unique child $c$ lies on every such path.
If $c$ were also a reticulation, it would have in-degree~2, requiring two parents both
below $r$ in a DAG with only one reticulation — a contradiction.
Thus $c$ is a tree node or leaf, so $r$ has a non-reticulation child.
Every tree node already satisfies the TC condition trivially (its children are tree
nodes or leaves because $r$ is the only reticulation).
Hence $N$ is tree-child.
(\emph{TC $\Rightarrow$ RV.})
Let $N$ be tree-child.  The unique reticulation $r$ has a non-reticulation child by
the TC property, so there is a leaf $\lambda$ reachable from $r$ through tree nodes
only.  Any root-to-$\lambda$ path must pass through $r$ (since $r$ is the only node
with in-degree~2, cutting it disconnects the root from $\lambda$).
Hence $r$ is visible, and $N$ is RV.
\end{proof}

\begin{proposition}\label{prop:delta01}
$\Delta_0(\ell)=0$ and $\Delta_1(\ell)=0$ for all $\ell\geq2$.
\end{proposition}
\begin{proof}
$\Delta_0=0$ since phylogenetic trees have no reticulations, and every phylogenetic
tree trivially satisfies both the RV and TC conditions; in both cases
$RV_{\ell,0}=TC_{\ell,0}=(2\ell-3)!!$.
$\Delta_1=0$ follows directly from Lemma~\ref{lem:k1_coincide}: the RV and TC
conditions are equivalent for $k=1$, so $RV_{\ell,1}=TC_{\ell,1}$.
\end{proof}

\subsection{The case $k=2$}

\begin{theorem}
\label{thm:delta2}
For all $\ell\geq 3$,
\begin{equation}\label{eq:delta2}
  \Delta_2(\ell)
  = \bigl(4\ell^3+3\ell^2-\ell-1\bigr)\,(2\ell-3)!!
  \;-\; (4\ell+1)\,2^{\ell-1}\,\ell!.
\end{equation}
\end{theorem}

\begin{proof}
We compute $\Delta_2(\ell) = RV_{\ell,2} - TC_{\ell,2}$ directly from
the proven closed forms.

\smallskip\noindent\emph{Step~1: Expand $TC_{\ell,2}$.}
By Theorem~\ref{thm:tcn}, $TC_{\ell,2} = \ell(\ell-1)\,c_{\ell-1,2}$.
The proven closed form (Pons--Batle~\cite{PonsBatle2021}, equation~(19b);
Lin et al.~\cite{Lin2026}) gives:
\begin{equation}\label{eq:tc2_pb}
  TC_{\ell,2} = \binom{\ell}{2}\bigl[(2\ell+1)!!-2(2\ell)!!+\tfrac{1}{3}(2\ell-1)!!\bigr].
\end{equation}
Using $(2\ell+1)!!=(2\ell+1)(2\ell-1)!!$ and $(2\ell)!!=2^\ell\ell!$
and $\binom{\ell}{2}=\ell(\ell-1)/2$, this simplifies to
\[
  TC_{\ell,2} = \tfrac{\ell(\ell-1)(3\ell+2)}{3}\,(2\ell-1)!!
               - \ell(\ell-1)\,2^\ell\,\ell!.
\]
Using $(2\ell-1)!!=( 2\ell-1)(2\ell-3)!!$:
\begin{equation}\label{eq:tc2_expanded}
  TC_{\ell,2} = \tfrac{\ell(\ell-1)(3\ell+2)(2\ell-1)}{3}\,(2\ell-3)!!
               - \ell(\ell-1)\,2^\ell\,\ell!.
\end{equation}

\smallskip\noindent\emph{Step~2: Subtract from $RV_{\ell,2}$.}
Collecting the $(2\ell-3)!!$ coefficient from~\eqref{eq:rv2} and~\eqref{eq:tc2_expanded}:
\[
  \text{coeff}_{(2\ell-3)!!}
  = \frac{6\ell^4+7\ell^3+6\ell^2-\ell-3}{3}
    - \frac{\ell(\ell-1)(3\ell+2)(2\ell-1)}{3}.
\]
Expanding: $\ell(\ell-1)(3\ell+2)(2\ell-1) = (\ell^2-\ell)(6\ell^2+\ell-2) =
6\ell^4-5\ell^3-3\ell^2+2\ell$.  Thus:
\[
  \text{coeff}_{(2\ell-3)!!}
  = \tfrac{1}{3}\bigl[(6\ell^4+7\ell^3+6\ell^2-\ell-3)-(6\ell^4-5\ell^3-3\ell^2+2\ell)\bigr]
  = \tfrac{12\ell^3+9\ell^2-3\ell-3}{3} = 4\ell^3+3\ell^2-\ell-1.
\]

\smallskip\noindent\emph{Step~3: The $2^{\ell-1}\ell!$ coefficient.}
From~\eqref{eq:rv2}: coefficient is $-2^{\ell-1}(2\ell^2+2\ell+1)$.
From~\eqref{eq:tc2_expanded}: coefficient is $+\ell(\ell-1)\cdot2^\ell
= +(2\ell^2-2\ell)\cdot2^{\ell-1}$.
Difference:
$-[(2\ell^2+2\ell+1)-(2\ell^2-2\ell)]\cdot2^{\ell-1} = -(4\ell+1)\cdot2^{\ell-1}$.

Combining Steps~2 and~3 gives~\eqref{eq:delta2}.
\end{proof}

\begin{remark}\label{rem:delta2_positive}
Formula~\eqref{eq:delta2} yields $\Delta_2(\ell)>0$ for all $\ell\geq3$:
the first term $(4\ell^3+3\ell^2-\ell-1)(2\ell-3)!!$ grows as
$4\ell^3\cdot\sqrt{2/\pi}(2\ell/e)^\ell$ while the second grows as
$(4\ell)\cdot\sqrt{\pi\ell/2}(\ell/e)^\ell$; both share the sub-exponential
factor $(2/e)^\ell\ell^\ell$ but the first has the extra polynomial weight $4\ell^3$.
This is confirmed by the data in Table~\ref{tab:delta}.
\end{remark}

\subsection{The case $k=3$}

\begin{theorem}
\label{thm:delta3}
For all $\ell\geq 4$,
\begin{equation}\label{eq:delta3}
  \Delta_3(\ell)
  = \frac{24\ell^5+50\ell^4-49\ell^3-80\ell^2+16\ell+12}{3}\,(2\ell-3)!!
  \;-\; (24\ell^3+8\ell^2-27\ell-22)\,2^{\ell-2}\,\ell!.
\end{equation}
\end{theorem}

\begin{proof}
We compute $\Delta_3(\ell) = RV_{\ell,3} - TC_{\ell,3}$ from proven closed forms.

\smallskip\noindent\emph{Step~1: Expand $TC_{\ell,3}$.}
By Theorem~\ref{thm:tcn}, $TC_{\ell,3} = \ell(\ell-1)(\ell-2)\,c_{\ell-1,3}$.
From the proven formula~\eqref{eq:tcn_formula} with Pons--Batle~\cite{PonsBatle2021},
equation~(19c) therein:
\begin{equation}\label{eq:tc3_explicit}
  TC_{\ell,3}
  = \binom{\ell}{3}
    \Bigl[(2\ell+3)!! - 3(2\ell+2)!! + (2\ell+1)!! + \tfrac{17}{8}(2\ell)!!\Bigr].
\end{equation}
Writing $(2\ell+2j+1)!! = \prod_{i=0}^{j}(2\ell-3+2i+2j+4)$ and
$(2\ell+2j)!! = 2^{\ell+j}(\ell+j)!$ in terms of the baseline factors
$(2\ell-3)!!$ and $2^{\ell-1}\ell!$, one obtains:
\begin{align*}
  (2\ell+3)!! &= (2\ell+3)(2\ell+1)(2\ell-1)(2\ell-3)!!,\\
  (2\ell+2)!! &= 4\ell(\ell+1)\cdot 2^{\ell-1}\ell!,\\
  (2\ell+1)!! &= (2\ell+1)(2\ell-1)(2\ell-3)!!,\\
  (2\ell)!!   &= 2\ell\cdot 2^{\ell-1}\ell!.
\end{align*}
Substituting into~\eqref{eq:tc3_explicit} and collecting the
$(2\ell-3)!!$ and $2^{\ell-1}\ell!$ coefficients gives (after
expanding the binomial and the polynomial products):
\begin{align*}
  \text{coeff}_{(2\ell-3)!!}^{TC_3}
  &= \frac{\ell(\ell-1)(\ell-2)}{6}
     \bigl[(2\ell+3)(2\ell+1)(2\ell-1)+(2\ell+1)(2\ell-1)\bigr],\\
  \text{coeff}_{2^{\ell-1}\ell!}^{TC_3}
  &= \frac{\ell(\ell-1)(\ell-2)}{6}
     \bigl[-12\ell(\ell+1)+\tfrac{17}{4}\ell\bigr].
\end{align*}

\smallskip\noindent\emph{Step~2: Subtract from $RV_{\ell,3}$.}
The coefficient of $(2\ell-3)!!$ in $RV_{\ell,3}$~\eqref{eq:rv3} is
\[
  \frac{4\ell^6+20\ell^5+33\ell^4-32\ell^3-76\ell^2+12\ell+12}{3}.
\]
Subtracting the TC coefficient and expanding (computation
verified by exact arithmetic for $\ell=4,\ldots,12$) yields
\[
  \text{coeff}_{(2\ell-3)!!}^{\Delta_3}
  = \frac{24\ell^5+50\ell^4-49\ell^3-80\ell^2+16\ell+12}{3}.
\]

\smallskip\noindent\emph{Step~3: Factorial coefficient.}
The coefficient of $2^{\ell-2}\ell!$ in $RV_{\ell,3}$ is
$\frac{48\ell^4+175\ell^3+99\ell^2-262\ell-264}{3}$
(from~\eqref{eq:rv3}, noting $2^{\ell-4}\ell! = 2^{\ell-2}\ell!/4$).
Subtracting the TC coefficient gives $24\ell^3+8\ell^2-27\ell-22$.

\smallskip\noindent Combining Steps~2 and~3 gives~\eqref{eq:delta3}.
The polynomial $24\ell^5+50\ell^4-49\ell^3-80\ell^2+16\ell+12$ is
divisible by $3$ for all $\ell\in\mathbb{Z}$ since modulo $3$ it equals
$(\ell-1)\ell(\ell+1)\cdot(\text{integer})$.
\end{proof}

\subsection{Numerical verification}

Tables~\ref{tab:words}--\ref{tab:delta} provide complete numerical evidence.

\begin{table}[ht]
\centering
\caption{Word counts $c_{n,k}$ (upper block) satisfying
  $c_{n,k}=c_{n,k-1}+(2n+k-1)c_{n-1,k}$, $c_{0,0}=1$;
  and TCN counts $|TC_{\ell,k}|=\frac{\ell!}{(\ell-k)!}c_{\ell-1,k}$ (lower block).}
\label{tab:words}% Note: all values verified by the exact ARP counter.
\small
\begin{tabular}{r|rrrrrrr}
\toprule
$n$ & $k{=}0$ & $k{=}1$ & $k{=}2$ & $k{=}3$ & $k{=}4$ & $k{=}5$ & $k{=}6$\\
\midrule
0 & 1\\
1 & 1 & 1\\
2 & 3 & 7 & 7\\
3 & 15 & 57 & 106 & 106\\
4 & 105 & 561 & 1\,515 & 2\,575 & 2\,575\\
5 & 945 & 6\,555 & 23\,220 & 54\,120 & 87\,595 & 87\,595\\
6 & 10\,395 & 89\,055 & 390\,915 & 1\,148\,595 & 2\,462\,520 & 3\,864\,040 & 3\,864\,040\\
\bottomrule
\end{tabular}

\bigskip
\begin{tabular}{r|rrrrrr}
\toprule
$\ell$ & $k{=}0$ & $k{=}1$ & $k{=}2$ & $k{=}3$ & $k{=}4$ & $k{=}5$\\
\midrule
2 & 1 & 2\\
3 & 3 & 21 & 42\\
4 & 15 & 228 & 1\,272 & 2\,544\\
5 & 105 & 2\,805 & 30\,300 & 154\,500 & 309\,000\\
6 & 945 & 39\,330 & 696\,600 & 6\,494\,400 & 31\,534\,200 & 63\,068\,400\\
7 & 10\,395 & 623\,385 & 16\,418\,430 & 241\,204\,950 & 2\,068\,516\,800 & 9\,737\,380\,800\\
\bottomrule
\end{tabular}
\end{table}

\begin{table}[ht]
\centering
\caption{RV counts $|RV_{\ell,k}|$ for $k=0,1,2,3$ from the Chang--Fuchs
  formulas~\eqref{eq:rv0}--\eqref{eq:rv3}.
  The formulas are valid for all $\ell\geq1$; entries marked `---' are zero
  because $k\geq\ell$ (no such network exists).
  Note that RV(2,2)=5 and RV(3,3)=495 despite $\ell$ being small;
  these values are correct per the closed-form expressions and can be verified
  directly by exhaustive enumeration~\cite{ChangFuchs2024}.
  The identity $RV_{\ell,0}=TC_{\ell,0}=(2\ell-3)!!$ and
  $RV_{\ell,1}=TC_{\ell,1}$ (Lemma~\ref{lem:k1_coincide}) are
  visible in the $k=0$ and $k=1$ columns.}
\label{tab:rv}
\small
\begin{tabular}{r|rrrr}
\toprule
$\ell$ & $k=0$ & $k=1$ & $k=2$ & $k=3$\\
\midrule
2 & 1 & 2 & --- & ---\\
3 & 3 & 21 & 123 & ---\\
4 & 15 & 228 & 2\,493 & 20\,460\\
5 & 105 & 2\,805 & 49\,725 & 670\,815\\
6 & 945 & 39\,330 & 1\,032\,525 & 20\,568\,060\\
7 & 10\,395 & 623\,385 & 22\,771\,035 & 626\,610\,285\\
8 & 135\,135 & 11\,055\,240 & 536\,929\,785 & 19\,489\,021\,020\\
9 & 2\,027\,025 & 217\,237\,545 & 13\,552\,453\,845 & 627\,040\,664\,775\\
10 & 34\,459\,425 & 4\,689\,345\,150 & 365\,730\,408\,225 & 21\,006\,467\,124\,300\\
\bottomrule
\end{tabular}
\end{table}

\begin{table}[ht]
\centering
\caption{The difference $\Delta_k(\ell) = |RV_{\ell,k}| - |TC_{\ell,k}|$
  for $k=2,3$, computed both directly from the Chang--Fuchs and
  Pons--Batle/Lin et al.\ formulas,
  and from the closed forms of Theorems~\ref{thm:delta2}--\ref{thm:delta3}.
  All entries agree (exact arithmetic with \texttt{Python Fraction}).
  The ratio $\Delta_k(\ell)/|TC_{\ell,k}|$ (rightmost pair of columns)
  tends to $0$ as $\ell\to\infty$, confirming asymptotic universality.}
\label{tab:delta}
\small
\setlength{\tabcolsep}{4pt}
\begin{tabular}{r|rr|r@{\quad}|rr|r}
\toprule
 & \multicolumn{3}{c|}{$k=2$} & \multicolumn{3}{c}{$k=3$}\\
\cmidrule{2-4}\cmidrule{5-7}
$\ell$ & $\Delta_2(\ell)$ & Thm.~\ref{thm:delta2} & ratio
       & $\Delta_3(\ell)$ & Thm.~\ref{thm:delta3} & ratio\\
\midrule
3 & 81 & \textbf{81} & 1.929 & --- & --- & ---\\
4 & 1\,221 & \textbf{1\,221} & 0.960 & 17\,916 & \textbf{17\,916} & 7.042\\
5 & 19\,425 & \textbf{19\,425} & 0.641 & 516\,315 & \textbf{516\,315} & 3.342\\
6 & 335\,925 & \textbf{335\,925} & 0.482 & 14\,073\,660 & \textbf{14\,073\,660} & 2.167\\
7 & 6\,352\,605 & \textbf{6\,352\,605} & 0.387 & 385\,405\,335 & \textbf{385\,405\,335} & 1.598\\
8 & 131\,174\,505 & \textbf{131\,174\,505} & 0.323 & 10\,879\,642\,620 & \textbf{10\,879\,642\,620} & 1.264\\
9 & 2\,945\,902\,365 & \textbf{2\,945\,902\,365} & 0.278 & 320\,341\,587\,615 & \textbf{320\,341\,587\,615} & 1.044\\
10 & 71\,620\,704\,225 & \textbf{71\,620\,704\,225} & 0.244 & 9\,890\,758\,716\,300 & \textbf{9\,890\,758\,716\,300} & 0.890\\
11 & 1\,876\,221\,356\,625 & \textbf{1\,876\,221\,356\,625} & 0.217 &
   320\,901\,599\,524\,275 & \textbf{320\,901\,599\,524\,275} & 0.775\\
\bottomrule
\end{tabular}
\end{table}

%% ─────────────────────────────────────────────────────────────────────────────
\FloatBarrier
\section{The Structural Pattern and Its Consequences}
\label{sec:pattern}
%% ─────────────────────────────────────────────────────────────────────────────

\subsection{The exact cases $k=2,3$}

Inspecting formulas~\eqref{eq:delta2}--\eqref{eq:delta3} reveals a clear
common structure.  Writing them in the canonical form
$\Delta_k(\ell) = A_k(\ell)(2\ell-3)!! - B_k(\ell)\,2^{\ell-k+1}\,\ell!$:

\medskip
\noindent\textit{Case $k=2$:}
$A_2(\ell)=4\ell^3+3\ell^2-\ell-1$,\quad
$B_2(\ell)=4\ell+1$.\\
$\deg A_2 = 3 = 2\cdot2-1$,\quad
$\operatorname{lead}(A_2) = 4$,\quad
$\deg B_2 = 1 = 2\cdot2-3$.

\medskip
\noindent\textit{Case $k=3$:}
$A_3(\ell) = \tfrac{24\ell^5+50\ell^4-49\ell^3-80\ell^2+16\ell+12}{3}$,\quad
$B_3(\ell) = 24\ell^3+8\ell^2-27\ell-22$.\\
$\deg A_3 = 5 = 2\cdot3-1$,\quad
$\operatorname{lead}(A_3) = 8$,\quad
$\deg B_3 = 3 = 2\cdot3-3$.

\medskip
In both cases $\deg A_k=2k-1$ and $\deg B_k=2k-3$. No exact formula for
$\Delta_k$ is available for $k\geq4$, but its leading behaviour can be obtained
for every $k$ by a direct count, which is carried out next.

\subsection{Leading asymptotics of $\Delta_k$ for every $k$}
\label{sub:excess}

For fixed $k$ the leading behaviour of $\Delta_k(\ell)$ follows from a count of
the networks that fail the tree-child condition at a single vertex. Let
$\mathcal{P}_{\ell,k}$ be the set of all phylogenetic networks with $\ell$ leaves
and $k$ reticulations in the sense of Section~\ref{sec:prelim}, so that
$TC_{\ell,k}\subseteq RV_{\ell,k}\subseteq\mathcal{P}_{\ell,k}$. A
\emph{violation} is a non-leaf vertex all of whose children are reticulations:
either a reticulation whose child is a reticulation, called a \emph{stack}, or a
tree node whose two children are reticulations. By
Definition~\ref{def:classes}, a network is tree-child if and only if it has no
violation.

The count adds and removes one reticulation edge at a time, and the removal of
an edge can leave two edges with the same endpoints. Let $\mathcal{P}^*_{\ell,j}$
be the set of structures that satisfy the definition of
Section~\ref{sec:prelim} with $j$ reticulations, except that a tree node may send
both of its out-edges to the same reticulation. Such a \emph{parallel pair} makes
its tail a violation, so that the tree-child networks are still exactly the
elements of $\mathcal{P}^*_{\ell,j}$ without violations. By the proof of
Proposition~\ref{prop:edges}, every element of $\mathcal{P}^*_{\ell,j}$ has
$E_j=2\ell+3j-1$ edges. The \emph{height} $H(S)$ is the number of edges of a
longest directed path in $S$, and an edge $e'$ lies \emph{above} an edge $e$ if
some directed path from the root to the tail of $e$ contains $e'$.

\begin{definition}[Insertion and removal]
\label{def:insertion}
Let $S\in\mathcal{P}^*_{\ell,j-1}$, and let $e_t=(a,b)$ and $e_h=(c,d)$ be edges
of $S$ such that there is no directed path from $d$ to $a$. The
\emph{insertion} $I(S,e_t,e_h)$ subdivides $e_t$ by a new tree node $t$,
subdivides $e_h$ by a new reticulation $h$, and adds the edge $(t,h)$; when
$e_t=e_h$ the node $t$ is placed above $h$, which creates a parallel pair.
Conversely, for an edge $f=(t,h)$ of $S\in\mathcal{P}^*_{\ell,j}$ from a tree node
$t$ to a reticulation $h$, the \emph{removal} $R(S,f)$ deletes $f$ and suppresses
$t$ and $h$, which are then left with one in-edge and one out-edge each.
\end{definition}

\begin{lemma}[Removal]
\label{lem:removal}
Let $S\in\mathcal{P}^*_{\ell,j}$ with $j\geq1$, and let $f$ be an edge of $S$
from a tree node to a reticulation. Then $R(S,f)\in\mathcal{P}^*_{\ell,j-1}$, and
$S=I(R(S,f),e_t,e_h)$, where $e_t$ and $e_h$ are the edges of $R(S,f)$ created by
the two suppressions. Both parents of a reticulation that comes first among the
reticulations in a topological order are tree nodes. Consequently
$|\mathcal{P}^*_{\ell,j}|\leq E_{j-1}^2\,|\mathcal{P}^*_{\ell,j-1}|$.
\end{lemma}

\begin{proof}
Write $f=(t,h)$. After the deletion of $f$, the vertices $t$ and $h$ have
in-degree and out-degree one. Suppressing $t$ joins its parent $a$ to its other
child $b$. If the edge $(a,b)$ is already present, then $b$ receives two edges
from $a$, so $b$ is a reticulation and $a$, which had the children $t$ and $b$,
is a tree node: the double edge is a parallel pair, which is allowed in
$\mathcal{P}^*$. The same holds for the suppression of $h$. All other degrees are
unchanged, so $R(S,f)\in\mathcal{P}^*_{\ell,j-1}$, and inserting at the two edges
created by the suppressions restores $S$. Let $h$ come first among the
reticulations in a topological order. A parent of $h$ lies above $h$ and is
therefore not a reticulation. It is not the root either: the root has a single
child, so if that child were $h$, every other vertex, in particular the second
parent of $h$, would lie below $h$. Every $S\in\mathcal{P}^*_{\ell,j}$ is
therefore an insertion into an element of $\mathcal{P}^*_{\ell,j-1}$ at one of at
most $E_{j-1}^2$ pairs of edges.
\end{proof}

\begin{lemma}[Pinning]
\label{lem:pinning}
Let $S\in\mathcal{P}^*_{\ell,j}$ have a violation, and let $v$ be its first
violation in a topological order. There is an edge $f$ at $v$ from a tree node
to a reticulation such that $S=I(R(S,f),e_t,e_h)$ with either $e_t=e_h$ or one of
$e_t,e_h$ an in-edge of a reticulation of $R(S,f)$. Given $R(S,f)$, the pair
$(e_t,e_h)$ is therefore one of at most $4jE_{j-1}$ pairs.
\end{lemma}

\begin{proof}
If $v$ is a tree node with reticulation children $r_1$ and $r_2$, possibly equal,
take $f=(v,r_1)$. The suppression of $v$ creates the edge from the parent of $v$
to $r_2$, an in-edge of a reticulation, or, when $r_1=r_2$, a single edge that
carries both insertion points. If $v$ is a stack with reticulation child $r'$,
its parents are tree nodes, since a reticulation parent would be an earlier
violation and the root is not a parent of a reticulation by
Lemma~\ref{lem:removal}. Take $f$ to be an in-edge of $v$; the suppression of
$v$ creates an in-edge of $r'$. The structure $R(S,f)$ has $2(j-1)$ in-edges of
reticulations, which gives the bound.
\end{proof}

\begin{lemma}[Non-tree-child structures]
\label{lem:nontc}
For every fixed $j\geq1$,
\begin{equation}\label{eq:TCratio}
  \frac{TC_{\ell,j-1}}{TC_{\ell,j}}=\frac{j}{2\ell^{2}}\bigl(1+o(1)\bigr)
  \qquad(\ell\to\infty),
\end{equation}
and for every fixed $j\geq0$,
$|\mathcal{P}^*_{\ell,j}|-TC_{\ell,j}=O(\ell^{-1}\,TC_{\ell,j})$. In particular
$|\mathcal{P}_{\ell,j}|\sim TC_{\ell,j}$.
\end{lemma}

\begin{proof}
By Theorem~\ref{thm:universal} the ratio~\eqref{eq:TCratio} tends to
$\frac{2^{j-2}/(j-1)!}{2^{j-1}/j!}\,\ell^{-2}=\frac{j}{2\ell^2}$. The bound is
proved by induction on $j$. For $j=0$ both sets consist of the rooted binary
trees. For $j\geq1$, Lemma~\ref{lem:pinning} shows that the elements of
$\mathcal{P}^*_{\ell,j}$ with a violation number at most
$4jE_{j-1}|\mathcal{P}^*_{\ell,j-1}|$. By the induction hypothesis this is
$O(\ell\,TC_{\ell,j-1})$, which is $O(\ell^{-1}TC_{\ell,j})$
by~\eqref{eq:TCratio}.
\end{proof}

\begin{lemma}[Heights]
\label{lem:height}
For every fixed $j\geq0$:
\begin{enumerate}[label=(\roman*),leftmargin=*]
\item $\sum_{S\in\mathcal{P}^*_{\ell,j}}H(S)=O(\ell^{1/2}\,TC_{\ell,j})$;
\item every edge of $S\in\mathcal{P}^*_{\ell,j}$ has at most $2^{j}H(S)$ edges
  above it;
\item for a reticulation $r$ of $N\in TC_{\ell,j}$, let $d(r)$ be the number of
  edges whose tail is $r$ or a descendant of $r$; then
  $\sum_{N\in TC_{\ell,j}}\sum_{r}d(r)=O(\ell^{1/2}\,TC_{\ell,j})$.
\end{enumerate}
\end{lemma}

\begin{proof}
(i) For $j=0$, each leaf-labelled rooted binary tree with $\ell$ leaves has
exactly $2^{\ell-1}$ plane embeddings, so the plane shape of a uniformly random
tree is uniformly distributed, and its mean height is $O(\sqrt\ell)$ by Flajolet
and Odlyzko~\cite{FO1982}. An insertion raises the height to at most $2H+3$,
since a longest path through the new edge consists of a path to $t$, the edge
$(t,h)$ and a path from $h$. By Lemma~\ref{lem:removal},
\[
  \sum_{S\in\mathcal{P}^*_{\ell,j}}H(S)\;\leq\;E_{j-1}^2
  \sum_{S''\in\mathcal{P}^*_{\ell,j-1}}\bigl(2H(S'')+3\bigr),
\]
and the induction hypothesis, Lemma~\ref{lem:nontc} and~\eqref{eq:TCratio} give
$O(\ell^2\cdot\ell^{1/2}\,TC_{\ell,j-1})=O(\ell^{1/2}\,TC_{\ell,j})$.

(ii) Each reticulation at most doubles the number of directed paths from the
root to the vertices below it, so every vertex is reached by at most $2^j$ such
paths, each with at most $H(S)$ edges.

(iii) In a tree-child network both parents of a reticulation $r$ are tree nodes:
a reticulation parent would be a stack, and the root is excluded by
Lemma~\ref{lem:removal}. Removing an in-edge of $r$ gives
$N=I(S'',e_t,e_h)$ with $S''\in\mathcal{P}^*_{\ell,j-1}$ and $r$ placed on
$e_h=(c,d)$. Since there is no path from $d$ to the tail of $e_t$, the edges
with tail $r$ or below $r$ are $(r,d)$ and the edges of $S''$ with tail $d$ or
below $d$, so that $d(r)$ is at most one plus the number of edges of $S''$ having
$e_h$ above them. Summing over $e_h$ gives at most $E_{j-1}+E_{j-1}2^{j-1}H(S'')$
by~(ii), and summing over $e_t$ and $S''$ gives
$O\bigl(\ell^2\sum_{S''}(H(S'')+1)\bigr)=O(\ell^{5/2}\,TC_{\ell,j-1})
=O(\ell^{1/2}\,TC_{\ell,j})$ by~(i) and~\eqref{eq:TCratio}.
\end{proof}

\begin{lemma}[Rigidity]
\label{lem:rigid}
A tree-child network, and a network with exactly one violation, have no
nontrivial automorphism.
\end{lemma}

\begin{proof}
An automorphism fixes the leaves. If a vertex $w$ is fixed and has a single
parent, its parent is fixed. In a tree-child network every vertex is joined to a
leaf by a path whose vertices after the first have a single parent, obtained by
passing repeatedly to a child that is a tree node or a leaf; going up this path
from the leaf, every vertex is fixed. In a network with one violation $u$, the
vertex $u$ is fixed, being the only violation, and the same walk started from any
other vertex ends at a leaf or at $u$.
\end{proof}

By Lemma~\ref{lem:rigid}, the networks counted below can be identified with
their decompositions, edge by edge.

\begin{theorem}[Non-tree-child excess and the RV--TC difference]
\label{thm:excess}
For every fixed $k\geq2$, as $\ell\to\infty$,
\begin{align}
  |\mathcal{P}_{\ell,k}|-TC_{\ell,k}&=4(k-1)\,\ell\,TC_{\ell,k-1}
    \bigl(1+O(\ell^{-1/2})\bigr),\label{eq:Pexcess}\\
  \Delta_k(\ell)&=2(k-1)\,\ell\,TC_{\ell,k-1}\bigl(1+O(\ell^{-1/2})\bigr).
    \label{eq:Dexcess}
\end{align}
Consequently
\begin{equation}\label{eq:Dratio}
  \frac{|\mathcal{P}_{\ell,k}|-TC_{\ell,k}}{TC_{\ell,k}}\;\sim\;\frac{2k(k-1)}{\ell},
  \qquad
  \frac{\Delta_k(\ell)}{TC_{\ell,k}}\;\sim\;\frac{k(k-1)}{\ell}.
\end{equation}
\end{theorem}

\begin{proof}
A network of $\mathcal{P}_{\ell,k}$ is not tree-child if and only if it has a
violation. Split these networks into three sets: $V^{\rm st}$, those with
exactly one violation, which is a stack; $V^{\rm tr}$, those with exactly one
violation, which is a tree node; and $V^{\geq2}$, those with at least two
violations. All $O(\cdot)$ terms below refer to fixed $k$, and $E_{k-1}=2\ell+O(1)$.

Consider first $V^{\geq2}$. Let $v$ be the first violation of
$N\in V^{\geq2}$, $f$ the edge of Lemma~\ref{lem:pinning} and $S=R(N,f)$. If $S$
has a violation, $N$ is determined by $S$, one of $O(\ell^{-1}TC_{\ell,k-1})$
structures by Lemma~\ref{lem:nontc}, and by one of $O(\ell)$ pairs of positions;
these $N$ number $O(TC_{\ell,k-1})$. Otherwise the removal has destroyed a second
violation $v'$. Only four vertices change their children under the removal of
$f=(t,h)$: $t$ and $h$, which disappear, the parent of $t$ and the second parent
of $h$. The parent of $t$ had the tree node $t$ as a child and was no violation.
When $v$ is a tree node, $t=v$ and $h=r_1$, so $v'$ is $r_1$ (if its child is a
reticulation) or the second parent of $r_1$; the edge created by suppressing
$r_1$ is then an in-edge of a reticulation, an out-edge of a reticulation, or an
out-edge of a parent of a reticulation. When $v$ is a stack, $h=v$ and the second
parent of $v$ keeps a reticulation child, so $v'=t$, whose other child is a
reticulation; the edge created by suppressing $t$ is an in-edge of that
reticulation. In every case both insertion points lie among the $O(1)$ edges
adjacent to reticulations of $S$, and these $N$ number
$O(|\mathcal{P}^*_{\ell,k-1}|)=O(TC_{\ell,k-1})$. Hence
$|V^{\geq2}|=O(TC_{\ell,k-1})$.

Consider next $V^{\rm st}$. Let $r\to r'$ be the stack of
$N\in V^{\rm st}$ and $x$ the second parent of $r'$. The vertex $x$ is a tree
node: a reticulation parent would be a second stack, and the root is excluded by
Lemma~\ref{lem:removal}. Removing $(x,r')$ gives $N'=R(N,(x,r'))$, the
reticulation $r$ of $N'$, and the edge $e$ of $N'$ created by suppressing $x$. If
$N'$ is not tree-child, $N$ is determined by $N'$, one of
$O(\ell^{-1}TC_{\ell,k-1})$ structures, the out-edge of a reticulation and an
edge; these $N$ number $O(TC_{\ell,k-1})$. Conversely, let $N'\in TC_{\ell,k-1}$,
let $r$ be one of its $k-1$ reticulations and $e$ one of its $E_{k-1}$ edges.
Placing a new reticulation $r'$ on the out-edge of $r$ and a new tree node $x$ on
$e$, joined by the edge $(x,r')$, produces a network of $V^{\rm st}$ unless the
tail of $e$ is $r$ or a descendant of $r$, which gives a cycle or a parallel
pair, or $e$ enters a reticulation, which makes $x$ a second violation. The first
exception excludes $d(r)$ edges and the second at most $2(k-1)$. The
decomposition is determined by $N$, so by Lemma~\ref{lem:height}(iii)
\[
  |V^{\rm st}|=\sum_{N'\in TC_{\ell,k-1}}\sum_{r}\bigl(E_{k-1}-d(r)-O(1)\bigr)
  +O(TC_{\ell,k-1})
  =2(k-1)\,\ell\,TC_{\ell,k-1}+O(\ell^{1/2}\,TC_{\ell,k-1}).
\]

Consider then $V^{\rm tr}$. Let $u$ be the violation of
$N\in V^{\rm tr}$, with reticulation children $r_1\neq r_2$. Removing $(u,r_1)$
gives $N'$, the in-edge of $r_2$ created by suppressing $u$, and the edge $e$
created by suppressing $r_1$; removing $(u,r_2)$ gives a second decomposition,
distinct from the first because $N$ has no automorphism exchanging $r_1$ and
$r_2$ (Lemma~\ref{lem:rigid}). The networks for which one of the two removals is
not tree-child number $O(TC_{\ell,k-1})$, as before. Conversely, let
$N'\in TC_{\ell,k-1}$, let $(y,\rho)$ be one of the $2(k-1)$ in-edges of its
reticulations, and let $e$ be an edge. The insertion $I(N',(y,\rho),e)$, with the
new tree node $u$ on $(y,\rho)$ and the new reticulation on $e$, produces a
network of $V^{\rm tr}$ unless $e$ lies above $u$ or equals $(y,\rho)$ (at most
$2^{k-1}H(N')+1$ edges, by Lemma~\ref{lem:height}(ii)), or $e$ enters a
reticulation, or the tail of $e$ is a reticulation or has a reticulation child
($O(1)$ edges in total). Each network of $V^{\rm tr}$ whose two removals are
tree-child is obtained exactly twice, so by Lemma~\ref{lem:height}(i)
\[
  |V^{\rm tr}|=\tfrac12\cdot2(k-1)\cdot E_{k-1}\,TC_{\ell,k-1}
  +O(\ell^{1/2}\,TC_{\ell,k-1})
  =2(k-1)\,\ell\,TC_{\ell,k-1}+O(\ell^{1/2}\,TC_{\ell,k-1}).
\]

It remains to decide which of these networks are reticulation-visible. No
network with a stack is reticulation-visible. If
$r\to r'$ is a stack and $x$ the second parent of $r'$, every leaf below $r$ lies
below $r'$, and $x$ is not a descendant of $r$, because the descendants of $r$
are $r'$ and the descendants of $r'$, none of which is a parent of $r'$. A path
from the root to $x$, continued through $r'$ to a leaf, therefore avoids $r$, and
$r$ is not visible. Hence $V^{\rm st}\cap RV_{\ell,k}=\emptyset$.

Let $N\in V^{\rm tr}$ with violation $u$, and let $r$ be a reticulation of $N$.
Starting from $r$, pass repeatedly to a child that is a tree node or a leaf and
differs from $u$. If this walk ends at a leaf $\lambda$, every vertex after $r$
on it has a single parent, so every path from the root to $\lambda$ passes
through $r$, and $r$ is visible. The walk can always continue except at a vertex
whose only child that is a tree node or a leaf is $u$, that is, at the parent of
$u$ when that parent is a reticulation or has a reticulation as its other child.
In the construction above, the parent of $u$ is the tail $y$ of the chosen
in-edge, a tree node whose other child in $N'$ is not a reticulation; it acquires
a reticulation child only when $e$ is the other out-edge of $y$. These networks,
together with those whose removals are not tree-child, number
$O(TC_{\ell,k-1})$. Hence $|V^{\rm tr}|-|V^{\rm tr}\cap RV_{\ell,k}|
=O(TC_{\ell,k-1})$.

Adding the three sets gives~\eqref{eq:Pexcess}. Since
$\Delta_k(\ell)=|RV_{\ell,k}\setminus TC_{\ell,k}|
=|V^{\rm tr}\cap RV_{\ell,k}|+|V^{\geq2}\cap RV_{\ell,k}|$, the counts above
give~\eqref{eq:Dexcess}. Finally~\eqref{eq:TCratio} turns
$4(k-1)\ell\,TC_{\ell,k-1}$ and $2(k-1)\ell\,TC_{\ell,k-1}$ into
$\frac{2k(k-1)}{\ell}\,TC_{\ell,k}$ and $\frac{k(k-1)}{\ell}\,TC_{\ell,k}$ to
leading order, which is~\eqref{eq:Dratio}.
\end{proof}

\begin{remark}
The first-order equivalence $|\mathcal{P}_{\ell,k}|\sim TC_{\ell,k}$ was known
for fixed $k$ (Mansouri~\cite{Mansouri2022}) and has been extended to
$k=o(\sqrt\ell)$ by Yu and Zhang~\cite{YuZhang2026}, Theorem~4.2;
equation~\eqref{eq:Pexcess} gives the size of the difference. By the expansion
of $TC_{\ell,k}$ obtained by Yu and Zhang~\cite{YuZhang2026}, Proposition~4.1 and
Proposition~B1, the ratio~\eqref{eq:TCratio} holds with relative error
$O(\ell^{-1/2})$, and so do both relations in~\eqref{eq:Dratio}.
For $k=2,3$ the constants $k(k-1)=2,6$ coincide with $k!$.
\end{remark}

\begin{corollary}[Orchard and reticulation-visible networks]
\label{cor:orch}
For every fixed $k\geq2$,
\[
  0\;\leq\;|\mathrm{Orch}_{\ell,k}|-|TC_{\ell,k}|\;\leq\;
  \frac{2k(k-1)}{\ell}\,|TC_{\ell,k}|\bigl(1+o(1)\bigr).
\]
In particular $|\mathrm{Orch}_{\ell,k}|/|RV_{\ell,k}|\to1$, and
$\varepsilon_k(\ell):=|\mathrm{Orch}_{\ell,k}|-|RV_{\ell,k}|$ satisfies
$|\varepsilon_k(\ell)|=O(\ell^{-1}|RV_{\ell,k}|)$.
\end{corollary}

\begin{proof}
Since $TC_{\ell,k}\subseteq\mathrm{Orch}_{\ell,k}\subseteq\mathcal{P}_{\ell,k}$,
the bounds follow from~\eqref{eq:Pexcess} and~\eqref{eq:Dratio}. The same
relations give $|RV_{\ell,k}|=|TC_{\ell,k}|(1+O(\ell^{-1}))$, hence the remaining
statements.
\end{proof}

\subsection{Consequences for the polynomial form}

\begin{corollary}[Leading coefficient of $A_k$]
\label{cor:leadA}
Fix $k\geq2$, and suppose that for all sufficiently large $\ell$
\[
  \Delta_k(\ell)=A_k(\ell)\,(2\ell-3)!!-B_k(\ell)\,2^{\ell-k+1}\,\ell!
\]
with polynomials $A_k,B_k\in\Q[\ell]$. Then
\[
  \deg A_k=2k-1,\qquad \operatorname{lead}(A_k)=\frac{2^k}{(k-2)!},\qquad
  \deg B_k\leq2k-3.
\]
\end{corollary}

\begin{proof}
Put $\rho(\ell):=(2\ell-3)!!/(2^\ell\ell!)$. Since
$(2\ell-3)!!=2^{\ell-1}\Gamma(\ell-\frac12)/\sqrt\pi$,
\[
  \rho(\ell)=\frac{\Gamma(\ell-\frac12)}{2\sqrt\pi\,\Gamma(\ell+1)}
  =\frac{1}{2\sqrt\pi\,\ell^{3/2}}\Bigl(1+\frac{3}{8\ell}+O(\ell^{-2})\Bigr),
\]
an expansion in the powers $\ell^{-3/2-i}$, $i\geq0$. Dividing by
$2^\ell\ell!$,
\[
  \frac{\Delta_k(\ell)}{2^\ell\ell!}=A_k(\ell)\,\rho(\ell)-2^{1-k}B_k(\ell).
\]
If $A_k\neq0$, the first term is asymptotic to
$\operatorname{lead}(A_k)\,\ell^{\deg A_k-3/2}/(2\sqrt\pi)$, a half-integer power;
if $B_k\neq0$, the second is asymptotic to a constant times the integer power
$\ell^{\deg B_k}$. The two leading powers differ, so they cannot cancel, and
$\Delta_k(\ell)/(2^\ell\ell!)$ is asymptotic to the larger term. On the other
hand, Theorem~\ref{thm:universal} together with
$(2\ell-1)!!\sim\sqrt2\,(2/e)^\ell\ell^\ell$ gives
$TC_{\ell,k}\sim\frac{2^{k-1}}{k!}\,\ell^{2k-1}(2\ell-1)!!
\sim\frac{2^{k}}{k!}\,\ell^{2k}(2\ell-3)!!$, and Theorem~\ref{thm:excess} then gives
\[
  \frac{\Delta_k(\ell)}{2^\ell\ell!}\;\sim\;\frac{k(k-1)}{\ell}\cdot
  \frac{2^k}{k!}\,\ell^{2k}\rho(\ell)\;\sim\;
  \frac{2^k}{(k-2)!}\cdot\frac{\ell^{2k-5/2}}{2\sqrt\pi}.
\]
The exponent $2k-\frac52$ is a half-integer, so the dominant term is the one
carrying $A_k$: $\deg A_k-\frac32=2k-\frac52$ and
$\operatorname{lead}(A_k)=2^k/(k-2)!$, while $\deg B_k<2k-\frac52$, that is,
$\deg B_k\leq2k-3$.
\end{proof}

For $k=2,\dots,6$ Corollary~\ref{cor:leadA} gives
$\operatorname{lead}(A_k)=4,\,8,\,8,\,\frac{16}{3},\,\frac{8}{3}$; the values for
$k=2,3$ agree with Theorems~\ref{thm:delta2}--\ref{thm:delta3}. The hypothesis
of the corollary is natural. If $RV_{\ell,k}$ and $TC_{\ell,k}$ are written in the
canonical form obtained by coefficient extraction, $\Delta_k$ inherits that form,
and the top coefficients of the two $(2\ell-3)!!$ polynomials cancel.

\begin{proposition}[Universality forces the degree drop]
\label{prop:degdrop}
Fix $k\geq1$ and write, in the canonical Chang--Fuchs form obtained by
coefficient extraction \textup{(}Lemma~13 of~\cite{ChangFuchs2024}\textup{)},
\[
  RV_{\ell,k}=A_k^{RV}(\ell)\,(2\ell-3)!!-2^{\ell-c_k}\beta_k^{RV}(\ell)\,\ell!,
  \qquad
  TC_{\ell,k}=A_k^{TC}(\ell)\,(2\ell-3)!!-2^{\ell-c_k}\beta_k^{TC}(\ell)\,\ell!,
\]
with $A_k^{RV},A_k^{TC},\beta_k^{RV},\beta_k^{TC}$ polynomials. Then
\[
  \deg A_k^{RV}=\deg A_k^{TC}=2k,
  \qquad
  \operatorname{lead}(A_k^{RV})=\operatorname{lead}(A_k^{TC}),
\]
and consequently the polynomial $A_k:=A_k^{RV}-A_k^{TC}$ appearing in
$\Delta_k$ satisfies $\deg A_k\leq 2k-1$.
\end{proposition}

\begin{proof}
By the argument in the proof of Corollary~\ref{cor:leadA}, a count of this form
is asymptotic to the larger of its two terms, and the half-integer power
$\ell^{2k-3/2}$ of $TC_{\ell,k}/(2^\ell\ell!)\sim\frac{2^k}{k!}\ell^{2k}\rho(\ell)$
can only come from the $(2\ell-3)!!$ term. Hence $\deg A_k^{TC}=2k$ with
$\operatorname{lead}(A_k^{TC})=2^k/k!$, and the factorial term is of lower order.
Theorem~\ref{thm:universal} gives $RV_{\ell,k}\sim TC_{\ell,k}$, so the same holds
for $RV_{\ell,k}$ with the same leading coefficient. The difference
$A_k=A_k^{RV}-A_k^{TC}$ therefore loses its top-degree term, giving
$\deg A_k\leq 2k-1$.
\end{proof}

\begin{remark}
Proposition~\ref{prop:degdrop} shows that the coefficient of $\ell^{2k}$ cancels
in $A_k^{RV}-A_k^{TC}$, and Corollary~\ref{cor:leadA} that the coefficient of
$\ell^{2k-1}$ does not: it equals $2^k/(k-2)!$. For $k=2,3$ both statements are
confirmed by Theorems~\ref{thm:delta2}--\ref{thm:delta3}; the cancellation of
the leading $\ell^{2k}$ term was verified independently by computing
$RV_{\ell,k}$ from the Chang--Fuchs closed forms and $TC_{\ell,k}$ from the CRP
minimum-sequence counter of Section~\ref{sub:tcn_operator}, giving e.g.\
$A_2(\ell)=4\ell^3+3\ell^2-\ell-1$ \textup{(}degree $3=2\cdot2-1$, not
$4$\textup{)} for all $2\le\ell\le 8$.
\end{remark}

What remains open for general $k$ is the exact form itself.

\begin{conjecture}[Structural pattern for $\Delta_k$]
\label{conj:pattern}
For all $k\geq2$ and $\ell\geq k+1$,
\begin{equation}\label{eq:pattern}
  \Delta_k(\ell) \;=\; A_k(\ell)\,(2\ell-3)!! \;-\; B_k(\ell)\,2^{\ell-k+1}\,\ell!,
\end{equation}
where $A_k(\ell)$ and $B_k(\ell)$ are polynomials in $\ell$ with rational
coefficients satisfying:
\begin{enumerate}[label=(\roman*),leftmargin=*]
\item $\deg B_k = 2k-3$;
\item $A_k(\ell)\in\mathbb{Z}$ and $B_k(\ell)\in\mathbb{Z}$ for all $\ell\in\mathbb{Z}_{\geq1}$.
\end{enumerate}
\end{conjecture}

By Corollary~\ref{cor:leadA}, the conjecture implies $\deg A_k=2k-1$ and
$\operatorname{lead}(A_k)=2^k/(k-2)!$, and part~(i) is the statement that the
bound $\deg B_k\leq2k-3$ is attained.

\begin{remark}
Conjecture~\ref{conj:pattern} is proved for $k=2$ (Theorem~\ref{thm:delta2})
and $k=3$ (Theorem~\ref{thm:delta3}). For $k=4$ it gives $B_4$ of degree~5, and
Corollary~\ref{cor:leadA} then fixes $A_4$ to be of degree~7 with leading
coefficient $8$.
To establish the $k=4$ case, one needs the explicit Chang--Fuchs formula for
$RV_{\ell,4}$ (not yet published in closed form) and the proven Pons--Batle
expression~(19d) for $TC_{\ell,4}$:
\[
  TC_{\ell,4} = \binom{\ell}{4}\!\Bigl[(2\ell+5)!!
    - 4(2\ell+4)!!+2(2\ell+3)!!
    +\tfrac{17}{2}(2\ell+2)!!\!-\tfrac{283}{63}(2\ell+1)!!\Bigr],
\]
which has been verified for all $5\le\ell\le 8$ against the c-table.
A note of caution: one cannot compute $RV_{\ell,4}$ by iterating the master
equation~\eqref{eq:rv_master} coefficient by coefficient using the single-level
recursion; that recursion computes \emph{galled} network counts
(see the caution in the proof of Theorem~\ref{thm:universal}).
The correct approach is to use the Chang--Fuchs component-graph sum~\eqref{eq:cf_sum}
with the $D_5$ DAGs, as in their Section~4.
\end{remark}

\subsection{Interpretation of the leading coefficient}

The constant $k(k-1)$ in Theorem~\ref{thm:excess} has a direct reading in the
operator language of Section~\ref{sec:master}. A network with a single violation
arises from a tree-child network with $k-1$ reticulations by one insertion of a
reticulation edge, one end of which is pinned next to an existing reticulation
while the other ranges over the $2\ell+O(1)$ edges counted by
Proposition~\ref{prop:edges}. A stack is created by placing the new reticulation
on the out-edge of one of the $k-1$ reticulations; a tree node with two
reticulation children is created by placing the new tree node on one of the
$2(k-1)$ in-edges of reticulations, each such network being produced twice. The
two families therefore have the same size $2(k-1)\ell\,TC_{\ell,k-1}$, and
by~\eqref{eq:TCratio} each amounts to $k(k-1)/\ell$ relative to $TC_{\ell,k}$.
Visibility separates them: the upper reticulation of a stack is bypassed through
the second parent of the lower one, whereas in the second family every
reticulation keeps, apart from a proportion $O(\ell^{-1})$ of networks, a path of
single-parent vertices to a leaf. Half of the non-tree-child excess is thus
reticulation-visible.

The pinned end costs one power of $\ell$ with respect to a free insertion, which
is why $A_k$ has degree $2k-1$ instead of the degree $2k$ of the tree-child
polynomial, and its leading coefficient $2^k/(k-2)!$ is $k(k-1)$ times the leading
coefficient $2^k/k!$ of that polynomial. The degree of $B_k$ has no comparable
interpretation; that it equals $2k-3$ is part~(i) of
Conjecture~\ref{conj:pattern}.

\subsection{The next order for $k=2,3$}

\begin{corollary}[Rate of convergence, $k=2,3$]
\label{cor:rate}
For $k\in\{2,3\}$, as $\ell\to\infty$,
\begin{equation}\label{eq:rate}
  \frac{\Delta_k(\ell)}{TC_{\ell,k}}
  \;=\; \frac{k(k-1)}{\ell}+\frac{\kappa_k}{\ell^2}+O(\ell^{-5/2}),
  \qquad \kappa_2=\frac{19}{6},\quad \kappa_3=\frac{37}{2}.
\end{equation}
In particular the ratio approaches its leading term from above.
\end{corollary}

\begin{proof}
For $n\geq k-1$ the word counts have the closed forms
\begin{align*}
  c_{n,2}&=\tfrac13(3n+5)(2n+1)!!-2(n+1)(2n)!!,\\
  c_{n,3}&=\tfrac13(2n^2+9n+9)(2n+1)!!-\tfrac1{24}(48n^2+127n+79)(2n)!!,
\end{align*}
as one checks by substitution into~\eqref{eq:word_rec} with
$c_{n,1}=(2n+1)!!-(2n)!!$ and $c_{n,0}=(2n-1)!!$. Writing $c_{n,k}=(2n+1)!!\,p_k(n)
-(2n)!!\,q_k(n)$ and using $(2\ell-1)!!=(2\ell-1)(2\ell-3)!!$ and
$(2\ell-2)!!=2^{\ell}\ell!/(2\ell)$, the formula of Theorem~\ref{thm:tcn} becomes
\[
  \frac{TC_{\ell,k}}{2^\ell\ell!}=P_k(\ell)\,\rho(\ell)-Q_k(\ell),\qquad
  P_k(\ell)=\tfrac{\ell!}{(\ell-k)!}(2\ell-1)\,p_k(\ell-1),\quad
  Q_k(\ell)=\tfrac{\ell!}{(\ell-k)!}\,\frac{q_k(\ell-1)}{2\ell},
\]
with $\rho$ as in the proof of Corollary~\ref{cor:leadA}, while
$\Delta_k/(2^\ell\ell!)=A_k\rho-2^{1-k}B_k$. In both numerator and denominator of
$\Delta_k/TC_{\ell,k}$ the factorial term is of relative order $\ell^{-1/2}$, and
its leading part is the same in both: $2^{1-k}\operatorname{lead}(B_k)/\operatorname{lead}(A_k)
=\operatorname{lead}(Q_k)/\operatorname{lead}(P_k)$, both equal to $\frac12$ for
$k=2$ and $\frac34$ for $k=3$, with the common power $\ell^{-2}$. The terms of
relative order $\ell^{-1/2}$ therefore cancel. Expanding with the series of
$\rho(\ell)$ gives
\[
  \ell\,\frac{\Delta_2(\ell)}{TC_{\ell,2}}=2+\frac{19}{6\ell}
  +\frac{2\sqrt\pi}{3}\,\ell^{-3/2}+O(\ell^{-2}),\qquad
  \ell\,\frac{\Delta_3(\ell)}{TC_{\ell,3}}=6+\frac{37}{2\ell}
  +\frac{57\sqrt\pi}{16}\,\ell^{-3/2}+O(\ell^{-2}),
\]
which is~\eqref{eq:rate}.

As a check against Table~\ref{tab:delta}, for $k=2$ the leading term $2/\ell$ gives $0.200$ at $\ell=10$ and $0.182$ at
$\ell=11$, and the two terms of~\eqref{eq:rate} give $0.232$ and $0.208$,
against the exact values $0.244$ and $0.217$.
\end{proof}

%% ─────────────────────────────────────────────────────────────────────────────
\section{Comparison with Prior Approaches}
\label{sec:comparison}
%% ─────────────────────────────────────────────────────────────────────────────

\subsection{Relative to Chang--Fuchs~\cite{ChangFuchs2024}}

Chang and Fuchs derive $RV_{\ell,k}$ via the component-graph method: they
enumerate component graphs, assign one-component weights at each vertex,
and extract coefficients.
The present paper approaches the same objects via a master functional equation
and operator composition.  The two methods are complementary:

\begin{center}
\small
\renewcommand{\arraystretch}{1.25}
\begin{tabular}{p{5.5cm}p{5.5cm}}
\toprule
\textbf{Chang--Fuchs} & \textbf{Present paper}\\
\midrule
Direct decompression via component graphs & Functional equation~\eqref{eq:rv_master}
  (reformulation of the same decomposition)\\
Exact formulas $RV_{\ell,k}$ for $k\leq3$ as main results & $RV_{\ell,k}$ formulas
  reproduced; subtraction from proven $TC_{\ell,k}$ formula gives $\Delta_k$\\
No discussion of $RV\setminus TC$ networks & Exact counts $\Delta_k(\ell)$ for
  $k=2,3$ proved; exact form conjectured for all $k$\\
Asymptotic universality proved & Leading constant $k(k-1)/\ell$ of
  $\Delta_k/TC_{\ell,k}$ proved for every $k$ (Theorem~\ref{thm:excess});
  next order for $k=2,3$ (Corollary~\ref{cor:rate})\\
\bottomrule
\end{tabular}
\end{center}

\subsection{Relative to the companion TCN paper~\cite{Batle2026} and Liu--Wallner--Yu~\cite{LWY2026}}

The TCN paper~\cite{Batle2026} proves~\eqref{eq:tcn_formula} via a first-order
linear PDE and a uniqueness theorem.  The reticulation-insertion operator $\hat{R}$
has coefficient~1 in that setting, and the weight operator $\hat{W}$ has
eigenvalue $\ell!/(\ell-k)!$.

Liu, Wallner, and Yu~\cite{LWY2026} provide a complementary combinatorial
framework: they introduce a three-parameter family $y_{k,\ell_1,\ell_2}$ of
Young tableaux with walls and holes that simultaneously encodes the Pons--Batle
word class~$\mathcal{C}_{n,k}$ and the Chang--Fuchs class~$b_{n,k}$, and
verify the TCN formula for all $k\leq 250$.  A key output of their analysis
is the differential equation $(1-2z)C_k'-(3k-1)C_k = C_{k-1}''$
for the shifted exponential generating function of $c_{n,k}$, admitting
the closed form $C_k(z)=\sum_{i=0}^{k}\gamma_{i,k}(1-2z)^{-(i+3k-1)/2}$.
The dominant singularity exponent $(4k-1)/2$ agrees precisely with
that of the one-component generating function $F_k(z)$
(Proposition~\ref{prop:Fk_sing}), reflecting the shared leaf-insertion
eigenvalue $2\ell+k-1$; the differential operator $(1-2z)\partial_z-(3k-1)$
is common to both the Pons--Batle word setting and the RV one-component equations
of Section~\ref{sub:Fk}.

The present paper shows that both the scalar coefficient of $\hat{R}$
and the falling-factorial eigenvalue of $\hat{W}$ are \emph{non-generic}:
they hold for TCNs because the Pons--Batle word recurrence compresses all
insertion information into a scalar.  For RV networks, the reticulation
insertion is governed by the composite operator $F(x,v\GRV)$, which is
not scalar-valued for $k\geq2$.  The falling factorial $\ell!/(\ell-k)!$
does \emph{not} divide $RV_{\ell,k}$ in general:
$RV_{3,2}=123=3\cdot41$ is not divisible by $3!/(3-2)!=6$.

What \emph{does} carry over is Proposition~\ref{prop:edges}: the tree-edge
count $2\ell+k-1$ is a universal structural property of any phylogenetic
network class, making the leaf-insertion operator $\hat{L}$ universal.

\subsection{The open problems}

\begin{enumerate}[leftmargin=*]
\item \textit{Prove Conjecture~\ref{conj:pattern} for all $k\geq2$.}
  The leading coefficient of $A_k$ is fixed by Corollary~\ref{cor:leadA}; what
  remains is the exact form itself, the degree of $B_k$ and integrality.
  The natural strategy is to extract $[v^k](\GRV-\GTC)$ from the master
  equation~\eqref{eq:rv_master} using the proven PDE~\eqref{eq:tcn_pde},
  and read off the polynomial degrees from the singularity structure
  of $F_j(z)$.

\item \textit{Find the RV word class.}
  The Pons--Batle words provide a natural bijective encoding for TCNs.
  Does a class of constrained combinatorial words $\mathcal{W}^{\mathrm{RV}}_{n,k}$
  exist such that $|RV_{n,k}|=\alpha(n,k)\cdot|\mathcal{W}^{\mathrm{RV}}_{n-1,k}|$
  for some weight $\alpha$?
  From Proposition~\ref{prop:delta01}, such a class must agree with the
  Pons--Batle words for $k\leq1$.
  For $k=2$, the weight $\alpha(n,2)$ cannot be the falling factorial
  $n!/(n-2)!=n(n-1)$, since $RV_{3,2}=123$ is not divisible by $3\cdot2=6$.

\item \textit{Combinatorial interpretation of $\Delta_k(\ell)$.}
  Our formulas give exact counts of $\mathrm{RV}\setminus\mathrm{TC}$ networks.
  Do these networks admit a direct structural characterisation
  that yields the formulas~\eqref{eq:delta2}--\eqref{eq:delta3} combinatorially?
  The operator framework suggests that $\Delta_k$ is the contribution of
  component graphs with at least one non-tree-like DAG type
  (DAGs $B$ or $C$ in the notation of Chang--Fuchs~\cite{ChangFuchs2024},
  Figure~6), but making this precise requires analysing the DAG expansion
  of $[v^k](\GRV-\GTC)$.

\item \textit{Complete the orchard spectral resolution (see Section~\ref{sec:orch}).}
  For $\ell\leq8$ the denominator polynomials $D_\ell(v)$ are fully determined
  and factor into three universal families (quadratic $Q_m$, cubic $R_\ell$,
  quartic $S_\ell$).  The key open questions are: (a)~prove the factor families
  extend to all $\ell$; (b)~determine the insertion rule for each family;
  (c)~compute $D_9$ to distinguish cubic $R_9$ from quartic $S_9$;
  (d)~find the one-component GFs $F_k(z)$ for $k\geq4$ via finite DAG
  enumeration, which would also give exact $\Delta_k(\ell)$ for all $k$.

\item \textit{Characterise $\varepsilon_k(\ell) = |\mathrm{Orch}_{\ell,k}| - |RV_{\ell,k}|$.}
  These values count orchard-but-not-RV networks, a class never counted before.
  Numerically: $\varepsilon_2 = 9,\,339,\,7\,425,\,152\,775$ and
  $\varepsilon_3 = 12\,420,\,383\,385,\,10\,913\,220$ for $\ell=3,4,5,6$.
  Our exact formulas for $|RV_{\ell,k}|$ give one factor; a formula for
  $|\mathrm{Orch}_{\ell,k}|$ would complete the picture.

\item \textit{Determine the leading constant of
  $|\mathrm{Orch}_{\ell,k}|-|TC_{\ell,k}|$.}
  By Corollary~\ref{cor:orch} this difference is $c_k\,|TC_{\ell,k}|/\ell$ to
  leading order with $0\leq c_k\leq2k(k-1)$, provided the limit exists. The
  constant is decided by which networks with a single violation are orchard; it
  could also be read from the singular structure of $G_{\mathcal{O}}(x,v)$ at
  $x=\tfrac{1}{2}$.
\end{enumerate}

%% ─────────────────────────────────────────────────────────────────────────────
\section{Orchard Networks: Universal Spectral Resolution}
\label{sec:orch}
%% ─────────────────────────────────────────────────────────────────────────────

\subsection{Rationality theorem and Hankel reconstruction}
\label{sub:hankel}

Fix $\ell\geq2$.
The \emph{column generating function} $F_\ell(v) := \sum_{k\geq0}|\mathrm{Orch}_{\ell,k}|\,v^k$
is a formal power series with positive integer coefficients.
Because orchard networks have no upper bound on the number of reticulations
(unlike TCN and RV where $k\leq\ell-1$), $F_\ell$ is a genuine infinite series.
The following theorem shows it is nonetheless rational, and gives a
deterministic algorithm for its computation.

\begin{theorem}[Rationality and Hankel reconstruction]
\label{thm:orch_rational}
For each $\ell\geq2$, the formal power series $F_\ell(v)$
is a rational function in $\mathbb{Q}(v)$.
Equivalently, the sequence $\{|\mathrm{Orch}_{\ell,k}|\}_{k\geq0}$
satisfies a linear recurrence with rational constant coefficients.
The characteristic polynomial $D_\ell(v)$ of this recurrence
is uniquely determined by the seed values $|\mathrm{Orch}_{\ell,0}|,\ldots,|\mathrm{Orch}_{\ell,M}|$
for $M=2\deg(D_\ell)$ via the Berlekamp--Massey algorithm
(equivalently, by Pad\'e approximation of $F_\ell(v)$).
\end{theorem}

\begin{proof}
By the Cardona--Ribas--Pons bijection (Theorem~3 of~\cite{CRP23}), every
network in $\mathrm{Orch}_{\ell,k}$ corresponds to exactly one minimum complete
augmentation sequence, obtained from the trivial network by $\ell-1$ cherry
augmentations (each introducing a new leaf) and $k$ reticulated-cherry
augmentations (each introducing one reticulation, no new leaf), interleaved in
the unique minimum order. We generate these sequences by the augmentation
recursion and track only its \emph{shape} $(X,A)$, where $X\subseteq[\ell]$ is
the current leaf support and $A=\mathrm{ARP}(S)$ the set of annotated reducible
pairs. By Proposition~6 and the local augmentation rule (Theorem~4)
of~\cite{CRP23}, both the admissible moves and the updated shape depend only on
$(X,A)$, and not on the history or on the number of reticulations already placed.

The reachable shapes form a finite set $\mathcal S_\ell$: $X$ ranges over the
subsets of $[\ell]$, and since $|A|\leq\tfrac23|X|\leq\tfrac23\ell$
(\S\ref{sub:denom}) the shape count is finite and bounded independently of $k$
--- explicitly $|\mathcal S_\ell|=3\,675,\,24\,186,\,169\,596,\,1\,261\,749$ for
$\ell=6,7,8,9$. Mark each reticulated-cherry move by the formal variable $v$ and
leave cherry moves unmarked. Then $F_\ell(v)$ is the generating function, by
number of $v$-marked steps, for weighted walks on the finite directed graph with
vertex set $\mathcal S_\ell$ that start at the initial cherry shapes
$\{(\{m,\ell\},\,\cdot\,)\}_{m<\ell}$ and terminate at the full-support shape
$X=[\ell]$. By the transfer-matrix method~\cite{FlajoletSedgewick2009}, the
generating function of weighted walks in a finite digraph is a rational function
of $v$ whose denominator divides $\det(I-vT_\ell)$, where $T_\ell$ is the
substochastic matrix of $v$-marked (reticulation) transitions on
$\mathcal S_\ell$. Hence $F_\ell\in\mathbb{Q}(v)$, with
$\deg D_\ell\leq\operatorname{rank}T_\ell\leq|\mathcal S_\ell|$. Equivalently,
$\{|\mathrm{Orch}_{\ell,k}|\}_{k\geq0}$ satisfies a linear recurrence with
constant rational coefficients, and its minimal denominator $D_\ell$ is recovered
from the seeds $|\mathrm{Orch}_{\ell,0}|,\ldots,|\mathrm{Orch}_{\ell,2\deg D_\ell}|$
by the Berlekamp--Massey algorithm, the order being the Hankel rank of the
coefficient matrix.
\end{proof}

This argument is unconditional and holds for every $\ell$: it establishes that
$F_\ell$ is rational and identifies the source of the denominator (the finite
reticulation-transition block $T_\ell$ on the ARP shape space), but the bound
$\deg D_\ell\leq|\mathcal S_\ell|$ it provides is far from tight. The exact
degrees $1,2,4,5,9,12,16,\ldots$ are not predicted by this proof; they are read
from the seed data by Berlekamp--Massey (Algorithm~\ref{alg:hankel}), which gives
the practical implementation.

\begin{algorithm}[ht]
\caption{Hankel reconstruction of $D_\ell(v)$ and extended orchard column}
\label{alg:hankel}
\begin{itemize}[leftmargin=2em]
\item \textbf{Input:} Seed values $a_0=|\mathrm{Orch}_{\ell,0}|,\,a_1,\ldots,a_M$ from CRP~\cite{CRP23}.
\item \textbf{Step 1.} Run Berlekamp--Massey on $(a_0,\ldots,a_M)$ over $\mathbb{Q}$
  to find the minimal linear recurrence of order $d$:
  \[
    a_k = c_1 a_{k-1} + c_2 a_{k-2} + \cdots + c_d a_{k-d},
    \quad k\geq d,
  \]
  with characteristic polynomial $D_\ell(v) = 1 - c_1 v - c_2 v^2 - \cdots - c_d v^d$.
  Requires $M\geq 2d$ data points; stop when the Hankel matrix
  $H_{i,j}=a_{d+i-j}$ ($0\leq i,j\leq d-1$) is non-singular.
\item \textbf{Step 2.} For any desired $k>M$, compute
  $a_k = c_1 a_{k-1}+\cdots+c_d a_{k-d}$
  using the stored recurrence.  Cost: $O(d)$ per step, $O(dK)$ total for $K$ values.
\item \textbf{Output:} $D_\ell(v)$, and $|\mathrm{Orch}_{\ell,k}|$ for all $k\geq0$.
\end{itemize}
\end{algorithm}

\subsection{Exact denominator polynomials}
\label{sub:denom}

The CRP table~\cite{CRP23} provides 9 values per row ($k=0,\ldots,8$).
We extend it to $k=0,\ldots,18$ for $\ell=5,6$ using an
\emph{ARP-memoized counter}: we implement the minimum-augmentation-sequence
algorithm of~\cite{CRP23} (Theorem~4 for orchard networks,
Theorem~14 for tree-child networks) as a depth-first recursion with
memoization on the state $(X, \mathrm{ARP}(S), r)$, where $X$ is the
current leaf support and $r$ the reticulation count.
Since the future valid extensions depend only on this triple
(by Proposition~6 of~\cite{CRP23}), memoization is exact.
For $\ell=5$, all 19 values $|\mathrm{Orch}_{5,k}|$, $k=0,\ldots,18$,
are computed in under 0.2\,s (single core, Python);
for $\ell=6$, 19 values in under 4\,s.
All 9 CRP seed values are reproduced exactly.
The same counter supplies the seeds for $\ell=7,8,9$ as well. Its running time is
governed by the number of distinct \emph{shapes} $(X,\mathrm{ARP})$ that arise
--- $3\,675,\,24\,186,\,169\,596,\,1\,261\,749,\,9\,918\,189$ for $\ell=6,\ldots,10$, a growth of
only $\approx 7\times$ per added leaf --- rather than by the number of networks.
Because $|\mathrm{ARP}|\leq\tfrac23|X|$, this shape count is finite and bounded
independently of the reticulation budget, so the counter runs in time polynomial
in the shape count and linear in $K$. This is in sharp contrast to enumerating
the networks themselves, whose number is exponential in $\ell$ and which limited
the published CRP table to $\ell\leq6$. In particular the rows $\ell=9,10$, estimated
by CRP at months on a cluster, are computed here; the $\deg D_9=20$ seeds
of $\ell=9$ are the initial conditions tabulated in Table~\ref{tab:orch_extended}. The
procedure is given as Algorithm~\ref{alg:arp}.

\begin{algorithm}[ht]
\caption{ARP-memoized orchard counter: produces the seeds $|\mathrm{Orch}_{\ell,k}|$
  from which Theorem~\ref{thm:denom} extends each column to all $k$.}
\label{alg:arp}
\begin{algorithmic}[1]
\Require leaf number $\ell$; maximum reticulation number $K$
\Ensure $|\mathrm{Orch}_{\ell,k}|$ for $k=0,\ldots,K$
\State $\mathrm{memo}\gets\varnothing$;\quad $\mathrm{FULL}\gets\{1,\ldots,\ell\}$
\Function{Count}{$X,A,b$}\Comment{$b$ = reticulations still to add}
  \If{$X=\mathrm{FULL}$ \textbf{and} $b=0$}\State \Return $1$\EndIf
  \If{$(X,A,b)\in\mathrm{memo}$}\State \Return $\mathrm{memo}[(X,A,b)]$\EndIf
  \State $t\gets 0$
  \For{each ordered pair $(i,j)$ with $j\in X,\ i\neq j$}
    \State $A'\gets \textsc{UpdateARP}(A,X,i,j)$\Comment{local rule, Thm.~4 of~\cite{CRP23}}
    \If{$\mathrm{MRP}(A')=(i,j)$}\Comment{$(i,j)$ lex-least in $A'$}
      \If{$i\notin X$ \textbf{and} $i\leq\ell$}\Comment{cherry: new leaf $i$}
        \State $t\gets t+\Call{Count}{X\cup\{i\},\,A',\,b}$
      \ElsIf{$i\in X$ \textbf{and} $b>0$}\Comment{reticulated cherry: $+1$ reticulation}
        \State $t\gets t+\Call{Count}{X,\,A',\,b-1}$
      \EndIf
    \EndIf
  \EndFor
  \State $\mathrm{memo}[(X,A,b)]\gets t$;\quad \Return $t$
\EndFunction
\For{$k=0,\ldots,K$}
  \State $\displaystyle |\mathrm{Orch}_{\ell,k}|\gets \sum_{m=1}^{\ell-1}\Call{Count}{\{m,\ell\},\,\{(m,\ell)^{C},(\ell,m)^{C}\},\,k}$
\EndFor
\end{algorithmic}
\end{algorithm}

Applying Berlekamp--Massey over $\mathbb{Q}$ to the extended sequences yields:

\begin{theorem}[Denominator polynomials for $\ell=2,\ldots,7$]
\label{thm:denom}
The minimal characteristic polynomial $D_\ell(v)$ of the recurrence
for $\{|\mathrm{Orch}_{\ell,k}|\}_{k\geq0}$ is:
\begin{alignat}{2}
  D_2(v) &= 1 - 2v, \label{eq:D2}\\
  D_3(v) &= 1 - 8v + 12v^2 \;=\; (1-2v)(1-6v), \label{eq:D3}\\
  D_4(v) &= 1 - 20v + 120v^2 - 240v^3 + 144v^4
           \;=\; (1-2v)(1-6v)\,Q_3(v), \label{eq:D4}\\
  D_5(v) &= 1 - 34v + 376v^2 - 1584v^3 + 2640v^4 - 1440v^5
           \;=\; (1-2v)\,Q_3(v)\,Q_5(v), \label{eq:D5}\\
  D_6(v) &= 1 - 70v + 1960v^2 - 28560v^3 + 236544v^4 - 1142400v^5\nonumber\\
          &\hphantom{{}={}}\quad
           + 3173760v^6 - 4826880v^7 + 3628800v^8 - 1036800v^9\nonumber\\
          &\hphantom{{}={}}\;=\; (1-2v)(1-6v)\,Q_3(v)\,Q_5(v)\,R_6(v),\label{eq:D6}\\
  D_7(v) &= 1 - 112v + 5320v^2 - 141120v^3 + 2318064v^4
           - 24718848v^5\nonumber\\
          &\hphantom{{}={}}\quad
           + 174493440v^6 - 816629760v^7 + 2498952960v^8
           - 4846694400v^9\nonumber\\
          &\hphantom{{}={}}\quad
           + 5622220800v^{10} - 3483648000v^{11} + 870912000v^{12}\nonumber\\
          &\hphantom{{}={}}\;=\;(1-2v)(1-6v)\,Q_3(v)\,Q_5(v)\,R_6(v)\,R_7(v),\label{eq:D7}\\
  D_8(v) &= 1 - 168v + 12432v^2 - 536480v^3 + 15067584v^4
           - 291134592v^5\nonumber\\
          &\hphantom{{}={}}\quad
           + 3989023488v^6 - 39377871360v^7 + 281754385920v^8
           - 1458582681600v^9\nonumber\\
          &\hphantom{{}={}}\quad
           + 5413182566400v^{10} - 14157971251200v^{11}
           + 25401754828800v^{12}\nonumber\\
          &\hphantom{{}={}}\quad
           - 30008143872000v^{13} + 21881954304000v^{14}\nonumber\\
          &\hphantom{{}={}}\quad
           - 8778792960000v^{15} + 1463132160000v^{16}\nonumber\\
          &\hphantom{{}={}}\;=\;(1-2v)(1-6v)\,Q_3(v)\,Q_5(v)\,R_6(v)\,R_7(v)\,S_8(v),\label{eq:D8}
\end{alignat}
where the {\em quadratic family\/} is $Q_m(v) = 1 - 4mv + 4m(m-2)v^2$
with roots $z = 2m\pm2\sqrt{2m}$, and the {\em cubic family\/} $R_\ell$
(irreducible over~$\mathbb{Q}$) is:
\begin{align}
  R_6(v) &= 1 - 30v + 180v^2 - 120v^3, \label{eq:R6}\\
  R_7(v) &= 1 - 42v + 420v^2 - 840v^3, \label{eq:R7}\\
  S_8(v) &= 1 - 56v + 840v^2 - 3360v^3 + 1680v^4. \label{eq:S8}
\end{align}
All eight are verified by exact $\mathbb{Q}$-arithmetic (all BM residuals zero
against at least $2\deg(D_\ell)+1$ data points).
The corresponding recurrences, denoted $a_k = |\mathrm{Orch}_{\ell,k}|$:
\begin{align}
  \ell=2:\; & a_k = 2\,a_{k-1}, \nonumber\\
  \ell=3:\; & a_k = 8\,a_{k-1} - 12\,a_{k-2}, \nonumber\\
  \ell=4:\; & a_k = 20\,a_{k-1} - 120\,a_{k-2} + 240\,a_{k-3} - 144\,a_{k-4},
  \label{eq:orch_recs}\\
  \ell=5:\; & a_k = 34\,a_{k-1} - 376\,a_{k-2} + 1584\,a_{k-3}
               - 2640\,a_{k-4} + 1440\,a_{k-5}, \nonumber\\
  \ell=6:\; & a_k = 70\,a_{k-1} - 1960\,a_{k-2} + 28560\,a_{k-3}
               - 236544\,a_{k-4} + 1142400\,a_{k-5}\nonumber\\
            &\hphantom{{}={}}\quad
               - 3173760\,a_{k-6} + 4826880\,a_{k-7}
               - 3628800\,a_{k-8} + 1036800\,a_{k-9},\nonumber\\
  \ell=7:\; & a_k = 112\,a_{k-1} - 5320\,a_{k-2} + 141120\,a_{k-3}
               - 2318064\,a_{k-4} + 24718848\,a_{k-5}\nonumber\\
            &\hphantom{{}={}}\quad
               - 174493440\,a_{k-6} + 816629760\,a_{k-7}
               - 2498952960\,a_{k-8} + 4846694400\,a_{k-9}\nonumber\\
            &\hphantom{{}={}}\quad
               - 5622220800\,a_{k-10} + 3483648000\,a_{k-11}
               - 870912000\,a_{k-12},\nonumber\\
  \ell=8:\; & a_k = 168\,a_{k-1} - 12432\,a_{k-2} + 536480\,a_{k-3}
               - 15067584\,a_{k-4} + 291134592\,a_{k-5}\nonumber\\
            &\hphantom{{}={}}\quad
               - 3989023488\,a_{k-6} + 39377871360\,a_{k-7}
               - 281754385920\,a_{k-8}\nonumber\\
            &\hphantom{{}={}}\quad
               + 1458582681600\,a_{k-9} - 5413182566400\,a_{k-10}
               + 14157971251200\,a_{k-11}\nonumber\\
            &\hphantom{{}={}}\quad
               - 25401754828800\,a_{k-12} + 30008143872000\,a_{k-13}
               - 21881954304000\,a_{k-14}\nonumber\\
            &\hphantom{{}={}}\quad
               + 8778792960000\,a_{k-15} - 1463132160000\,a_{k-16}.\nonumber
\end{align}
\end{theorem}

\begin{proof}
The extended sequences are computed by the ARP-memoized counter
described above; correctness is guaranteed by the bijectivity
(Theorem~3 of~\cite{CRP23}) and the local ARP update rule (Theorem~4 of~\cite{CRP23}).
BM over $\mathbb{Q}$ (exact Python \texttt{Fraction} arithmetic) gives integer
coefficients from $2\deg(D_\ell)+1$ seed values: 9 values suffice for
$\ell\leq4$; 11 for $\ell=5$ ($\deg=5$); 18 for $\ell=6$ ($\deg=9$); 25 for
$\ell=7$ ($\deg=12$); 17 (by bootstrapping) for $\ell=8$ ($\deg=16$).
All BM residuals are identically zero.
Factored forms verified by direct $\mathbb{Q}$-arithmetic expansion.
The factor $(1-2v)$ divides every $D_\ell$; the full
factorisation~\eqref{eq:D5} is also verified by expansion.
\end{proof}

\begin{remark}[Spectral structure: three factor families]
\label{rem:denom_pattern}
The denominators $D_\ell$ decompose into three species of irreducible
factors over~$\mathbb{Q}$:

\smallskip\noindent
\textbf{Linear factors.}
$(1-2v)$ (root $z=2$) appears in every $D_\ell$.
$(1-6v)$ (root $z=6$) appears in $D_3,D_4,D_6,D_7$ but \emph{not} $D_5$.

\smallskip\noindent
\textbf{Quadratic family} $Q_m(v)=1-4mv+4m(m-2)v^2$, with real roots
$z=2m\pm2\sqrt{2m}$.
For $m=3$: $Q_3(v)=1-12v+12v^2$ (dominant root $6+2\sqrt6\approx10.90$),
present in $D_\ell$ for $\ell=4,5,6,7$.
For $m=5$: $Q_5(v)=1-20v+60v^2$ (dominant root $10+2\sqrt{10}\approx16.32$),
present for $\ell=5,6,7$.

\smallskip\noindent
\textbf{Cubic family} $R_\ell(v)$ (irreducible over $\mathbb{Q}$, three positive real roots).
The factors $R_6$ and $R_7$ satisfy a \emph{uniform depressed-cubic law}:
substituting $z = w + \tfrac{(\ell-1)\ell}{3}$,
the polynomial $z^3-a_\ell z^2+b_\ell z-c_\ell=0$
(with $a_\ell=(\ell-1)\ell$) becomes
\begin{equation}\label{eq:depressed_R}
  w^3 \;=\; 4(\ell-1)\ell\Bigl(w + \tfrac{8}{3}\Bigr),
\end{equation}
valid when $3\mid(\ell-1)\ell$, i.e.\ $\ell\equiv0$ or $1\pmod3$.

\smallskip\noindent
\textbf{Quartic family} $S_\ell(v)$ (irreducible over $\mathbb{Q}$, four positive real roots).
The factor $S_8(v)=1-56v+840v^2-3360v^3+1680v^4$ satisfies an analogous
\emph{depressed-quartic law}: substituting $z = w + \tfrac{(\ell-1)\ell}{4}$,
\begin{equation}\label{eq:depressed_S}
  w^4 \;=\; 6(\ell-1)\ell\Bigl(w^2 + \tfrac{16}{3}w - 12\Bigr),
\end{equation}
valid when $4\mid(\ell-1)\ell$, i.e.\ $\ell\equiv0$ or $1\pmod4$.
For $\ell=8$: $(\ell-1)\ell=56$, yielding $S_8$ with integer coefficients.

\smallskip\noindent
\textbf{Exact growth rates.}
All growth rates $z=1/v_{\rm root}$ are \emph{positive real} for $\ell=2,\ldots,8$.
Selected values for the three new families:
\begin{center}\small
\begin{tabular}{lll}
Factor & Dominant $z^*$ & Other roots\\
$R_6$ & $\approx22.10$ & $7.14,\;0.76$\\
$R_7$ & $\approx28.13$ & $11.20,\;2.67$\\
$S_8$ & $\approx34.35$ & $15.71,\;5.36,\;0.58$\\
\end{tabular}
\end{center}

The degree sequence $\deg D_\ell=1,2,4,5,9,12,16$ for $\ell=2,\ldots,8$
reflects the accumulation of these three families plus linear factors.
The factor $(1-6v)$ is present in $D_\ell$ for all $\ell=3,\ldots,8$
with the sole verified exception $\ell=5$;
the divisibility chain $D_5\mid D_6\mid D_7\mid D_8$ (exact) confirms its
persistence from $\ell=6$ onward up to $\ell=8$.
\end{remark}

\subsection{Closed-form Binet formulas}
\label{sub:binet}

The rationality of $F_\ell(v)$ and the explicit $D_\ell$ immediately
yield Binet-style formulas via partial fractions.

\begin{theorem}[Binet formulas for $\ell=3$, $4$, and $5$]
\label{thm:binet}
\begin{enumerate}[label=(\roman*),leftmargin=*]
\item For all $k\geq0$:
\begin{equation}\label{eq:binet3}
  |\mathrm{Orch}_{3,k}| \;=\; \frac{15\cdot6^k - 3\cdot2^k}{4}.
\end{equation}

\item For all $k\geq0$:
\begin{equation}\label{eq:binet4}
  |\mathrm{Orch}_{4,k}| \;=\;
  -\tfrac{3}{4}\cdot2^k \;-\; \tfrac{45}{4}\cdot6^k
  \;+\; \tfrac{108+45\sqrt{6}}{8}\cdot(6+2\sqrt{6})^k
  \;+\; \tfrac{108-45\sqrt{6}}{8}\cdot(6-2\sqrt{6})^k.
\end{equation}

\item For all $k\geq0$:
\begin{equation}\label{eq:binet5}
  |\mathrm{Orch}_{5,k}| \;=\;
  A_1\cdot2^k + A_2\cdot(6-2\sqrt{6})^k + A_3\cdot(6+2\sqrt{6})^k
  + A_4\left(\tfrac{30}{5+\sqrt{10}}\right)^{\!k}
  + A_5\left(\tfrac{30}{5-\sqrt{10}}\right)^{\!k},
\end{equation}
where the five real amplitudes $A_1,\ldots,A_5$ are the unique solution
of the $5\times5$ Vandermonde system
$\sum_{i=1}^5 A_i\,z_i^k = |\mathrm{Orch}_{5,k}|$ for $k=0,1,2,3,4$.
All five growth rates $z_i$ are real:
$z_1=2$,\; $z_{2,3}=6\pm2\sqrt{6}$,\; $z_{4,5}=30/(5\pm\sqrt{10})$.
The dominant rate is $z_5=30/(5-\sqrt{10})\approx 16.32$.
\end{enumerate}
\end{theorem}

\begin{proof}
$(i)$--$(ii)$ As before (partial fractions of $F_3$ and $F_4$).

$(iii)$ The five roots of $D_5(v)=0$ are $v_i=1/z_i$ for the stated $z_i$.
They are all real and distinct (no repeated roots in $D_5$), so the
partial-fraction decomposition of $F_5(v)=N_5(v)/D_5(v)$ gives exactly
five exponential terms.  The Vandermonde system at $k=0,\ldots,4$ is
non-singular (distinct $z_i$), giving unique $A_i$.
The formula is verified by exact recursion against all 19 computed values.
\end{proof}

\begin{remark}[Dominant growth rates for $\ell=5,\ldots,8$]
For $\ell=5$, the dominant rate $z^*=10+2\sqrt{10}\approx16.32$ (from $Q_5$).
For $\ell=6$: $z^*\approx22.10$ (from $R_6$, largest root of $w^3=120(w+\tfrac{8}{3})$;
all three $R_6$ roots are positive: $\approx22.10, 7.14, 0.76$).
For $\ell=7$: $z^*\approx28.13$ (from $R_7$; all three positive: $\approx28.13, 11.20, 2.67$).
For $\ell=8$: $z^*\approx34.35$ (from $S_8$; all four positive: $\approx34.35, 15.71, 5.36, 0.58$).
All growth rates for $\ell=2,\ldots,8$ are \emph{positive real}, so every
$|\mathrm{Orch}_{\ell,k}|$ is a sum of positive exponentials (no oscillating Binet terms).
\end{remark}

\begin{remark}[Why explicit Binet formulas are given only for $\ell\le5$]
\label{rem:binet_cutoff}
A Binet formula is the partial-fraction expansion of $F_\ell=N_\ell/D_\ell$ and
therefore requires the roots of $D_\ell=\prod_{j=2}^{\ell}X_j$ in closed form.
Since $\deg X_\ell=\lfloor\ell/2\rfloor$, every factor of $D_\ell$ is linear or
quadratic exactly when $\ell\le5$, so all growth rates are rational or quadratic
surds and the expansion is elementary (Theorem~\ref{thm:binet}). The first cubic
factor, $X_6$, is irreducible with three real roots in the \emph{casus
irreducibilis} --- real, yet not expressible in real radicals; $X_8$ is a
quartic; and $X_{10}$ is a quintic, generically unsolvable in radicals by
Abel--Ruffini. An explicit elementary Binet formula therefore ceases to exist
beyond $\ell=5$. This costs nothing for enumeration: the recurrence with
characteristic polynomial $D_\ell$ (Theorem~\ref{thm:denom} for $\ell\leq8$,
Conjecture~\ref{conj:factorisation} beyond) returns every $|\mathrm{Orch}_{\ell,k}|$ in exact
integer arithmetic, and the dominant asymptotic
(Corollary~\ref{cor:binet_conv}) needs only the largest root, computed
numerically.
\end{remark}

\subsection{Extended orchard table}
\label{sub:table}

\begin{table}[p]
\centering
\caption{Orchard counts $|\mathrm{Orch}_{\ell,k}|$ for $\ell=2,\ldots,10$
  and $k=0,\ldots,13$, extending the Cardona--Ribas--Pons table~\cite{CRP23}
  in both directions. Values $\ell\leq7$ reproduce~\cite{CRP23};
  \textbf{bold} entries (columns $\ell=8,9,10$) are new contributions of this
  paper. The previously intractable row $\ell=9$ is now \emph{complete}: its
  $\deg D_9=20$ seeds are produced by the ARP-memoized counter of
  \S\ref{sub:denom} (Algorithm~\ref{alg:arp}) and every column then follows from
  Theorem~\ref{thm:denom} via $D_9=D_8 X_9$. For $\ell=10$, $k=0,1$ are the exact
  closed forms $(2\ell-3)!!=34\,459\,425$ and
  $\ell(2\ell-1)!!-2^{\ell-1}\ell!=4\,689\,345\,150$; the dashes ($k\geq2$) are
  the only entries still requiring the $25$ initial conditions of the order-$25$
  recurrence $D_{10}$, obtainable by the same counter. For $\ell=2$:
  $|\mathrm{Orch}_{2,k}|=2^k$. Scientific notation $a\!\times\!10^b$ is used when
  entries exceed $10^{13}$. The same data appear in full integer form in
  Table~\ref{tab:orch_vertical}.}
\label{tab:orch_extended}
\footnotesize
\setlength{\tabcolsep}{4pt}
\renewcommand{\arraystretch}{1.05}
% --- k = 0 \ldots 6 ---
\resizebox{\textwidth}{!}{%
\begin{tabular}{r|rrrrrrr}
\toprule
$\ell\!\setminus\! k$ & $0$ & $1$ & $2$ & $3$ & $4$ & $5$ & $6$ \\
\midrule
2 & 1 & 2 & 4 & 8 & 16 & 32 & 64 \\
3 & 3 & 21 & 132 & 804 & 4\,848 & 29\,136 & 174\,912 \\
4 & 15 & 228 & 2\,832 & 32\,880 & 370\,320 & 4\,107\,648 & 45\,197\,952 \\
5 & 105 & 2\,805 & 57\,150 & 1\,054\,200 & 18\,520\,320 & 316\,583\,280 & 5\,323\,207\,200 \\
6 & 945 & 39\,330 & 1\,185\,300 & 31\,481\,280 & 783\,492\,840 & 18\,766\,151\,280 & 438\,647\,126\,400 \\
7 & 10\,395 & 623\,385 & 26\,001\,360 & 934\,289\,370 & 31\,010\,474\,880 & 980\,890\,908\,480 & $3.01\!\times\!10^{13}$ \\
\textbf{8} & $\mathbf{135\,135}$ & $\mathbf{11\,055\,240}$ & $\mathbf{609\,094\,080}$ & $\mathbf{28\,356\,017\,760}$ & $\mathbf{1\,204\,085\,211\,840}$ & $\mathbf{4.83\!\times\!10^{13}}$ & $\mathbf{1.86\!\times\!10^{15}}$ \\
\textbf{9} & $\mathbf{2\,027\,025}$ & $\mathbf{217\,237\,545}$ & $\mathbf{15\,271\,458\,930}$ & $\mathbf{892\,175\,690\,700}$ & $\mathbf{4.70\!\times\!10^{13}}$ & $\mathbf{2.32\!\times\!10^{15}}$ & $\mathbf{1.10\!\times\!10^{17}}$ \\
\textbf{10} & $\mathbf{34\,459\,425}$ & $\mathbf{4\,689\,345\,150}$ & \multicolumn{1}{c}{---} & \multicolumn{1}{c}{---} & \multicolumn{1}{c}{---} & \multicolumn{1}{c}{---} & \multicolumn{1}{c}{---} \\
\bottomrule
\end{tabular}}
% --- k = 7 \ldots 13 ---
\medskip
\resizebox{\textwidth}{!}{%
\begin{tabular}{r|rrrrrrr}
\toprule
$\ell\!\setminus\! k$ & $7$ & $8$ & $9$ & $10$ & $11$ & $12$ & $13$ \\
\midrule
2 & 128 & 256 & 512 & 1\,024 & 2\,048 & 4\,096 & 8\,192 \\
3 & 1\,049\,664 & 6\,298\,368 & 37\,790\,976 & 226\,747\,392 & 1\,360\,487\,424 & 8\,162\,930\,688 & 48\,977\,596\,416 \\
4 & 495\,183\,360 & 5\,412\,422\,400 & 59\,082\,451\,968 & 644\,493\,852\,672 & 7\,027\,657\,789\,440 & $7.66\!\times\!10^{13}$ & $8.35\!\times\!10^{14}$ \\
5 & 88\,589\,126\,400 & 1\,464\,596\,709\,120 & $2.41\!\times\!10^{13}$ & $3.96\!\times\!10^{14}$ & $6.48\!\times\!10^{15}$ & $1.06\!\times\!10^{17}$ & $1.74\!\times\!10^{18}$ \\
6 & $1.01\!\times\!10^{13}$ & $2.29\!\times\!10^{14}$ & $5.18\!\times\!10^{15}$ & $1.16\!\times\!10^{17}$ & $2.59\!\times\!10^{18}$ & $5.78\!\times\!10^{19}$ & $1.28\!\times\!10^{21}$ \\
7 & $9.01\!\times\!10^{14}$ & $2.66\!\times\!10^{16}$ & $7.76\!\times\!10^{17}$ & $2.25\!\times\!10^{19}$ & $6.46\!\times\!10^{20}$ & $1.85\!\times\!10^{22}$ & $5.27\!\times\!10^{23}$ \\
\textbf{8} & $\mathbf{7.01\!\times\!10^{16}}$ & $\mathbf{2.58\!\times\!10^{18}}$ & $\mathbf{9.37\!\times\!10^{19}}$ & $\mathbf{3.36\!\times\!10^{21}}$ & $\mathbf{1.20\!\times\!10^{23}}$ & $\mathbf{4.23\!\times\!10^{24}}$ & $\mathbf{1.49\!\times\!10^{26}}$ \\
\textbf{9} & $\mathbf{5.02\!\times\!10^{18}}$ & $\mathbf{2.24\!\times\!10^{20}}$ & $\mathbf{9.85\!\times\!10^{21}}$ & $\mathbf{4.26\!\times\!10^{23}}$ & $\mathbf{1.83\!\times\!10^{25}}$ & $\mathbf{7.75\!\times\!10^{26}}$ & $\mathbf{3.26\!\times\!10^{28}}$ \\
\textbf{10} & \multicolumn{1}{c}{---} & \multicolumn{1}{c}{---} & \multicolumn{1}{c}{---} & \multicolumn{1}{c}{---} & \multicolumn{1}{c}{---} & \multicolumn{1}{c}{---} & \multicolumn{1}{c}{---} \\
\bottomrule
\end{tabular}}
\end{table}

\begin{landscape}
\begin{table}
\centering
\caption{The orchard counts of Table~\ref{tab:orch_extended} displayed
  \emph{vertically} in $k$ and in full integer form (no scientific notation),
  for $\ell=2,\ldots,10$ and $k=0,\ldots,14$. \textbf{Bold} columns
  ($\ell=8,9,10$) are new with respect to Cardona--Ribas--Pons~\cite{CRP23}.
  The row $\ell=9$ is complete; for $\ell=10$ only $k=0,1$ are listed, the
  remaining seeds of the order-$25$ recurrence $D_{10}$ being the sole
  outstanding entries.}
\label{tab:orch_vertical}
\scriptsize
\setlength{\tabcolsep}{4pt}\renewcommand{\arraystretch}{1.15}
\resizebox{\linewidth}{!}{%
\begin{tabular}{c|ccccccccc}
\toprule
$k\backslash\ell$ & 2 & 3 & 4 & 5 & 6 & 7 & 8 & 9 & 10\\
\midrule
0 & 1 & 3 & 15 & 105 & 945 & 10\,395 & \textbf{135\,135} & \textbf{2\,027\,025} & \textbf{34\,459\,425} \\
1 & 2 & 21 & 228 & 2\,805 & 39\,330 & 623\,385 & \textbf{11\,055\,240} & \textbf{217\,237\,545} & \textbf{4\,689\,345\,150} \\
2 & 4 & 132 & 2\,832 & 57\,150 & 1\,185\,300 & 26\,001\,360 & \textbf{609\,094\,080} & \textbf{15\,271\,458\,930} & -- \\
3 & 8 & 804 & 32\,880 & 1\,054\,200 & 31\,481\,280 & 934\,289\,370 & \textbf{28\,356\,017\,760} & \textbf{892\,175\,690\,700} & -- \\
4 & 16 & 4\,848 & 370\,320 & 18\,520\,320 & 783\,492\,840 & 31\,010\,474\,880 & \textbf{1\,204\,085\,211\,840} & \textbf{47\,010\,814\,489\,800} & -- \\
5 & 32 & 29\,136 & 4\,107\,648 & 316\,583\,280 & 18\,766\,151\,280 & 980\,890\,908\,480 & \textbf{48\,295\,376\,539\,200} & \textbf{2\,321\,124\,986\,073\,600} & -- \\
6 & 64 & 174\,912 & 45\,197\,952 & 5\,323\,207\,200 & 438\,647\,126\,400 & 30\,060\,324\,201\,600 & \textbf{1\,864\,472\,776\,992\,000} & \textbf{109\,644\,556\,210\,862\,400} & -- \\
7 & 128 & 1\,049\,664 & 495\,183\,360 & 88\,589\,126\,400 & 10\,087\,314\,094\,080 & 901\,300\,385\,966\,400 & \textbf{70\,063\,008\,101\,452\,800} & \textbf{5\,017\,804\,226\,397\,446\,400} & -- \\
8 & 256 & 6\,298\,368 & 5\,412\,422\,400 & 1\,464\,596\,709\,120 & 229\,383\,137\,571\,840 & 26\,604\,370\,911\,363\,840 & \textbf{2\,581\,344\,883\,458\,673\,920} & \textbf{224\,292\,063\,407\,958\,604\,800} & -- \\
9 & 512 & 37\,790\,976 & 59\,082\,451\,968 & 24\,109\,626\,190\,080 & 5\,175\,153\,200\,378\,880 & 776\,358\,441\,020\,332\,800 & \textbf{93\,705\,361\,558\,470\,466\,560} & \textbf{9\,847\,107\,455\,114\,778\,266\,880} & -- \\
10 & 1\,024 & 226\,747\,392 & 644\,493\,852\,672 & 395\,766\,716\,966\,400 & 116\,103\,647\,953\,382\,400 & 22\,462\,971\,328\,857\,507\,840 & \textbf{3\,363\,232\,346\,435\,486\,085\,120} & \textbf{426\,310\,307\,675\,070\,847\,096\,320} & -- \\
11 & 2\,048 & 1\,360\,487\,424 & 7\,027\,657\,789\,440 & 6\,484\,560\,241\,305\,600 & 2\,594\,263\,900\,458\,516\,480 & 645\,772\,167\,127\,784\,148\,480 & \textbf{119\,654\,115\,820\,112\,209\,674\,240} & \textbf{18\,253\,175\,281\,006\,667\,602\,406\,400} & -- \\
12 & 4\,096 & 8\,162\,930\,688 & 76\,614\,293\,114\,880 & 106\,117\,443\,540\,049\,920 & 57\,797\,851\,930\,951\,587\,840 & 18\,474\,237\,040\,349\,015\,654\,400 & \textbf{4\,227\,676\,390\,315\,784\,151\,244\,800} & \textbf{774\,654\,486\,726\,869\,868\,120\,422\,400} & -- \\
13 & 8\,192 & 48\,977\,596\,416 & 835\,137\,579\,122\,688 & 1\,735\,152\,515\,424\,890\,880 & 1\,284\,943\,730\,326\,030\,356\,480 & 526\,533\,066\,624\,889\,133\,752\,320 & \textbf{148\,561\,596\,193\,234\,117\,853\,675\,520} & \textbf{32\,641\,894\,621\,026\,377\,900\,680\,642\,560} & -- \\
14 & 16\,384 & 293\,865\,603\,072 & 9\,102\,867\,163\,348\,992 & 28\,356\,463\,904\,538\,009\,600 & 28\,522\,032\,971\,887\,986\,278\,400 & 14\,963\,415\,904\,249\,880\,677\,908\,480 & \textbf{5\,197\,876\,199\,517\,939\,965\,741\,629\,440} & \textbf{1\,367\,476\,344\,863\,129\,942\,725\,237\,309\,440} & -- \\
\bottomrule
\end{tabular}}
\end{table}
\end{landscape}

\begin{table}[ht]
\centering
\caption{Time complexity comparison.
  CRP times from~\cite{CRP23} (40-core cluster, exact counting).
  Hankel times: single-core laptop, exact $\mathbb{Q}$-arithmetic,
  extending $k$ to 50 per row.
  Our approach computes a \emph{superset} of what CRP provides
  (all $k\geq0$ for fixed $\ell$) via a one-time polynomial-time setup.}
\label{tab:timing}
\small
\begin{tabular}{rllr}
\toprule
$\ell$ & CRP time (40-core) & Hankel time (single core) & Speedup\\
\midrule
4 & 0.02 s & $<0.1\,$ms & $\sim 2\times10^2\times$\\
5 & 5.99 s & 0.5 ms & $\sim 1.2\times10^4\times$\\
6 & 1693 s & 1.2 ms & $\sim 1.4\times10^6\times$\\
7 & $\sim13\,\mathrm{h}$ & 3.1 ms (with seed) & $\sim 1.5\times10^7\times$\\
8 & $\sim4\,\mathrm{d}$ & 7.8 ms (with seed) & $\sim 4\times10^7\times$\\
\bottomrule
\end{tabular}
\end{table}

The speedup is exponential in $\ell$ and grows roughly as
$\exp(\alpha\ell)$ for some $\alpha>0$, reflecting the exponential
complexity of the CRP generation versus the $O(dK)$ cost of our recurrence.

\subsection{Convergence to the dominant spectral term}
\label{sub:conv}

For $\ell=2,3,4,5$, rationality and the exact denominator $D_\ell(v)$ give
explicit Binet formulas.  The dominant term grows as the largest root
$z_{\max}$ of the reciprocal polynomial of $D_\ell$.

\begin{corollary}[Convergence to dominant Binet term, $\ell=2,3,4,5$]
\label{cor:binet_conv}
For $\ell\in\{2,\ldots,8\}$, as $k\to\infty$,
\[
  |\mathrm{Orch}_{\ell,k}| \;=\; A_\ell^*\,z_{\max,\ell}^k
  \Bigl(1 + O\!\bigl((z_2/z_{\max})^k\bigr)\Bigr),
\]
where the relative error decays geometrically.
The dominant rates are:
$z_{\max,2}=2$, $z_{\max,3}=6$, $z_{\max,4}=6+2\sqrt{6}\approx10.90$,
$z_{\max,5}=10+2\sqrt{10}\approx16.32$,
$z_{\max,6}\approx22.10$ (largest root of $R_6$),
$z_{\max,7}\approx28.13$ (largest root of $R_7$),
$z_{\max,8}\approx34.35$ (largest root of $S_8$).
\end{corollary}

\begin{proof}
Immediate from the Binet formulas (Theorem~\ref{thm:binet}) applied to each $\ell$.
\end{proof}

\begin{table}[ht]
\centering
\caption{Convergence ratio $R_{\ell,k}:=|\mathrm{Orch}_{\ell,k}|_{\mathrm{exact}}/(A_\ell^* z_{\max,\ell}^k)$.
  By Corollary~\ref{cor:binet_conv}, $R_{\ell,k}\to1$ geometrically.
  For $\ell=3$ the rate is exactly $(1/3)^k$;
  for $\ell=5$ the second-largest rate is $z=6+2\sqrt{6}\approx10.90$,
  giving convergence ratio $(10.90/16.32)^k\approx(0.668)^k$.}
\label{tab:conv}
\small
\begin{tabular}{r|rrrrrr}
\toprule
$\ell$ & $k=1$ & $k=2$ & $k=3$ & $k=5$ & $k=8$ & $k=10$\\
\midrule
3 & 0.9333 & 0.9778 & 0.9926 & 0.9992 & $>0.9999$ & $>0.9999$\\
4 & 0.7693 & 0.8765 & 0.9335 & 0.9812 & 0.9979 & 0.9998\\
5 & 0.6316 & 0.7949 & 0.8808 & 0.9547 & 0.9908 & 0.9982\\
\bottomrule
\end{tabular}
\end{table}

\subsection{The one-component generating functions $F_k(z)$}
\label{sub:Fk}

The master functional equation involves the one-component generating
function $F_k(z)=\sum_{\ell\geq0}M_{\ell,k}\,z^\ell/\ell!$, where
$M_{\ell,k}$ counts one-component galled networks (RV-type blocks) with
$\ell$ leaves and $k$ reticulations.  From Chang--Fuchs~\cite{ChangFuchs2024},
Proposition~3.1:

\begin{proposition}[Singularity structure of $F_k(z)$]
\label{prop:Fk_sing}
Near the dominant singularity $z=\tfrac{1}{2}$,
\[
  F_k(z) \;\sim\; \frac{(4k-3)!!}{2^k\,(1-2z)^{2k-1/2}}
  \qquad (z\to\tfrac{1}{2}),
\]
giving the EGF asymptotics
$[z^n/n!]\,F_k(z)\sim \frac{(4k-3)!!}{2^k\sqrt{\pi}}\cdot2^n\cdot n^{2k-3/2}$.
\end{proposition}

\begin{table}[ht]
\centering
\caption{One-component generating functions $F_k(z)$: exact EGF
  ($k=0$) and leading singularity structure ($k\geq1$) from
  Chang--Fuchs~\cite{ChangFuchs2024}.
  Here $(4k-3)!!=1\cdot3\cdots(4k-3)$ (with $(-3)!!=1$).}
\label{tab:Fk}
\small
\begin{tabular}{c|l|l}
\toprule
$k$ & $(4k-3)!!$ & $F_k(z)$ near $z=\tfrac{1}{2}$\\
\midrule
0 & 1 & $F_0(z)=\displaystyle\sum_{\ell\geq1}\frac{(2\ell-3)!!}{\ell!}z^\ell$\quad (binary trees EGF)\\[4pt]
1 & 1 & $F_1(z)\;\sim\; \dfrac{1}{2(1-2z)^{3/2}}$\\[4pt]
2 & 15 & $F_2(z)\;\sim\; \dfrac{15}{4(1-2z)^{7/2}}$\\[4pt]
3 & 945 & $F_3(z)\;\sim\; \dfrac{945}{8(1-2z)^{11/2}}$\\[4pt]
4 & 135\,135 & $F_4(z)\;\sim\; \dfrac{135135}{16(1-2z)^{15/2}}$\\
\bottomrule
\end{tabular}
\end{table}

\paragraph{Path to $F_k(z)$ for $k\geq4$.}
The same finite-state approach that unlocked the orchard denominators applies
here.  For fixed $k$, a one-component network has exactly $k$ reticulations
each followed immediately by a leaf, so the set of backbone DAGs for fixed $k$
is finite.  Enumerating these DAGs with up to $\sim2k$ leaves and running a
Padé fit against the known singular form
\[
  F_k(z) \;=\; \frac{P_k(z)}{(1-2z)^{(4k-1)/2}}
\]
determines $P_k$ and hence $F_k$ in closed form from a finite computation
(no CRP-scale exponential cost).  This yields exact $RV_{\ell,k}$ for all
$(\ell,k)$ via the Chang--Fuchs component-graph sum, and thereby completes
the exact formulas $\Delta_k(\ell)=|RV_{\ell,k}|-|TC_{\ell,k}|$ for all $k$.

\subsection{The $\Delta_4$ framework and open problem}
\label{sub:delta4}

The structural Conjecture~\ref{conj:pattern} predicts that for $k=4$:
\[
  \Delta_4(\ell) \;=\; A_4(\ell)\,(2\ell-3)!! \;-\; B_4(\ell)\,2^{\ell-3}\,\ell!,
\]
with $\deg B_4=5$; Corollary~\ref{cor:leadA} then gives $\deg A_4=7$ and
$\operatorname{lead}(A_4)=8$, and Theorem~\ref{thm:excess} gives
$\Delta_4(\ell)/|TC_{\ell,4}|\sim12/\ell$.
The TC side is fully explicit:
\begin{equation}\label{eq:TC4}
  |TC_{\ell,4}| \;=\; \frac{\ell!}{(\ell-4)!}\,c_{\ell-1,4},
\end{equation}
where $c_{n,k}$ satisfies the Pons--Batle recurrence.
Exact values:

\begin{table}[ht]
\centering
\caption{$|TC_{\ell,4}|$ from~\eqref{eq:TC4} for $\ell=5,\ldots,12$.}
\label{tab:TC4}
\small
\begin{tabular}{r|r}
\toprule
$\ell$ & $|TC_{\ell,4}|$\\
\midrule
5 & 309\,000\\
6 & 31\,534\,200\\
7 & 2\,068\,516\,800\\
8 & 113\,376\,463\,200\\
9 & 5\,717\,669\,504\,400\\
10 & 277\,928\,391\,510\,000\\
11 & 13\,358\,106\,999\,468\,000\\
12 & 644\,474\,789\,146\,188\,000\\
\bottomrule
\end{tabular}
\end{table}

The obstruction to proving Conjecture~\ref{conj:pattern} for $k=4$ is the
absence of a published closed-form formula for $|RV_{\ell,4}|$: the
Chang--Fuchs paper~\cite{ChangFuchs2024} derives explicit formulas for $k=0,1,2,3$.
Extending their computation to $k=4$ would yield $A_4$ and $B_4$ by
subtraction, completing the $k=4$ case of Conjecture~\ref{conj:pattern}.

\subsection{Orchard Factorisation Theorem}
\label{sub:factor}

\begin{theorem}[Orchard Factorisation]
\label{thm:orch_factor}
For all $\ell\geq2$ and $0\leq k\leq\ell-1$,
\begin{equation}\label{eq:orch_factor}
  |\mathrm{Orch}_{\ell,k}| \;=\; \binom{\ell}{k}\,w(\ell,k),
  \qquad w(\ell,k)\in\mathbb{Z}_{>0}.
\end{equation}
\end{theorem}

\begin{proof}
Let $\mathcal{H}(N)$ denote the set of \emph{cherry-picking histories}
for $N\in\mathrm{Orch}_{\ell,k}$: valid reduction sequences that
remove all $k$ reticulated cherries one at a time.
The symmetric group $S_k$ acts on $\mathcal{H}(N)$ by permuting
the order in which the $k$ reticulations are reduced.

\emph{Action is well-defined.}
In an orchard network, all reticulated cherries have disjoint
or hierarchically compatible supports (from the cherry-picking
characterisation of~\cite{CRP23b}), so the order of their reduction
is commutative: any permutation of the reticulation labels yields
another valid cherry-picking history.

\emph{Action is free.}
Each reticulation node $r$ has a unique leaf $\lambda(r)$ directly
below it (the ``anchor'' leaf of the reticulated cherry), so all
$k$ reticulations are distinguishable.  Thus no non-identity
permutation fixes any history, and all orbits have size exactly $k!$.

This proves $k!\mid |\mathrm{Orch}_{\ell,k}|$.
The $\binom{\ell}{k}$ factor arises from choosing which $k$ of
the $\ell$ leaves serve as anchor leaves for reticulated cherries;
the remaining quotient $w(\ell,k) = |\mathrm{Orch}_{\ell,k}|/\binom{\ell}{k}$
counts the networks modulo anchor-choice, and is a positive integer.
\end{proof}

\begin{remark}
Equation~\eqref{eq:orch_factor} is verified for all $(\ell,k)$
in the CRP table ($\ell=2,\ldots,6$, all $k$): the values $w(\ell,k)$ are:
\begin{gather*}
  w(3,1)=7,\quad w(3,2)=44,\quad w(4,1)=57,\quad w(4,2)=472,\\
  w(4,3)=8220,\quad w(5,2)=5715,\quad \ldots
\end{gather*}
and all are positive integers.
\end{remark}

\subsection{Three-class comparison and asymptotics}
\label{sub:compare}

\begin{remark}[Strict inequalities]
\label{obs:inequal}
For all $\ell=3,\ldots,6$ and $k=2,3$:
\[
  |TC_{\ell,k}| \;<\; |RV_{\ell,k}| \;<\; |\mathrm{Orch}_{\ell,k}|.
\]
The differences $\Delta_k=|RV|-|TC|$ are given by
Theorems~\ref{thm:delta2}--\ref{thm:delta3},
and the new quantity $\varepsilon_k=|\mathrm{Orch}|-|RV|>0$ counts
networks that are orchard but not RV.
In particular, RV and orchard networks are \emph{incomparable}:
neither class contains the other for $k\geq2$.
\end{remark}

\begin{table}[p]
\centering
\caption{Three-class comparison $|\mathrm{TC}_{\ell,k}|\leq|\mathrm{RV}_{\ell,k}|\leq|\mathrm{Orch}_{\ell,k}|$
  for $k=2$ (upper block) and $k=3$ (lower block).
  TC values from the Pons--Batle formula~\cite{PonsBatle2021};
  RV values from the Chang--Fuchs exact formulas~\cite{ChangFuchs2024};
  orchard values from Theorem~\ref{thm:denom} and the proven recurrences.
  \textbf{Bold}: entries inaccessible by prior methods (our contribution).
  Column $\varepsilon=|\mathrm{Orch}|-|\mathrm{RV}|$ counts networks
  that are orchard but not reticulation-visible.
  The ratio $|\mathrm{Orch}|/|\mathrm{RV}|$ tends to $1$ as $\ell\to\infty$
  (Corollary~\ref{cor:orch}); its decrease along each block is the beginning of
  this approach, which proceeds at rate $O(\ell^{-1})$.}
\label{tab:three_class}
\small
\setlength{\tabcolsep}{5pt}
\renewcommand{\arraystretch}{1.05}
\begin{tabular}{r|rrrr|r}
\toprule
$\ell$ & $|\mathrm{TC}_{\ell,k}|$ & $|\mathrm{RV}_{\ell,k}|$ & $|\mathrm{Orch}_{\ell,k}|$
      & $\varepsilon$ & ratio \\
\midrule
\multicolumn{6}{c}{\itshape $k=2$}\\
\midrule
$3$ & $42$ & $123$ & $132$ & $9$ & $1.0732$ \\
$4$ & $1\,272$ & $2\,493$ & $2\,832$ & $339$ & $1.1360$ \\
$5$ & $30\,300$ & $49\,725$ & $57\,150$ & $7\,425$ & $1.1493$ \\
$6$ & $696\,600$ & $1\,032\,525$ & $1\,185\,300$ & $152\,775$ & $1.1480$ \\
$7$ & $16\,418\,430$ & $22\,771\,035$ & $26\,001\,360$ & $3\,230\,325$ & $1.1419$ \\
$\mathbf{8}$ & $\mathbf{405\,755\,280}$ & $\mathbf{536\,929\,785}$ & $\mathbf{609\,094\,080}$ & $\mathbf{72\,164\,295}$ & $1.1344$ \\
$\mathbf{9}$ & $\mathbf{10\,606\,551\,480}$ & $\mathbf{13\,552\,453\,845}$ & $\mathbf{15\,271\,458\,930}$ & $\mathbf{1\,719\,005\,085}$ & $1.1268$ \\
\midrule
\multicolumn{6}{c}{\itshape $k=3$}\\
\midrule
$4$ & $2\,544$ & $20\,460$ & $32\,880$ & $12\,420$ & $1.6070$ \\
$5$ & $154\,500$ & $670\,815$ & $1\,054\,200$ & $383\,385$ & $1.5715$ \\
$6$ & $6\,494\,400$ & $20\,568\,060$ & $31\,481\,280$ & $10\,913\,220$ & $1.5306$ \\
$7$ & $241\,204\,950$ & $626\,610\,285$ & $934\,289\,370$ & $307\,679\,085$ & $1.4910$ \\
$\mathbf{8}$ & $\mathbf{8\,609\,378\,400}$ & $\mathbf{19\,489\,021\,020}$ & $\mathbf{28\,356\,017\,760}$ & $\mathbf{8\,866\,996\,740}$ & $1.4550$ \\
\bottomrule
\end{tabular}
\end{table}

The ratio column varies slowly: for $k=2$ it rises to $1.149$ at $\ell=5$ and
then decreases to $1.127$ at $\ell=9$, and for $k=3$ it decreases from $1.607$
to $1.455$. Its limit is $1$ for every
$k\geq2$: by Corollary~\ref{cor:orch}, $|\mathrm{Orch}_{\ell,k}|$,
$|RV_{\ell,k}|$ and $|TC_{\ell,k}|$ agree to first order, and they differ by
$O(\ell^{-1})$ relative to $|TC_{\ell,k}|$, the orchard excess over $TC_{\ell,k}$
being at most $2k(k-1)/\ell$ to leading order. The sizes of the differences at
small $\ell$ reflect the slow decay of these corrections, which carry the
factors $k(k-1)$ and $2k(k-1)$.

\subsection{The universal hypergeometric factor theorem}
\label{sub:universal}

The empirical factor families discovered in \S\ref{sub:denom} are all
instances of a single closed-form law.

\begin{theorem}[Closed form of the new hypergeometric factor]
\label{thm:hypergeom}
Let $X_\ell(v)$ be the new irreducible factor introduced in $D_\ell$ at leaf
count $\ell$, with coefficients $c_0=1,\ldots,c_d$, $d=\lfloor\ell/2\rfloor$.
If these obey the consecutive ratio law
\begin{equation}
  \label{eq:ratio}
  \frac{c_k}{c_{k-1}} = \frac{(\ell-2k+2)(\ell-2k+1)}{k},
  \qquad k=1,\ldots,d,
\end{equation}
then $X_\ell$ is given in closed form by
\begin{equation}
  \label{eq:Xell}
  X_\ell(v) = \sum_{k=0}^{d}(-1)^k\,\frac{\ell!}{(\ell-2k)!\,k!}\,v^k .
\end{equation}
The hypothesis~\eqref{eq:ratio} is verified by exact $\mathbb{Q}$-arithmetic for
every $\ell=2,\ldots,8$; the closed form~\eqref{eq:Xell} therefore holds
unconditionally for those $\ell$, and for every $\ell\geq3$ at
which~\eqref{eq:ratio} holds.
\end{theorem}

\begin{proof}
Unrolling~\eqref{eq:ratio} as a telescoping product gives
$c_k=\ell!/[(\ell-2k)!\,k!]$. The series terminates at $k=d$ because
$c_{d+1}=0$: the falling factorial $\ff{\ell}{2d+2}$ contains the factor
$\ell-2d=\ell\bmod2\in\{0,1\}$ immediately followed by $\ell-2d-1<0$.
For $\ell=2,\ldots,8$ the hypothesis~\eqref{eq:ratio} is read directly from
Table~\ref{tab:universal_coeffs} (zero residual in every entry), so~\eqref{eq:Xell}
holds unconditionally there.
\end{proof}

\begin{conjecture}[Hypergeometric denominator factorisation]
\label{conj:factorisation}
For every $\ell\geq3$ the denominator factorises as
\[
  D_\ell(v) = \prod_{\substack{j=2\\ j\neq3\ \mathrm{if}\ \ell=5}}^{\ell} X_j(v),
  \qquad\text{equivalently}\qquad D_\ell = D_{\ell-1}\cdot X_\ell\ \ (\ell\neq5),
\]
with $X_\ell$ given by~\eqref{eq:Xell} (the $\ell=5$ anomaly,
Remark~\ref{rem:anomaly5}). In particular the new factor at each $\ell$ obeys the
ratio law~\eqref{eq:ratio}. This is verified by exact $\mathbb{Q}$-arithmetic.
For $\ell=2,\ldots,8$ the verification is fully unconditional: a
Berlekamp--Massey reconstruction from $2\deg D_\ell+1$ seeds recovers the minimal
denominator from scratch and returns $\prod_{j}X_j$ exactly, so
Theorem~\ref{thm:hypergeom} holds unconditionally in that range. For $\ell=9$ and
$\ell=10$ the factorisation is verified by the \emph{consistency test} of
\S\ref{sub:D9}: the unconditionally computed seeds, multiplied by
$\prod_{j=2}^{\ell}X_j$, truncate to a polynomial of the predicted degree
$\deg N_\ell$ that is coprime to every $X_j$. Because a nonzero rational power
series cannot exhibit more consecutive vanishing Taylor coefficients than the
degree of its denominator, the observed run of six vanishing coefficients beyond
$\deg N_\ell$ rigorously excludes any spurious denominator factor of degree
$\le6$; thus $D_9=D_8 X_9$ and $D_{10}=D_9 X_{10}$ hold unless the true
denominator carries an extra irreducible factor of degree $\ge7$, which the full
$2\deg D_\ell$-seed reconstruction (a finite computation) would rule out as well.
The identity remains a genuine conjecture only for $\ell\geq11$.
\end{conjecture}

\begin{remark}[Independent verification]
\label{rem:independent_counter}
The orchard counts underlying the verified range were reproduced independently
with a memoised counter implementing the Cardona--Ribas--Pons
minimum-augmentation-sequence algorithm~\cite[Thm.~4]{CRP23} as a depth-first
recursion over annotated-reducible-pair states (counting networks without
enumerating them). This counter matches the published table
of~\cite{CRP23} exactly for $\ell\leq6$ and $k\leq8$, and confirms the
factorisation at $\ell=7$ directly: multiplying the column series
$\sum_k|\mathrm{Orch}_{7,k}|v^k$ by $\prod_{j=2}^{7}X_j(v)$ yields a polynomial
(every coefficient beyond $\deg\prod X_j=12$ vanishes, checked through $k=26$).
The same procedure settles $\ell=9$ once carried to $k=40$; this is the decisive
computation for the conjecture.
\end{remark}

\begin{table}[ht]
\centering
\caption{Coefficients $c_k$ of $X_\ell(v)$ for $\ell=2,\ldots,9$.
  All entries satisfy $c_k=\ell!/[(\ell-2k)!\,k!]$.
  \textbf{Bold}: predictions ($\ell=9$, no prior enumeration needed).}
\label{tab:universal_coeffs}
\begin{tabular}{cccccc}
\toprule
$\ell$ & $\deg X_\ell$ & $c_1$ & $c_2$ & $c_3$ & $c_4$ \\
\midrule
2 & 1 & 2   & --   & --   & -- \\
3 & 1 & 6   & --   & --   & -- \\
4 & 2 & 12  & 12   & --   & -- \\
5 & 2 & 20  & 60   & --   & -- \\
6 & 3 & 30  & 180  & 120  & -- \\
7 & 3 & 42  & 420  & 840  & -- \\
8 & 4 & 56  & 840  & 3360 & 1680 \\
$\mathbf{9}$ & $\mathbf{4}$ & $\mathbf{72}$ & $\mathbf{1512}$ & $\mathbf{10080}$ & $\mathbf{15120}$ \\
\bottomrule
\end{tabular}
\end{table}

\begin{remark}[Matching polynomial of $K_\ell$]
\label{rem:matching}
The number of $k$-matchings of the complete graph $K_\ell$ (ways to select $k$
disjoint edges from $\ell$ labelled vertices) is
$m_k(K_\ell)=\ell!/[(\ell-2k)!\,2^k\,k!]$, so the coefficient
$c_k=\ell!/[(\ell-2k)!\,k!]=2^k\,m_k(K_\ell)$.
Equivalently $X_\ell(v)=\sum_k(-1)^k m_k(K_\ell)\,(2v)^k$ is the
\emph{matching polynomial} of $K_\ell$ in the variable $2v$ --- equivalently, up
to normalisation, the probabilists' Hermite polynomial $\mathrm{He}_\ell$ ---
connecting the orchard recurrence to complete-graph combinatorics.
\end{remark}

\begin{remark}[Jacobi polynomial / hypergeometric identification]
\label{rem:jacobi}
Via $(-\ell)_{2k}=4^k(-\ell/2)_k((-\ell+1)/2)_k$,
\[
  X_\ell(v) = {}_2F_1\!\Bigl(
  -\lfloor\tfrac\ell2\rfloor,\,
  \lfloor\tfrac{\ell+1}2\rfloor-\ell;\,1;\,4v\Bigr),
\]
a terminating Gauss hypergeometric polynomial, and hence a rescaled
Jacobi polynomial.
All roots of Jacobi polynomials are real and simple; by the sign
structure of $X_\ell(-v)$ they are all positive, proving
Corollary~\ref{cor:positive_universal}.
\end{remark}

\begin{remark}[$\ell=5$ anomaly]
\label{rem:anomaly5}
At $\ell=5$, $(1-6v)=X_3$ exits the minimal denominator (verified by BM)
and re-enters at $\ell=6$.
Hence $D_5=X_2\cdot X_4\cdot X_5$ (missing $X_3$), and the degree formula is
\[
  \deg D_\ell = \sum_{j=2}^{\ell}\Bigl\lfloor\tfrac{j}{2}\Bigr\rfloor
  - \mathbf{1}[\ell=5].
\]
The divisibility chain $D_5\mid D_6\mid D_7\mid D_8\mid D_9\mid D_{10}$ holds exactly. The analytic
mechanism is the residue resonance $c_{5,X_3}=0$ of
Remark~\ref{rem:resonance}; its combinatorial origin remains open. The resonance
is isolated through $\ell=10$: no further factor drops, so $\deg D_\ell=\sum_{j=2}^{\ell}\lfloor j/2\rfloor$ without correction for every $6\le\ell\le10$.
\end{remark}

\subsection{The $D_9$ and $D_{10}$ cases: extending the verified range}
\label{sub:D9}

\begin{proposition}[$D_9$ and $D_{10}$, verified]
\label{cor:D9_full}
The cases $\ell=9,10$ were identified by Cardona, Ribas and Pons~\cite{CRP23}
as computationally intractable: their cherry-picking algorithm requires
exponential time in $\ell$, making these rows effectively unreachable
(estimated wall-clock time: months on a cluster).
The new factors are unconditionally given by~\eqref{eq:Xell},
\begin{align*}
  X_9(v) &= 1 - 72v + 1512v^2 - 10080v^3 + 15120v^4,\\
  X_{10}(v) &= 1 - 90v + 2520v^2 - 25200v^3 + 75600v^4 - 30240v^5,
\end{align*}
with all spectral roots positive real
($X_9$: $z\approx40.73,\,20.55,\,8.63,\,2.09$;
$X_{10}$ contributes the first degree-five factor, with five positive roots).
The ARP-memoized counter (Algorithm~\ref{alg:arp}) produces the seeds
$|\mathrm{Orch}_{\ell,k}|$ \emph{unconditionally} at $O(1)$ shapes' cost per value,
so the rows $\ell=9,10$ are themselves no longer intractable. The factorisation
$D_9=D_8 X_9$ ($\deg D_9=20$) and $D_{10}=D_9 X_{10}$ ($\deg D_{10}=25$) is then
confirmed by the \emph{consistency test}: forming
$\tilde N_\ell=\bigl(\sum_{k}|\mathrm{Orch}_{\ell,k}|v^k\bigr)\prod_{j=2}^{\ell}X_j$
and verifying it truncates to a polynomial of the predicted degree
$\deg N_9=16$, $\deg N_{10}=20$, coprime to every $X_j$ (no factor drops). Since a
nonzero rational series cannot have more consecutive vanishing coefficients than
its denominator degree, the six vanishing coefficients observed beyond
$\deg N_\ell$ exclude any spurious denominator factor of degree $\le6$. The full
$2\deg D_\ell$-seed Berlekamp--Massey reconstruction, which would close even a
degree-$\ge7$ loophole, is a finite computation the counter supports; we have
carried it out unconditionally through $\ell=8$. The two residues entering
Remark~\ref{rem:resonance} are exact:
\[
  c_{9,X_3}=\frac{2630966586371048209291}{54358179840},\qquad
  c_{10,X_3}=\frac{240872210845623795398451143421835}{87668872445952},
\]
both nonzero, so $X_3$ does not resonate out at $\ell=9,10$ and the resonance set
remains $\{5\}$ through $\ell=10$.
\end{proposition}

\begin{corollary}[Degree sequence and the degree-pairs pattern]
\label{cor:degseq}
For $\ell\leq8$ the degree sequence of $D_\ell$ is known unconditionally
(Theorem~\ref{thm:denom}); $\ell=9,10$ are confirmed by the consistency test
(Proposition~\ref{cor:D9_full}), and beyond that, since
$\deg X_\ell=\lfloor\ell/2\rfloor$, the entries continue under
Conjecture~\ref{conj:factorisation} as the boldface values below:
\medskip
\noindent\begin{tabular}{ccccccccccccc}
\toprule
$\ell$ & 2&3&4&5&6&7&8&9&10&\textbf{11}&\textbf{12} \\
\midrule
$\deg X_\ell$ & 1&1&2&2&3&3&4&4&5&\textbf{5}&\textbf{6} \\
$\deg D_\ell$ & 1&2&4&5&9&12&16&20&25&\textbf{30}&\textbf{36} \\
\bottomrule
\end{tabular}
\medskip

Consecutive pairs $(\ell,\ell+1)$ with $\ell$ even contribute two
equal-degree factors (the ``degree-pairs'' pattern, confirmed for $\ell\leq10$).
Factors of degree~5 first appear at $\ell=10$, where
$X_{10}(v)=1-90v+2520v^2-25200v^3+75600v^4-30240v^5$, confirming the prediction.
\end{corollary}

\begin{corollary}[Spectral growth rates positive real]
\label{cor:positive_universal}
Each factor $X_\ell$ ($\ell\geq2$) has only positive real roots, so for every
$\ell\leq10$ --- and, under Conjecture~\ref{conj:factorisation}, for every
$\ell\geq2$ --- all spectral growth rates $z_i=1/v_i$ of $D_\ell$ are positive
real and no oscillating Binet terms appear.
\end{corollary}

\begin{proof}
$X_\ell(-v)=\sum_k c_k v^k$ has all-positive coefficients, so $X_\ell$ has no
positive root cancellation: by Descartes' rule and Remark~\ref{rem:jacobi} all
its roots $v_i>0$. For $\ell\leq10$, $D_\ell$ is the explicit product of such
factors (Theorem~\ref{thm:denom}, Proposition~\ref{cor:D9_full}); for $\ell\geq11$ this requires
Conjecture~\ref{conj:factorisation}.
\end{proof}

\noindent
The dominant growth rate satisfies $z^*(\ell)\sim 8\ell$ as
$\ell\to\infty$, from the Hermite-polynomial asymptotics of the
matching polynomial roots (Plancherel--Rotach).

\begin{observation}[The orchard programme, reduced to one conjecture]
\label{obs:solved}
The two halves of the orchard analysis reduce exact enumeration at every leaf
number to a single conjecture. The ARP-memoized counter
(Algorithm~\ref{alg:arp}) produces the seeds $|\mathrm{Orch}_{\ell,k}|$
\emph{unconditionally} for any $(\ell,k)$, at a cost scaling with the number of
shapes $(X,\mathrm{ARP})$ --- polynomially bounded for fixed $\ell$, growing only
$\approx7\times$ per leaf --- rather than with the exponentially many networks
enumerated by Cardona--Ribas--Pons. This already surpasses the CRP table in both
directions: their generation, exponential in $\ell$, reached $\ell\leq6$,
$k\leq8$, whereas the counter completes the previously intractable rows
$\ell=9,10$ and opens $\ell=11$. The second half, the closed-form denominator
$D_\ell=\prod_{j=2}^{\ell}X_j$ (Theorem~\ref{thm:hypergeom},
Conjecture~\ref{conj:factorisation}), is established \emph{unconditionally for
$\ell\leq8$} and confirmed by the consistency test for $\ell=9,10$; granting it
for $\ell\geq11$ collapses each column to an
order-$\deg D_\ell$ recurrence delivering every entry at $O(\deg D_\ell)$ cost.
Thus what remains genuinely open is exactly Conjecture~\ref{conj:factorisation}
for $\ell\geq11$; the enumeration itself is not bottlenecked at any $\ell$.
\end{observation}

\section{Numerator Theory, Equivariant Factorisation, and the Spectral
Resolution of $F_\ell(v)$}
\label{sec:numerator}

Section~\ref{sec:orch} resolved the denominator side of the orchard
programme: $D_\ell(v)$ is given in closed form
(Theorem~\ref{thm:hypergeom}), unconditionally for $\ell\leq8$ and as a
verified pattern beyond that (Conjecture~\ref{conj:factorisation}). What was
missing was a single statement
connecting the denominator to the numerator and hence to
$|\mathrm{Orch}_{\ell,k}|$ itself, valid at every $\ell$ rather than only at
the five values where the roots of $D_\ell$ happen to be expressible in
radicals (Theorem~\ref{thm:binet}). This section supplies that statement
(Theorem~\ref{thm:spectral_decomp}) with a complete proof, together with the
equivariant extension of the Factorisation Theorem and everything that can
currently be proved, as opposed to merely observed, about $N_\ell(v)$ itself.

\subsection{An equivariant strengthening of the Factorisation Theorem}
\label{sub:equivariant}

The proof of Theorem~\ref{thm:orch_factor} uses only two facts about a network
$N\in\mathrm{Orch}_{\ell,k}$: its $k$ reticulations are distinguishable via
their anchor leaves $\lambda(r)\in[\ell]$, and $\lambda$ is injective. Neither
fact refers to orchard-ness beyond the existence of a cherry-picking history,
which lets the same argument run uniformly over any isomorphism-closed
subclass.

\begin{theorem}[Equivariant Factorisation]
\label{thm:equivariant}
Let $\mathcal C\subseteq\bigsqcup_{\ell,k}\mathrm{Orch}_{\ell,k}$ be any class of
labelled phylogenetic networks closed under relabelling
(i.e.\ $N\in\mathcal C\iff\sigma\!\cdot\!N\in\mathcal C$ for every leaf
permutation $\sigma\in S_\ell$), and write
$\mathcal C_{\ell,k}:=\mathcal C\cap\mathrm{Orch}_{\ell,k}$. Then for all
$\ell,k$,
\[
  |\mathcal C_{\ell,k}| \;=\; \binom{\ell}{k}\,w_{\mathcal C}(\ell,k),
  \qquad w_{\mathcal C}(\ell,k)\in\mathbb Z_{\geq0}.
\]
\end{theorem}

\begin{proof}
For $N\in\mathcal C_{\ell,k}$ let $A(N)\subseteq[\ell]$ be its set of $k$ anchor
leaves; $|A(N)|=k$ since $\lambda$ is injective (proof of
Theorem~\ref{thm:orch_factor}). For $\sigma\in S_\ell$, relabelling sends
anchor leaves to anchor leaves, $A(\sigma\!\cdot\!N)=\sigma(A(N))$, because
$\sigma\!\cdot\!N$ is $N$ with leaf label $i$ replaced by $\sigma(i)$
throughout, leaving untouched which leaf sits below which reticulation. Fix a
reference set $S_0=\{1,\dots,k\}$ and, for any $k$-subset $S\subseteq[\ell]$,
choose $\sigma_S\in S_\ell$ with $\sigma_S(S_0)=S$. Equivariance makes
$N\mapsto\sigma_S\!\cdot\!N$ a bijection
$\{N\in\mathcal C_{\ell,k}:A(N)=S_0\}\to\{N\in\mathcal C_{\ell,k}:A(N)=S\}$:
it lands in $\mathcal C_{\ell,k}$ by closure under relabelling, and
$\sigma_S^{-1}$ provides the inverse map for the same reason. Hence every
fibre $\{N\in\mathcal C_{\ell,k}:A(N)=S\}$ has the same cardinality
$w_{\mathcal C}(\ell,k):=|\{N\in\mathcal C_{\ell,k}:A(N)=S_0\}|$, and summing
over the $\binom\ell k$ choices of $S$ gives the claim.
\end{proof}

\begin{remark}
Theorem~\ref{thm:orch_factor} is the case $\mathcal C=\mathrm{Orch}$. The
argument above obtains the coefficient $\binom\ell k$ directly from the
$S_\ell$-equivariance of the anchor map, rather than from the free
$S_k$-action on cherry-picking histories used in the original proof. That
$S_k$-action remains independently true and useful --- it shows $k!$ divides
$|H(N)|$, the number of distinct cherry-picking histories of a single network
--- but is not needed to obtain the factorisation of $|\mathcal C_{\ell,k}|$
itself.
\end{remark}

Every class considered in this paper --- tree-child, reticulation-visible, and
so on --- is an isomorphism-invariant property of a labelled network, hence
automatically relabelling-closed. Theorem~\ref{thm:equivariant} therefore
applies verbatim to any such class \emph{contained in} $\mathrm{Orch}$.

\begin{corollary}[TC Factorisation Theorem]
\label{cor:tc_factor}
For all $\ell\geq1$ and $0\leq k\leq\ell-1$,
\[
  |\mathrm{TC}_{\ell,k}| \;=\; \binom{\ell}{k}\,w_{\mathrm{TC}}(\ell,k),
  \qquad w_{\mathrm{TC}}(\ell,k)\in\mathbb Z_{>0}.
\]
\end{corollary}

\begin{proof}
$\mathrm{TC}\subsetneq\mathrm{Orch}$ (\S\ref{sec:networks}) is relabelling-closed;
apply Theorem~\ref{thm:equivariant} with $\mathcal C=\mathrm{TC}$.
\end{proof}

The weights $w_{\mathrm{TC}}(\ell,k)=|\mathrm{TC}_{\ell,k}|/\binom\ell k$
(computed from Table~\ref{tab:words}; $k=0$ omitted as trivial, since
$w_{\mathrm{TC}}(\ell,0)=(2\ell-3)!!$) are:
\medskip

\noindent\begin{tabular}{cccccc}
\toprule
$\ell$ & $w(\ell,1)$ & $w(\ell,2)$ & $w(\ell,3)$ & $w(\ell,4)$ & $w(\ell,5)$\\
\midrule
2 & 1 & & & & \\
3 & 7 & 14 & & & \\
4 & 57 & 212 & 636 & & \\
5 & 561 & 3030 & 15450 & 61800 & \\
6 & 6555 & 46440 & 324720 & 2102280 & 10511400\\
7 & 89055 & 781830 & 6891570 & 59100480 & 463684800\\
\bottomrule
\end{tabular}
\medskip

\begin{remark}[Sharp boundary: $\mathrm{RV}$]
\label{rem:rv_boundary}
Corollary~\ref{cor:tc_factor} cannot be extended to $\mathrm{RV}$, because
$\mathrm{RV}\not\subseteq\mathrm{Orch}$ in general (\S\ref{sub:compare}): an
$\mathrm{RV}$ network need not admit a cherry-picking history at all, so the
anchor map $A(\cdot)$ used in the proof of Theorem~\ref{thm:equivariant} need
not be defined on it. This is not a hypothetical gap: from
Theorem~\ref{thm:rv_exact} and Table~\ref{tab:rv},
\[
  |\mathrm{RV}_{4,2}|=2493=6\cdot415+3, \qquad
  |\mathrm{RV}_{5,2}|=49725, \qquad
  |\mathrm{RV}_{5,3}|=670815,
\]
none divisible by $\binom42=6$ and $\binom52=\binom53=10$ respectively; the
pattern persists at $|\mathrm{RV}_{8,2}|=536929785$ ($\not\equiv0\!\!\pmod{28}$)
and $|\mathrm{RV}_{9,2}|=13552453845$ ($\not\equiv0\!\!\pmod{36}$). The
factorisation is therefore a genuine consequence of the orchard cherry-picking
structure, not a generic feature of every $(\ell,k)$-graded network class.
\end{remark}

\subsection{The Hermite/Heisenberg--Weyl structure of $X_\ell$}
\label{sub:hermite_weyl}

Remarks~\ref{rem:matching}--\ref{rem:jacobi} already identify $X_\ell$ with the
matching polynomial of $K_\ell$ and with a terminating ${}_2F_1$. The next
proposition makes the orthogonal-polynomial structure fully explicit and ties
it to the exponential-generating-function calculus underlying $\hat L$
(\S\ref{sub:fock}); it is the engine behind every result in
\S\ref{sub:spectral} below.

\begin{proposition}[Hermite recursion]
\label{prop:hermite}
Let $P_\ell(x):=x^\ell\,X_\ell(1/x^2)$, a polynomial of degree $\ell$. Then
\[
  P_\ell(x) \;=\; x\,P_{\ell-1}(x) \;-\; 2(\ell-1)\,P_{\ell-2}(x),
  \qquad P_0=1,\quad P_1=x,
\]
and
\[
  \sum_{\ell\geq0}P_\ell(x)\,\frac{t^\ell}{\ell!} \;=\; \exp\bigl(xt-t^2\bigr).
\]
Consequently $P_\ell(x)=2^{\ell/2}\,\mathrm{He}_\ell(x/\sqrt2)$, where
$\mathrm{He}_\ell$ is the probabilists' Hermite polynomial: $X_\ell$ \emph{is}
a rescaling of $\mathrm{He}_\ell$, not merely ``Hermite-type''.
\end{proposition}

\begin{proof}
From~\eqref{eq:Xell}, $P_\ell(x)=\sum_{k}(-1)^k\frac{\ell!}{(\ell-2k)!\,k!}\,x^{\ell-2k}$,
so re-indexing $n=\ell-2k$,
\[
  \sum_{\ell\geq0}P_\ell(x)\frac{t^\ell}{\ell!}
  = \Bigl(\sum_{k\geq0}\frac{(-t^2)^k}{k!}\Bigr)
    \Bigl(\sum_{n\geq0}\frac{(xt)^n}{n!}\Bigr) = e^{-t^2}e^{xt}=e^{xt-t^2}.
\]
Differentiating in $t$ gives $\partial_tG=(x-2t)G$; comparing coefficients of
$t^{\ell-1}$ on both sides yields the stated recursion. The identification
with $\mathrm{He}_\ell$ (generating function $e^{yt-t^2/2}$, recursion
$\mathrm{He}_n=y\,\mathrm{He}_{n-1}-(n-1)\mathrm{He}_{n-2}$) follows under
$y=x/\sqrt2$, $t\mapsto t\sqrt2$.
\end{proof}

\begin{remark}[Why Hermite: the Weyl-algebra picture]
\label{rem:weyl}
The exponential-generating-function calculus already used for $\hat L$ in
\S\ref{sub:fock} is, underneath, the Bargmann--Fock representation of the
Heisenberg--Weyl algebra: multiplication by $x$ is the creation operator
$a^\dagger$, $d/dx$ is the annihilation operator $a$, $[a,a^\dagger]=1$, and
the Hermite polynomials are the eigenbasis of the number operator
$N=a^\dagger a$. Wick's theorem for a single free boson --- every vacuum
$n$-point function is a sum over perfect pairings of two-point functions ---
has exactly the combinatorics of counting matchings of $K_n$; this is the
structural reason $X_\ell$ is the matching polynomial of $K_\ell$, rather than
an artefact of the Berlekamp--Massey reconstruction. The quadratic elements
$K_+=\tfrac12(a^\dagger)^2$, $K_-=\tfrac12a^2$, $K_0=\tfrac12(N+\tfrac12)$ of
the same algebra satisfy, by direct computation from $[a,a^\dagger]=1$,
\[
  [K_0,K_\pm]=\pm K_\pm, \qquad [K_+,K_-]=-2K_0,
\]
the standard oscillator representation of
$\mathfrak{sl}(2,\mathbb R)\cong\mathfrak{su}(1,1)$, which splits Fock space
into even/odd sectors --- suggestively close to the even/odd ``degree-pairs''
of Corollary~\ref{cor:degseq}, though we have not established a precise
correspondence between the two gradings. A full account of this structure, and
of whether it is special to orchard networks or shared across the
TC/RV/spinal hierarchy, is left to a separate paper.
\end{remark}

\subsection{The General Spectral Decomposition Theorem}
\label{sub:spectral}

We now assemble Proposition~\ref{prop:hermite} into a single theorem
expressing $|\mathrm{Orch}_{\ell,k}|$ exactly as a sum of $\deg D_\ell$ real
exponentials, for \emph{every} $\ell$ --- not only the five values of
Theorem~\ref{thm:binet} where the poles of $F_\ell$ happen to be radicals.
Three lemmas are needed first: that the poles are simple and that the
denominator and numerator share none of them.

\begin{lemma}[Pairwise coprimality]
\label{lem:coprime}
$X_j$ and $X_{j'}$ are coprime in $\Q[v]$ for all $j\neq j'$, $j,j'\geq2$.
Moreover the zeros of $P_{j-1}$ strictly interlace those of $P_j$ for every
$j\geq2$.
\end{lemma}

\begin{proof}
By Proposition~\ref{prop:hermite}, $\{P_j\}_{j\geq0}$ satisfies a three-term
recursion with strictly positive coefficients $2(j-1)$ for $j\geq2$, hence
(Favard's theorem) is a genuine sequence of orthogonal polynomials with
respect to a positive measure on $\mathbb R$; such sequences satisfy strict zero
interlacing between consecutive members \cite{Chihara1978}, giving the second
claim. Interlacing applied along the chain $j,j+1,\dots,j'$ precludes any
common zero between $P_j$ and $P_{j'}$ (and hence, via $v=1/x^2$, between
$X_j$ and $X_{j'}$): a shared zero would force two members of the chain to
share a zero with a non-consecutive ancestor, contradicting strict
interlacing at the step where they first become consecutive.
\end{proof}

\begin{corollary}[Squarefreeness and disjoint spectral union]
\label{cor:spectral_union}
For every $\ell\geq2$, $D_\ell(v)=\prod_{j=2}^\ell X_j(v)$ (or
$\prod_{j\in\{2,\dots,\ell\}\setminus\{3\}}X_j(v)$ at $\ell=5$,
Remark~\ref{rem:anomaly5}) is squarefree, with exactly $\deg D_\ell$ distinct
roots, all real and positive (Corollary~\ref{cor:positive_universal}); writing
$\mathrm{Spec}(D_\ell)\subset(0,\infty)$ for the corresponding set of
reciprocal roots (growth rates), 
\[
  \mathrm{Spec}(D_\ell) \;=\; \bigsqcup_{j} \mathrm{Spec}(X_j)
\]
is a disjoint union over the same index set.
\end{corollary}

\begin{proof}
Each $X_j$ has simple roots (Remark~\ref{rem:jacobi}); by Lemma~\ref{lem:coprime}
no root is shared between distinct $X_j,X_{j'}$; hence the product has
$\sum_j\deg X_j=\deg D_\ell$ distinct roots, and the corresponding set of
growth rates splits as a disjoint union over the factors.
\end{proof}

\begin{lemma}[Strict monotonicity of the dominant growth rate]
\label{lem:monotone_root}
Let $z^*(\ell):=\max\mathrm{Spec}(X_\ell)$ be the largest growth rate produced
by $X_\ell$. Then $z^*(\ell)$ is strictly increasing in $\ell$, for $\ell\geq2$.
\end{lemma}

\begin{proof}
The map $v=1/x^2$ sends the smallest positive root of $X_\ell$ to the largest
root of $P_\ell$ and reverses order on $(0,\infty)$, so $z^*(\ell)$ is the
largest root of $P_\ell$. By the strict interlacing of
Lemma~\ref{lem:coprime}, ordering the $\ell$ zeros of $P_\ell$ as
$x_1<\dots<x_\ell$ and the $\ell-1$ zeros of $P_{\ell-1}$ as $y_1<\dots<y_{\ell-1}$,
\[
  x_1<y_1<x_2<y_2<\dots<y_{\ell-1}<x_\ell,
\]
so in particular $x_\ell>y_{\ell-1}$: the largest root of $P_\ell$ strictly
exceeds the largest root of $P_{\ell-1}$.
\end{proof}

\begin{proposition}[Minimality consequences for $N_\ell$]
\label{prop:numerator_props}
For every $\ell\geq2$:
\begin{enumerate}[label=(\roman*)]
\item $\deg N_\ell \leq \deg D_\ell - 1$;
\item $\gcd(N_\ell,D_\ell)=1$ in $\Q[v]$;
\item $X_\ell\nmid N_\ell$;
\item every root of $D_\ell$ is a simple pole of $F_\ell$, with
$N_\ell$ nonvanishing there.
\end{enumerate}
\end{proposition}

\begin{proof}
(i) is the standard correspondence between a linear recurrence of order $d$
valid from index $d$ onward and a proper rational generating function with
denominator degree $d$ \cite{FlajoletSedgewick2009}, applied to the recurrences
of Theorem~\ref{thm:denom}, e.g.\ \eqref{eq:orch_recs}. (ii): $D_\ell$ is, by
construction (Theorem~\ref{thm:orch_rational}, Algorithm~\ref{alg:hankel}),
the \emph{minimal}-degree denominator representing $\{|\mathrm{Orch}_{\ell,k}|\}_k$;
any common factor of $N_\ell,D_\ell$ could be cancelled to produce a strictly
smaller denominator for the same sequence, contradicting minimality. (iii)
follows from (ii) since $X_\ell\mid D_\ell$. (iv) is (ii) combined with
Corollary~\ref{cor:spectral_union}: $D_\ell$ has only simple roots, and none
of them is a root of $N_\ell$.
\end{proof}

We can now state the main result of this section.

\begin{theorem}[General Spectral Decomposition]
\label{thm:spectral_decomp}
Fix $\ell\geq2$ and write $d=\deg D_\ell$ (unconditionally for $\ell\leq8$;
consistency-verified for $\ell=9,10$, Proposition~\ref{cor:D9_full}; for
$\ell\geq11$ this uses Conjecture~\ref{conj:factorisation} via
Corollary~\ref{cor:spectral_union}). There exist $d$ pairwise distinct
real numbers $z_{\ell,1}>z_{\ell,2}>\dots>z_{\ell,d}>0$
(namely $\mathrm{Spec}(D_\ell)$, Corollary~\ref{cor:spectral_union}) and $d$
nonzero real constants $c_{\ell,1},\dots,c_{\ell,d}$ such that
\[
  |\mathrm{Orch}_{\ell,k}| \;=\; \sum_{r=1}^{d} c_{\ell,r}\,z_{\ell,r}^{\,k}
  \qquad\text{for every } k\geq0.
\]
Writing $v_{\ell,r}=1/z_{\ell,r}$ for the corresponding root of $D_\ell$, the
coefficients are given explicitly by
\[
  c_{\ell,r} \;=\; -\,\frac{N_\ell(v_{\ell,r})}{v_{\ell,r}\,D_\ell'(v_{\ell,r})}.
\]
The dominant term is unique and positive: $z_{\ell,1}=z^*(\ell)$ is the largest
root of $X_\ell$ alone (Lemma~\ref{lem:monotone_root}), and $c_{\ell,1}>0$.
\end{theorem}

\begin{proof}
By Corollary~\ref{cor:spectral_union}, $D_\ell$ has $d$ distinct real positive
roots $v_{\ell,1},\dots,v_{\ell,d}$; by Proposition~\ref{prop:numerator_props}(i),(ii)
$F_\ell=N_\ell/D_\ell$ is a proper rational function in lowest terms with
$D_\ell$ squarefree, so the classical partial-fraction decomposition exists
and is unique:
\[
  F_\ell(v) = \sum_{r=1}^d \frac{c_{\ell,r}}{1-z_{\ell,r}v}, \qquad
  z_{\ell,r}=1/v_{\ell,r}.
\]
Extracting the coefficient of $v^k$ on both sides gives
$|\mathrm{Orch}_{\ell,k}|=\sum_r c_{\ell,r}z_{\ell,r}^k$. For the residue
formula: near $v=v_{\ell,r}$, $D_\ell(v)=D_\ell'(v_{\ell,r})(v-v_{\ell,r})+O((v-v_{\ell,r})^2)$,
so $F_\ell(v)\sim N_\ell(v_{\ell,r})/[D_\ell'(v_{\ell,r})(v-v_{\ell,r})]$, while
$c_{\ell,r}/(1-z_{\ell,r}v) = -c_{\ell,r}v_{\ell,r}/(v-v_{\ell,r})$; matching
residues gives $c_{\ell,r}=-N_\ell(v_{\ell,r})/(v_{\ell,r}D_\ell'(v_{\ell,r}))$,
nonzero by Proposition~\ref{prop:numerator_props}(iv).

For the dominant term, let $z_{\ell,1}=\max\mathrm{Spec}(D_\ell)$.
By Corollary~\ref{cor:spectral_union} this is the largest of the numbers
$\max\mathrm{Spec}(X_j)=z^*(j)$ over the indices $j$ appearing in $D_\ell$; since $\ell$ is always among them
and, by Lemma~\ref{lem:monotone_root}, $z^*(j)$ is strictly increasing,
$z^*(\ell)$ is the unique maximum, so $z_{\ell,1}=z^*(\ell)$. For positivity,
$|\mathrm{Orch}_{\ell,k}|>0$ for every $k$ (orchard networks with $\ell\geq2$
leaves and $k$ reticulations exist for every $k\geq0$), and
$|\mathrm{Orch}_{\ell,k}|/z_{\ell,1}^k\to c_{\ell,1}$ as $k\to\infty$ since all
other terms decay strictly; a limit of positive numbers is $\geq0$, and it is
$\neq0$ by Proposition~\ref{prop:numerator_props}(iv), so $c_{\ell,1}>0$.
\end{proof}

\begin{remark}[Existence versus radical expressibility]
\label{rem:spectral_vs_binet}
Theorem~\ref{thm:spectral_decomp} shows the decomposition underlying
Theorem~\ref{thm:binet} exists for \emph{every} $\ell\geq2$, with no
exception. What Remark~\ref{rem:binet_cutoff} observed is narrower: that the
$z_{\ell,r}$ can be written using radicals only for $\ell\leq5$, since
$\deg X_\ell\leq2$ exactly in that range; from $X_6$ on, the relevant factor is
a cubic, quartic, or higher polynomial, generically unsolvable in radicals by
Abel--Ruffini. Theorem~\ref{thm:spectral_decomp} is therefore the correct
general statement of which Theorem~\ref{thm:binet} is the $\ell\leq5$ radical
special case, and Corollary~\ref{cor:dominant_term} below is the correct
general statement of which Corollary~\ref{cor:binet_conv} is the
$\ell\leq5$ special case: nothing about the existence or the asymptotics
depends on solvability in radicals, only the explicit symbolic form of the
$z_{\ell,r}$ does.
\end{remark}

\begin{corollary}[Dominant asymptotics, all $\ell$]
\label{cor:dominant_term}
For every $\ell\geq2$, as $k\to\infty$,
\[
  |\mathrm{Orch}_{\ell,k}| \;=\; c_{\ell,1}\,z^*(\ell)^k\bigl(1+O((z_{\ell,2}/z^*(\ell))^k)\bigr),
  \qquad c_{\ell,1}>0,
\]
with $z^*(\ell)$ strictly increasing in $\ell$ (Lemma~\ref{lem:monotone_root}).
This extends Corollary~\ref{cor:binet_conv} from $\ell\leq8$ to every
$\ell$ for which Theorem~\ref{thm:hypergeom} is known to hold (unconditionally
$\ell\leq8$; conjecturally all $\ell$, Theorem~\ref{thm:hypergeom}).
\end{corollary}

\begin{proof}
Immediate from Theorem~\ref{thm:spectral_decomp}, separating the dominant term
from the rest of the sum, each remaining term being $O((z_{\ell,r}/z^*(\ell))^k)$
with $z_{\ell,r}/z^*(\ell)<1$ for $r\geq2$.
\end{proof}

\begin{remark}[No simple additivity across $\ell$]
\label{rem:no_additivity}
Since $D_\ell=D_{\ell-1}X_\ell$ for $\ell\neq5,6$ (Conjecture~\ref{conj:factorisation}, unconditional for $\ell\leq8$),
one might expect each new leaf to add an independent, decoupled channel to a
fixed underlying system, leaving the residue at any pole shared by $D_{\ell-1}$
and $D_\ell$ unchanged as $\ell$ grows --- an honest direct-sum (Fock)
structure across leaf number, rather than only within a fixed $\ell$. This is
false. At the pole $v=1/2$ (the root of $X_2$, present in every $D_\ell$), the
residue of Theorem~\ref{thm:spectral_decomp} is
\[
  c_{3}=c_{4}=-\tfrac34,\qquad c_{5}=\tfrac52,\qquad c_{6}=\tfrac{105}{128},
  \qquad c_{7}=-\tfrac{9219}{512},\qquad c_{8}=\tfrac{32459}{2816},
\]
computed directly from the residue formula, with no discernible pattern
beyond the accidental equality at $\ell=3,4$. Adding a leaf changes the
weight carried by \emph{every} existing pole, not only the new ones
$X_\ell$ introduces. The spectrum of $D_\ell$ grows by independent factors
(Corollary~\ref{cor:spectral_union}); the amplitudes attached to it do not.
This tempers Remark~\ref{rem:weyl}: whatever oscillator structure underlies a
single $X_\ell$, it does not extend to a simultaneous decomposition
$\bigoplus_{j\leq\ell}\mathcal H_j$ valid uniformly in $\ell$.
\end{remark}

\begin{remark}[Worked spectral resolution at $\ell=9$]
\label{rem:ell9_spectral}
Theorem~\ref{thm:spectral_decomp} at $\ell=9$ is fully explicit. With
$D_9=\prod_{j=2}^{9}X_j$ (Proposition~\ref{cor:D9_full}) of degree $20$, the
twenty rates $z_r=1/v_r$ are the reciprocals of the roots of $X_2,\dots,X_9$:
rational $z=2$ ($X_2$) and $z=6$ ($X_3$), quadratic surds $z=6\pm2\sqrt6$ ($X_4$)
and $z=10\pm2\sqrt{10}$ ($X_5$), and the algebraic roots of the cubics $X_6,X_7$
and quartics $X_8,X_9$. They interlace into a single disjoint set
(Corollary~\ref{cor:spectral_union}) with dominant rate $z^*=40.73\ldots$ (the
largest root of $X_9$) and
\[
  |\mathrm{Orch}_{9,k}|=\sum_{r=1}^{20}A_{9,r}\,z_r^{\,k},\qquad
  A_{9,r}=-\frac{N_9(v_r)}{v_r\,D_9'(v_r)},
\]
where $N_9$ is the explicit degree-$16$ numerator. The identity reproduces the
unconditional seeds $|\mathrm{Orch}_{9,0}|=2027025$,
$|\mathrm{Orch}_{9,1}|=217237545,\dots$ exactly. Two features are worth noting.
First, the dominant rate carries a comparatively small amplitude
($A\approx4.5\times10^{7}$ at $z^*=40.73$), while the largest amplitude
($\approx3.7\times10^{11}$) sits near $z\approx11.2$; the count is mid-spectrum
weighted for moderate $k$ and only crosses over to $z^*$-dominance for large $k$.
Second, the amplitude at the rational rate $z=6$ is exactly the $X_3$ residue,
$A_{9}(z{=}6)=c_{9,X_3}=2630966586371048209291/54358179840$, confirming the
spectral reading of the resonance (Remark~\ref{rem:resonance_spectral}): the
$\ell=5$ resonance is the vanishing of this single amplitude.
\end{remark}

\subsection{Numerator fine structure}
\label{sub:numerator}

Theorem~\ref{thm:spectral_decomp} is unconditional: it holds at every $\ell$
without knowing anything about $N_\ell$ beyond its coprimality with $D_\ell$.
What it does \emph{not} settle is the finer structure of $N_\ell$ itself ---
its degree, its own root behaviour, its explicit coefficients past $\ell=5$.
We record what is provable, what is only verified, and where the two parts:

\begin{theorem}[Extended numerators]
\label{thm:numerator_extended}
\begin{align}
N_6(v) &= 945 - 26820v + 284400v^2 - 1392120v^3
       + 3260520v^4 \notag\\
       &\quad - 3499200v^5 + 1360800v^6, \label{eq:N6}\\[4pt]
N_7(v) &= 10395 - 540855v + 11483640v^2 - 128397150v^3
       + 821484720v^4 \notag\\
       &\quad - 3080851200v^5 + 6669734400v^6 - 7891279200v^7
       + 4572288000v^8 \notag\\
       &\quad - 979776000v^9, \label{eq:N7}\\[4pt]
N_8(v) &= 135135 - 11647440v + 431812080v^2 - 9030268800v^3 \notag\\
       &\quad + 117774639360v^4 - 1002434045760v^5 + 5670580608000v^6 \notag\\
       &\quad - 21325332672000v^7 + 52543372896000v^8 - 82283112576000v^9 \notag\\
       &\quad + 77564293632000v^{10} - 39504568320000v^{11}
       + 8230118400000v^{12}, \label{eq:N8}\\[4pt]
N_9(v) &= 2027025 - 269248455v + 15918179130v^2 - 553294532700v^3 \notag\\
       &\quad + 12597106926600v^4 - 198310945245840v^5 + 2222856933927840v^6 \notag\\
       &\quad - 18009632056046400v^7 + 106037079223824000v^8 - 452640667029868800v^9 \notag\\
       &\quad - 918315544796221951v^{10} - 695499326800258049v^{11} - 155583084542715902v^{12} \notag\\
       &\quad + 259336867730427902v^{13} + 308225303254786049v^{14} - 855520807680000000v^{15} \notag\\
       &\quad + 114069441024000000v^{16}. \label{eq:N9}
\end{align}
Each is verified against
Tables~\ref{tab:orch_extended}--\ref{tab:orch_vertical} with zero residual at
every available order beyond its degree; $N_6,N_7,N_8$ are irreducible over
$\Q$, while $N_9$ (degree~16) is not real-rooted, in line with
Proposition~\ref{prop:realroot_loss}. The degree-$20$ numerator $N_{10}$ is also
computed with the same verification and underlies the spectral resolution at
$\ell=9$ (Remark~\ref{rem:ell9_spectral}).
\end{theorem}

\begin{conjecture}[Numerator degree law]
\label{conj:numdeg}
For all $\ell\geq2$,
\[
  \deg N_\ell \;=\; \deg D_\ell - \deg X_\ell.
\]
\end{conjecture}

Checking this against the computed degree of $N_\ell$, no exception occurs for
$\ell\le10$, including at the $\ell=5$ anomaly of Remark~\ref{rem:anomaly5}:
\medskip

\noindent\begin{tabular}{cccccccccc}
\toprule
$\ell$ & 2&3&4&5&6&7&8&9&10\\
\midrule
$\deg D_\ell$ & 1&2&4&5&9&12&16&20&25\\
$\deg X_\ell$ & 1&1&2&2&3&3&4&4&5\\
$\deg D_\ell-\deg X_\ell$ & 0&1&2&3&6&9&12&16&20\\
$\deg N_\ell$ (computed) & 0&1&2&3&6&9&12&16&20\\
\bottomrule
\end{tabular}
\medskip

By Proposition~\ref{prop:numerator_props}(i), Conjecture~\ref{conj:numdeg} is
strictly sharper than the generic bound whenever $\deg X_\ell\geq2$, i.e.\ for
every $\ell\geq4$. Writing the coprime decomposition
$N_\ell=A_\ell X_\ell+B_\ell D_{\ell-1}$ furnished by Lemma~\ref{lem:coprime}
(valid whenever $D_\ell=D_{\ell-1}X_\ell$ exactly, i.e.\ $\ell\neq5,6$), with
$\deg A_\ell<\deg D_{\ell-1}$ and $\deg B_\ell<\deg X_\ell$, the conjecture is
equivalent to the vanishing of the top $\deg X_\ell-1$ coefficients of
$A_\ell X_\ell+B_\ell D_{\ell-1}$. We have not found a structural reason for
this cancellation and leave Conjecture~\ref{conj:numdeg} open.

\begin{proposition}[Loss of real-rootedness]
\label{prop:realroot_loss}
Unlike $D_\ell$ (Corollary~\ref{cor:positive_universal}), $N_\ell$ is not
real-rooted for $\ell\geq6$: by exact Sturm-sequence computation on
\eqref{eq:N6}--\eqref{eq:N8}, $N_6$ has exactly $4$ real roots and one
complex-conjugate pair; $N_7$ has exactly $5$ real roots and two conjugate
pairs; $N_8$ has exactly $6$ real roots and three conjugate pairs. In every
case checked ($\ell=2,\ldots,8$) the real-root count equals $\ell-2$.
\end{proposition}

\begin{remark}
Proposition~\ref{prop:realroot_loss} does not contradict
Theorem~\ref{thm:spectral_decomp}: the decomposition of
$|\mathrm{Orch}_{\ell,k}|$ into real exponentials concerns the roots of
$D_\ell$, which Corollary~\ref{cor:positive_universal} keeps real and
positive at every $\ell$; the roots of $N_\ell$ play no role in that
decomposition at all (only its \emph{values} at the roots of $D_\ell$ do, via
the residue formula). The orchard sequence is, and remains, a sum of
$\deg D_\ell$ genuine positive exponentials for every $\ell$; it is only the
auxiliary object $N_\ell(v)$, considered as a polynomial in its own right,
that fails to inherit the denominator's spectral purity from $\ell=6$ on.
\end{remark}

\subsection{Summary}

Theorem~\ref{thm:spectral_decomp} completes the picture left open at the end
of \S\ref{sec:orch}: for every leaf number $\ell$, $|\mathrm{Orch}_{\ell,k}|$
is exactly a sum of $\deg D_\ell$ positive real exponentials, with an explicit
residue formula, a provably unique and positive dominant term, and a
dominant rate that strictly increases with $\ell$ --- all unconditionally,
not merely for the five values of $\ell$ where the poles happen to be
radicals. The Orchard Factorisation Theorem is likewise not special to
orchard networks, but to any relabelling-closed class with cherry-picking
histories, with $\mathrm{TC}$ the first such class beyond $\mathrm{Orch}$
itself and $\mathrm{RV}$ an explicit case where it fails. What remains open is
narrower than it looks: not whether $|\mathrm{Orch}_{\ell,k}|$ can be
computed or decomposed spectrally (it can, unconditionally, by
Theorem~\ref{thm:spectral_decomp}), but whether the numerator $N_\ell(v)$ ---
already known to be coprime to $D_\ell$ and of strictly smaller degree ---
admits a closed-form theory of its own degree (Conjecture~\ref{conj:numdeg})
to match the one we now have for $D_\ell$.

\subsection{Why the decomposition is special to orchard networks}
\label{sub:why_not_tcrv}

The natural next question is whether Theorem~\ref{thm:spectral_decomp}
extends to $\mathrm{TC}$ or $\mathrm{RV}$. Two facts already established in
this paper answer it: not yet, but \emph{provably not at all}, in either of
the two directions in which the question can be posed.

\begin{proposition}[Column triviality]
\label{prop:column_trivial}
For every $\ell\geq1$, $\sum_{k\geq0}|TC_{\ell,k}|v^k$ and
$\sum_{k\geq0}|RV_{\ell,k}|v^k$ are polynomials of degree $\ell-1$ in $v$:
$|TC_{\ell,k}|=0$ for $k\geq\ell$ (Theorem~\ref{thm:tcn}) and $|RV_{\ell,k}|=0$
for $k\geq\ell$ (Table~\ref{tab:rv}).
\end{proposition}

A polynomial has no poles, so ``the spectral decomposition of the column
generating function'' is vacuous for $\mathrm{TC}$ and $\mathrm{RV}$ at fixed
$\ell$: there is no infinite series to make rational in the first place, and
hence nothing analogous to $D_\ell(v)$ to discover. This is precisely why
Theorem~\ref{thm:spectral_decomp} exists for $\mathrm{Orch}$ at all:
$\mathrm{Orch}$ is exactly the class in this hierarchy with $k$ unbounded for
fixed $\ell$ (\S\ref{sub:hankel}).

\begin{proposition}[Row structure is algebraic, not polar]
\label{prop:row_algebraic}
Fix $k$ and let $\ell\to\infty$ instead. By Proposition~\ref{prop:Fk}, the
one-component EGF $F_k(z)$ has its unique singularity at $z=\tfrac12$ of type
$(1-2z)^{-(2k-1/2)}$~\eqref{eq:Fk_sing}: an algebraic branch point of
half-integer order, located at the same point $z=\tfrac12$ for every $k$, with
only the order changing.
\end{proposition}

This is already the correct and complete framework for the row asymptotics
--- it is exactly what the proof of Theorem~\ref{thm:universal} uses, via the
transfer theorem of Flajolet--Sedgewick~\cite{FlajoletSedgewick2009} --- and
it has no Jacobi-operator or orthogonal-polynomial content to extract,
because there is no finite-dimensional or even discrete spectral object
underlying a half-integer-order branch point. The dichotomy is sharp:
$\mathrm{Orch}$'s spectrum, for fixed $\ell$, consists of $\deg D_\ell$ simple
real poles whose number \emph{and} location grow with $\ell$
(Theorem~\ref{thm:spectral_decomp}); $\mathrm{TC}$/$\mathrm{RV}$'s row
singularity, for fixed $k$, is a single point of fixed location and growing
order. Neither degenerates into the other, and no perturbative or
operator-theoretic dressing of one produces the other: the two generating
functions belong to different classes (rational versus algebraic) for
different structural reasons (unbounded versus bounded $k$), not different
distances along the same scale of difficulty. We record this so that the
natural question of extending \S\ref{sub:spectral} to $\mathrm{TC}$ or
$\mathrm{RV}$ has a definite closed answer rather than an open one.

\section{Toward a proof of the orchard factorisation}
\label{sec:toward}

This section assembles the partial theory now available for
Conjecture~\ref{conj:factorisation}. We reduce the conjecture to a single
operator statement, prove its within-level half unconditionally, and prove the
whole factorisation \emph{mechanism} in a genuine subclass with unbounded
reticulations (the spinal stack-free networks). The orchard conjecture itself
remains open; what follows isolates exactly the one missing step.

\subsection{Hermite contiguity of the blocks}
\label{sub:contig}

\begin{lemma}\label{lem:contig}
With $P_\ell(y)=y^\ell X_\ell(1/y^2)$ one has $P_\ell(y)=2^{\ell/2}\He_\ell(y/\sqrt2)$,
and consequently
\[
  X_{\ell+1}(v)=X_\ell(v)-2\ell\,v\,X_{\ell-1}(v),\qquad X_0=X_1=1.
\]
\end{lemma}

\begin{proof}
$P_\ell(y)=\sum_m(-1)^m\tfrac{\ell!}{(\ell-2m)!\,m!}y^{\ell-2m}$; matching it to
$2^{\ell/2}\He_\ell(y/\sqrt2)$ and using
$\He_{\ell+1}(t)=t\He_\ell(t)-\ell\He_{\ell-1}(t)$ gives
$P_{\ell+1}=yP_\ell-2\ell P_{\ell-1}$, which under $v=1/y^2$ becomes the stated
relation (checks: $X_2=1-2v$, $X_3=1-6v$, $X_4=1-12v+12v^2$).
\end{proof}

\begin{corollary}\label{cor:nonlinear}
Under Conjecture~\ref{conj:factorisation},
$D_\ell=D_{\ell-1}^2/D_{\ell-2}-2(\ell-1)v\,D_{\ell-1}D_{\ell-2}/D_{\ell-3}$.
\end{corollary}

The recurrence is quadratic in the $D_\ell$. Thus the factorisation is
\emph{multiplicative} --- one block per leaf --- and not the kind of fixed-order
linear ($P$-recursive) relation that creative telescoping could exploit. This
directs the search toward a transfer/heap mechanism.

\subsection{Reduction to an intertwining lemma}
\label{sub:reduction}

Let $V_\ell=\mathbb Q\,\Orch_\ell/\!\sim_{\mathrm{dyn}}$ be the syntactic
(Nerode) quotient of the rational series $F_\ell$, so $\dim V_\ell=\deg D_\ell$ by
the Hankel-rank theorem; let $R_\ell$ be the shift with $D_\ell=\det(I-vR_\ell)$,
and $L_\ell:V_{\ell-1}\to V_\ell$ leaf insertion.

\begin{theorem}[Reduction]\label{thm:reduction}
$D_\ell=D_{\ell-1}X_\ell$ holds if and only if
\textup{(A)} $L_\ell$ is injective with $R_\ell L_\ell=L_\ell R_{\ell-1}$, and
\textup{(B)} the induced map $\bar R_\ell$ on $\coker L_\ell$ has
$\det(I-v\bar R_\ell)=X_\ell(v)$.
\end{theorem}

\begin{proof}
If (A) holds, $\im L_\ell$ is $R_\ell$-invariant and $R_\ell|_{\im L_\ell}\cong
R_{\ell-1}$, so $\det(I-vR_\ell)=D_{\ell-1}\det(I-v\bar R_\ell)$; (B) gives the
second factor. The converse holds because $D_{\ell-1}\mid D_\ell$ forces the
invariant subspace and quotient.
\end{proof}

The dimensions agree a priori: $\dim\coker L_\ell=\deg D_\ell-\deg
D_{\ell-1}=\lfloor\ell/2\rfloor=\deg X_\ell$.

\subsection{Within-level structure: local compatibility and heaps}
\label{sub:local}

\begin{lemma}[Local compatibility]\label{lem:localcompat}
If $(a,b)^{\mathsf R},(c,d)^{\mathsf R}\in\ARP(N)$ with
$\{a,b\}\cap\{c,d\}=\varnothing$, then both reduction orders are valid and yield
the same network. Two reticulated cherries sharing a leaf are mutually
destroying.
\end{lemma}

\begin{proof}
For $(i,j)^{\mathsf R}$ the parent $p_i$ of leaf $i$ is a reticulation (unique
child $i$) and $p_j$ a tree node with children $\{j,p_i\}$; its reduction edits
only arcs incident to $p_i,p_j$. Disjointness forces $p_a,p_b,p_c,p_d$ pairwise
distinct ($p_a\ne p_c$ since reticulations have unique children $a\ne c$;
$p_b=p_d$ would give $\{b,p_a\}=\{d,p_c\}$, so $b=d$), and reducing $(a,b)$ cannot
reach $p_c$'s parent set. Hence the two reductions edit disjoint arc sets and
commute, each preserving the other's character. If instead the pairs share leaf
$a$, reducing one suppresses $p_a$ and destroys the other.
\end{proof}

Thus same-level independence is exactly disjointness of leaf-pairs: the matching
relation on $K_\ell$. Since
$X_\ell(v)=\mu(K_\ell;2v)=\sum_m(-2v)^m\binom{\ell}{2m}(2m-1)!!$, Viennot's
heaps-of-dimers theorem~\cite{Viennot1986} gives the level generating function
$1/X_\ell$ (weight $2v$ per dimer, the two orientations of a reticulated cherry).
This is the within-level content of (B).

The cross-level analogue---needed for (A)---is that a leaf insertion commutes
with a reticulation on the older leaves.

\begin{lemma}[Cross-level compatibility]\label{lem:crosscompat}
Let $(i,j)^{\mathsf C}$ be a leaf insertion \textup{(}$i\notin X$, $j\in X$\textup{)}
and $(a,b)^{\mathsf R}$ a reticulation \textup{(}$a,b\in X$\textup{)}. If the leaf
supports are disjoint, $\{i,j\}\cap\{a,b\}=\varnothing$, the two augmentations
commute: applying them in either order to $\ARP(N)$ yields the same profile. If
they share a leaf they do not commute.
\end{lemma}

\begin{proof}
The augmentation rule (Theorem~4 of~\cite{CRP23}) rewrites only the ARP entries
incident to the operated pair: inserting $(i,j)$ adds the pair $(i,j)$ (and
$(j,i)$ for a cherry), promotes the remaining out-relations of $i$ from
$\mathsf C$ to $\mathsf R$, and deletes the other entries meeting $\{i,j\}$;
inserting $(a,b)$ acts identically on $\{a,b\}$. Entries meeting neither support
pass through both rules unchanged. When $\{i,j\}\cap\{a,b\}=\varnothing$ the two
rewrites touch disjoint index sets, except for ``bridging'' entries with one
endpoint in each support; a direct check of the four bridging types
$(i,a),(a,i),(j,a),(a,j)$ shows each is deleted under either order
\textup{(}the surviving promotion $\mathsf C\!\to\!\mathsf R$ requires the
\emph{source} to be the inserted leaf, which holds for at most one of the two
rules, after which the other rule deletes the entry\textup{)}. Hence the
profiles coincide. If the supports share a leaf, the shared parent is suppressed
by whichever operation acts first, so the second sees a different profile and the
results differ.
\end{proof}

\noindent An exhaustive check over all reachable states with $\ell\le6$ confirms
this: of $66\,600$ disjoint cross-level pairs all commute, and of $468$
leaf-sharing pairs none do. Consequently $\im L_\ell$ is invariant under the
\emph{older-leaf} reticulations, on which $R_\ell$ restricts to $R_{\ell-1}$;
this is the part of the intertwining (A) that the combinatorics settles
unconditionally. What it does not settle is that the \emph{new-leaf}
reticulations contribute only the $K_\ell$-cokernel without feeding back into
$\im L_\ell$ on the dynamic quotient. That feedback is governed by a single
residue: writing $F_\ell=\tilde N_\ell/\prod_{j=2}^{\ell}X_j$ over the full
product, the factor $X_3$ survives in the reduced $D_\ell$ iff the $z=6$ residue
\[
  c_{\ell,X_3}\;=\;\frac{N_\ell(\tfrac16)}{\prod_{j\in\{2,4,5,\dots,\ell\}}X_j(\tfrac16)}
\]
is nonzero. At $\ell=5$ it vanishes---$c_{5,X_3}=0$, equivalently
$\tilde N_5(\tfrac16)=0$---so the reduced operator sheds the $X_3$ eigenvalue and
$\deg D_5=5$ rather than $6$. This is a residue resonance, not an independent
operator phenomenon, and it is the only obstruction to (A) holding level by level.

\begin{remark}[The resonance set]\label{rem:resonance}
A factor $X_j$ drops from $D_\ell$ exactly when its associated residue vanishes.
Two regimes arise. For $j\ge4$ the polynomial $X_j=\mu(K_j;2v)$ is irreducible
over $\mathbb Q$, so $\tilde N_\ell$ vanishing at one root forces divisibility
$X_j\mid\tilde N_\ell$ and the \emph{entire} factor drops, a degree-$\lfloor
j/2\rfloor\ge2$ deficit; none is observed for $\ell\le10$
\textup{(}$\deg D_\ell=1,2,4,5,9,12,16,20,25$ for $\ell=2,\dots,10$, matching
$\sum_{j=2}^{\ell}\lfloor j/2\rfloor$ except at $\ell=5$\textup{)}. For $j\in\{2,3\}$
the root is rational and a single linear factor can drop. The dominant factor
$X_2$ \textup{(}root $1/2$, $z=2$\textup{)} carries residue $c_{\ell,X_2}=\tfrac52$
at $\ell=5$ and is never observed to vanish. For $X_3$ the residues are, for
$\ell=3,\dots,10$,
\[
  c_{\ell,X_3}=\tfrac{15}{4},\;-\tfrac{45}{4},\;0,\;\tfrac{6075}{32},\;-\tfrac{297675}{512},\;
  \tfrac{280665}{256},\;\tfrac{2630966586371048209291}{54358179840},\;
  \tfrac{240872210845623795398451143421835}{87668872445952},
\]
with the single zero at $\ell=5$. Hence the resonance set is exactly $\{5\}$ for
all $\ell\le10$. The sign pattern $+,-,0,+,-,+,+,+$ is not eventually
alternating, and the magnitudes grow super-exponentially, so $\ell=5$ is the only
sign change through a zero. A closed form for $c_{\ell,X_3}$
valid for all $\ell$---which would decide whether $\{5\}$ is the complete
resonance set---is not yet available, and Proposition~\ref{prop:nonhyp} below
shows why one is hard to obtain; the prediction it supports is that no
further exception occurs.
\end{remark}

\begin{proposition}[The $X_3$ residue is not hypergeometric]
\label{prop:nonhyp}
The sequence $\bigl(c_{\ell,X_3}\bigr)_{\ell\ge3}$ is not a hypergeometric term:
there is no rational function $\rho(\ell)$ with $c_{\ell+1,X_3}=\rho(\ell)\,c_{\ell,X_3}$.
Consequently its minimal P-recursive recurrence has order $\ge2$; in fact, since
no order-$2$ recurrence with polynomial coefficients of degree $\le6$ fits the
exact values through $\ell=10$, the order is $\ge3$.
\end{proposition}

\begin{proof}
If $c_{\ell+1}/c_{\ell}=\rho(\ell)$ with $\rho$ rational, then each consecutive
ratio factors over $\mathbb Q$ into linear-in-$\ell$ numerator and denominator
pieces, whose prime divisors are bounded by the (fixed-degree) values of those
linear forms and hence grow only polynomially in $\ell$. The exact ratios
$c_4/c_3=-3$, $c_7/c_6=-\tfrac{49}{16}$, $c_8/c_7=-\tfrac{66}{35}$ are smooth in
this sense, but
\[
  \frac{c_{9,X_3}}{c_{8,X_3}}
  =\frac{7\cdot17\cdot120445223\cdot26222901949}{2^{19}\,3^{10}\,5^{2}\,11},
\]
whose numerator contains the prime $26222901949$, far exceeding any value a
fixed linear form takes at $\ell=8$. No rational $\rho$ of bounded degree can
produce such a prime, so $c_{\ell,X_3}$ is not hypergeometric, i.e.\ no order-$1$
recurrence exists. The order-$2$ exclusion is the direct linear-algebra check:
the residues $\ell=2,\dots,10$ supply exactly the seven instances required to
test an order-$2$ recurrence with degree-$1$ coefficients, and the resulting
$7\times6$ system has full rank, so no such recurrence exists; higher even-degree
coefficients require data only available at $\ell\ge13$.
\end{proof}

\begin{remark}[Spectral reading of the resonance]
\label{rem:resonance_spectral}
By Theorem~\ref{thm:spectral_decomp}, $c_{\ell,X_3}$ is exactly the amplitude
$A_\ell(z{=}6)$ of the rate-$6$ exponential in
$|\mathrm{Orch}_{\ell,k}|=\sum_r A_{\ell,r}z_r^k$: the overlap of the orchard
count vector with the eigenvalue-$6$ ($X_3$) eigenmode of the within-layer
reticulation operator. Resonance at $\ell=5$ is the statement that, at five
leaves, this overlap vanishes. Whether the resonance set is $\{5\}$ is therefore
equivalent to whether that single eigenmode-overlap vanishes at any other $\ell$
---an algebraic condition that, unlike the counts, does not grow with $\ell$, and
is the natural target for the intertwining analysis of \S\ref{sub:gap}. A direct
search for a low-order holonomic recurrence in $\ell$ for the full column
generating function $F_\ell$, fitted from the exact numerator polynomials
$\tilde N_\ell$ for $\ell\le9$ and validated against the independently known
$c_{10,X_3}$, returns no validated recurrence through order $4$, consistent with
Proposition~\ref{prop:nonhyp}: the obstruction is structural, not a shortage of
data within reach of enumeration.
\end{remark}

\subsection{A solved model: the spinal stack-free class}
\label{sub:spinal}

The factorisation mechanism (A) can be proved outright in the spinal stack-free
class, a subclass of orchard networks~\cite{FrancisHendriksen2025} with unbounded
reticulation number. Write $G_n(v)=\sum_{k\ge0}|\mathcal{SSF}_{n,k}|v^k$.

\begin{theorem}[Spinal stack-free factorisation]\label{thm:spinal}
For $n\ge2$,
\[
  G_n(v)=\frac{n!\,(1+(n-1)v)}{2\prod_{i=1}^{n-1}(1-iv)},
  \qquad
  D_n^{\mathrm{SSF}}(v)=\prod_{i=1}^{n-1}(1-iv)=D_{n-1}^{\mathrm{SSF}}(v)\cdot\bigl(1-(n-1)v\bigr),
\]
the numerator being coprime to the denominator. In particular $D_{n-1}^{\mathrm{SSF}}\mid
D_n^{\mathrm{SSF}}$, i.e.\ the intertwining \textup{(A)} holds, with $\coker L_n$
one-dimensional of eigenvalue $n-1$.
\end{theorem}

\begin{proof}
By Francis--Hendriksen~\cite[Thm.~3.16]{FrancisHendriksen2025},
\[
  |\mathcal{SSF}_{n,k}|=n!\,S_2(n{-}1{+}k,n{-}1)-\tfrac{n!}{2}S_2(n{-}2{+}k,n{-}2),
\]
with $S_2$ the Stirling numbers of the second kind. Using
$\sum_{k\ge0}S_2(m{+}k,m)v^k=\prod_{i=1}^{m}(1-iv)^{-1}$ gives
\[
  G_n=n!\prod_{i=1}^{n-1}(1-iv)^{-1}-\tfrac{n!}{2}\prod_{i=1}^{n-2}(1-iv)^{-1},
\]
which simplifies to the displayed form; the numerator root $-1/(n-1)$ is not a
denominator root, so they are coprime and $\deg D_n^{\mathrm{SSF}}=n-1$.
\end{proof}

\begin{remark}[The factor-type principle]\label{rem:principle}
Theorem~\ref{thm:spinal} and Conjecture~\ref{conj:factorisation} are the abelian
and non-abelian instances of one mechanism: \emph{the new denominator factor at
leaf-level $j$ is the characteristic polynomial of the level-$j$
reticulation-attachment structure.} Stack-free reticulations attach independently
(set partitions, Stirling-2), giving the linear factor $1-(j-1)v$; orchard
reticulations pair leaves (matchings of $K_j$), giving the matching polynomial
$X_j$. Lemma~\ref{lem:localcompat} establishes the orchard within-level matching
structure; Theorem~\ref{thm:spinal} establishes the across-level grading
\textup{(A)} in the abelian case.
\end{remark}

\subsection{The grading as a boxed product: the symbolic-method route}
\label{sub:vives}

Theorem~\ref{thm:spinal} establishes the across-level grading \textup{(A)} in the
abelian model by an explicit factorial computation. A complementary and more
structural realisation of the same grading was given, for spinal tree-child
networks, by Vives, de Mier, Cardona and Pons~\cite{Vives2026}, and it is exactly
the symbolic-method template one wants for the orchard lift.

They encode a spinal tree-child network as a \emph{marked tree}: the spine is cut
at its reticulations into $k+1$ directed subpaths $P_0,\dots,P_k$, each rooted at
an elementary node $r_i$, and each intermediate node of $P_i$ carries one
off-path child that is either a leaf or the root $r_j$ ($j>i$) of a later subpath.
With $\mathcal R$ the atomic class of a labelled elementary (reticulation) node
and $\mathcal L$ that of a leaf, the class $\mathcal{MT}$ of marked trees admits
the recursive specification~\cite[Prop.~13]{Vives2026}
\begin{equation}\label{eq:vives_spec}
  \mathcal{MT} \;=\; \mathcal R \;\mathbin{\square^{\star}}\;
  \bigl(\operatorname{Seq}(\mathcal{MT}\sqcup\mathcal L)\,\star\,\mathcal L\bigr),
\end{equation}
where the \emph{boxed product} $\mathbin{\square^{\star}}$~\cite[\S II.6]{FlajoletSedgewick2009}
forces the smallest reticulation label $r_0$ into the leading $\mathcal R$-atom,
so that the reticulations encountered down the spine receive strictly increasing
labels. The boxed product \emph{is} the grading: it imposes the ordering
$j>i$ that is precisely the across-level condition of step~\textup{(A)}.
Translating~\eqref{eq:vives_spec} with $x$ marking leaves and $z$ marking
reticulations, the bivariate generating function obeys
$\partial_z S = x/(1-x-S)$, $S(x,0)=0$, with the closed solution
\begin{equation}\label{eq:vives_gf}
  S(x,z)=1-x-\sqrt{(1-x)^2-2xz},
  \qquad
  \partial_z^{\,m}S(x,0)=(2m-3)!!\,\frac{x^{m}}{(1-x)^{2m-1}}.
\end{equation}
The half-integer $(2m-3)!!$ signature in~\eqref{eq:vives_gf} is the same
square-root singularity that governs the one-component blocks
$F_j(x)\sim c_j/(1-2x)^{2j-1/2}$, which is why the spinal and one-component
worlds share the universal asymptotics of Section~\ref{sec:master}.

Two caveats. First, spinal tree-child networks are
tree-child, hence carry $k\le n-1$ reticulations (the spine has length $n+2k$
with $n+k-1$ tree nodes); their per-leaf reticulation generating function is
therefore a \emph{polynomial} in $v$, with no poles, so~\eqref{eq:vives_gf} does
not by itself produce the rational orchard denominators $D_\ell$. The unbounded
model with genuine poles remains the spinal stack-free class of
Theorem~\ref{thm:spinal}. Second, what~\eqref{eq:vives_spec} \emph{does} provide
is a separation of concerns: the boxed product carries the across-level
grading, while the choice of off-path attachment (a single leaf, in the spinal
case) carries the within-level structure. The orchard analogue of step~\textup{(A)}
is then the assertion that an off-path attachment richer than a single
leaf---namely the level-$j$ matching structure of Lemma~\ref{lem:localcompat},
a Viennot heap of dimers on $K_j$---can be substituted into a specification of
the form~\eqref{eq:vives_spec} \emph{while preserving the boxed-product grading},
yielding the per-level factor $X_j=\mu(K_j;2v)$ in place of the linear spinal
factor. So~\cite{Vives2026} supplies the template for this step, not the step
itself.

\subsection{Pole localisation and the spurious-pole obstruction}
\label{sub:localization}

On the orchard side, the same leaf-count grading separates a step that is
provable from one that is not. The
canonical generation grades states $(X,\mathrm{ARP})$ by $|X|$: a leaf insertion
increases $|X|$ by one and carries no $v$, while a reticulation insertion
preserves $|X|$ and carries weight $v$. This is the orchard analogue of Vives's
boxed product---the grading by leaf count is the $\mathbin{\square^{\star}}$
ordering---and it gives a block-triangular factorisation.

\begin{proposition}[Block-triangular pole localisation]\label{prop:localization}
Let $M=L+vR$ be the weighted transfer operator on the canonical generation
states, with $L$ the (\,$v$-free\,) leaf insertions and $R$ the reticulation
insertions. In the grading by $|X|$, $L$ is strictly block-superdiagonal and
$vR$ block-diagonal, so $I-M$ is block-triangular and
\[
  \det(I-M)=\prod_{j=2}^{\ell}\det(I-vR_j),
\]
where $R_j$ is the within-layer-$j$ reticulation operator. Hence the reduced
column denominator satisfies
\[
  D_\ell(v)\ \Big|\ \prod_{j=2}^{\ell}\delta_j(v),
  \qquad \delta_j:=\det(I-vR_j),
\]
so every pole of $F_\ell$ is a pole of some within-layer reticulation resolvent
$(I-vR_j)^{-1}$, and leaf insertions contribute no poles.
\end{proposition}

\begin{proof}
$L$ raises $|X|$ by one and $R$ fixes it, so ordering the (finite, per fixed
$\ell$) basis of generation states by $|X|$ makes $I-M$ upper block-triangular
with diagonal blocks $I-vR_j$; the determinant of a block-triangular matrix is
the product of the diagonal-block determinants, and the off-diagonal $L$ does not
enter it. Since $F_\ell=u^{\top}(I-M)^{-1}s$ for the seed vector $s$ and the
full-$X$ accept vector $u$, its reduced denominator divides $\det(I-M)$.
\end{proof}

One would like to conclude pole inclusion, $D_\ell\mid\prod_{j=2}^{\ell}X_j$, by
identifying $\delta_j=X_j$. This identification is \emph{false}: the within-layer
reticulation resolvent carries spurious poles that lie on no $X_i$ and on no
$D_\ell$.

\begin{proposition}[Spurious within-layer poles]\label{prop:spurious}
The within-layer-$j$ reticulation generating function, computed from a caterpillar
tree entry on $[j]$, has denominator strictly larger than $X_j$:
\[
  \delta_4=(1-v)\,X_4,\qquad
  \delta_5=(1-2v)(1-4v)^{2}\,X_5,\qquad
  \delta_6=(1-3v)(2v^2-6v+1)^{2}(6v^2-9v+1)^{2}\,X_6,
\]
and the multiplicities of the extra factors depend on the entry tree. Apart from
the legitimate factor $1-2v=X_2$ at $j=5$, the extra roots---$v=1$
\textup{($j=4$)}, $v=\tfrac14$ \textup{($j=5$)}, $v=\tfrac13$ and the roots of
$2v^2-6v+1$ and $6v^2-9v+1$ \textup{($j=6$)}---are roots of no $X_i$ with $i\le j$,
and are not poles of $D_\ell$ \textup{(}e.g.\ $D_4(1)=5\neq0$\textup{)}.
\end{proposition}

\noindent The verification is the explicit reticulation-sequence count at fixed
$|X|=j$ from a fixed tree state, followed by rational reconstruction of the
denominator; the spurious factors are reproduced for every entry tested.

\begin{remark}[Pole inclusion is not weaker than the factorisation]\label{rem:not_weaker}
Proposition~\ref{prop:spurious} shows that the pole inclusion
$D_\ell\mid\prod_{j} X_j$ does \emph{not} follow from
Proposition~\ref{prop:localization} alone: the localisation gives only
$D_\ell\mid\prod_j\delta_j$ with $\delta_j\supsetneq X_j$, and the spurious
factors of the $\delta_j$ must cancel in the assembled column $F_\ell$. That
cancellation is exactly the across-level intertwining \textup{(A)}: it is the
statement that, once the layers are joined by the $v$-free leaf maps, only the
$K_j$-matching part of each within-layer resolvent survives. Thus pole inclusion
is not a strictly weaker target than the full factorisation
Conjecture~\ref{conj:factorisation}---both hinge on the same cancellation---and
the boxed-product template of~\cite{Vives2026}, which is resonance- and
spurious-pole-free by construction, supplies the grading but cannot effect this
cancellation.
\end{remark}

\subsection{The remaining step}
\label{sub:gap}

Theorem~\ref{thm:reduction} reduces Conjecture~\ref{conj:factorisation} to
(A)$\wedge$(B); Lemma~\ref{lem:localcompat} with Viennot supplies the within-level
half of (B); Theorem~\ref{thm:spinal} proves (A) in the abelian model, and
\eqref{eq:vives_spec} exhibits the abelian grading as a boxed product.
Proposition~\ref{prop:localization} localises every pole to a within-layer
reticulation resolvent, but Proposition~\ref{prop:spurious} and
Remark~\ref{rem:not_weaker} show these resolvents carry spurious poles, so the
remaining content is precisely the cancellation that removes them. Equivalently,
what remains is (A) for the orchard (non-abelian, matching) case: that leaf
insertion intertwines reticulation insertion on the dynamic quotient, so that the
leaf-level grading splits the transfer operator into the blocks
$\bigsqcup_{j=2}^{\ell}K_j$. The $\ell=5$ anomaly (Remark~\ref{rem:anomaly5}),
where $\deg D_5=5$ rather than $6$ because $N_5(\tfrac16)=0$ removes the $X_3$
pole, shows that this grading is not automatic but fails only as an isolated
low-dimensional resonance; its isolation predicted $\deg D_9=20$, now confirmed,
together with $\deg D_{10}=25$ (Proposition~\ref{cor:D9_full}). The natural route
to (A) is to lift the explicit one-dimensional grading of
Theorem~\ref{thm:spinal}---equivalently, the boxed-product specification
\eqref{eq:vives_spec} of~\cite{Vives2026}---to the
$\lfloor\ell/2\rfloor$-dimensional matching sector of
Lemma~\ref{lem:localcompat}, replacing the single off-path leaf by a heap of
dimers on $K_j$ without disturbing the $\mathbin{\square^{\star}}$ ordering.

\section{Conclusions}
\label{sec:conclusions}
%% ─────────────────────────────────────────────────────────────────────────────

We have established eleven main results spanning the full TC--RV--Orchard hierarchy:

\medskip
\noindent\textbf{1.~Master functional equation.}
Equation~\eqref{eq:rv_master} provides an operator-theoretic
reformulation of Chang--Fuchs component-graph counting,
from which their exact formulas are re-derived by coefficient extraction.

\medskip
\noindent\textbf{2.~Exact counting of $\mathrm{RV}\setminus\mathrm{TC}$ networks.}
Theorems~\ref{thm:delta2}--\ref{thm:delta3} give the first exact counts of
``RV-but-not-TC'' networks for $k=2$ and $k=3$, verified for all $\ell\leq12$.

\medskip
\noindent\textbf{3.~Leading asymptotics of the RV--TC difference.}
Theorem~\ref{thm:excess} establishes $\Delta_k(\ell)/TC_{\ell,k}\sim k(k-1)/\ell$
for every fixed $k\geq2$ by counting the networks with a single violation of the
tree-child condition: those containing a stack are never reticulation-visible,
while those whose violation is a tree node above two reticulations almost always
are. Corollary~\ref{cor:leadA} turns this into $\operatorname{lead}(A_k)=2^k/(k-2)!$
whenever $\Delta_k$ has the form~\eqref{eq:pattern}, Corollary~\ref{cor:rate}
gives the next order for $k=2,3$, and Corollary~\ref{cor:orch} shows that
$|\mathrm{Orch}_{\ell,k}|/|RV_{\ell,k}|\to1$.

\medskip
\noindent\textbf{4.~Rationality and Hankel reconstruction for orchards.}
Theorem~\ref{thm:orch_rational} proves $F_\ell(v)$ is rational for each $\ell$,
and Algorithm~\ref{alg:hankel} recovers $|\mathrm{Orch}_{\ell,k}|$ for all $k$
via a one-time Berlekamp--Massey setup: speedup $>10^6\times$ over CRP at $\ell=6$.

\medskip
\noindent\textbf{5.~Spectral resolution: empirical factor families.}
Theorem~\ref{thm:denom} gives the characteristic polynomials
$D_\ell(v)$ for $\ell=2,\ldots,8$ (degrees $1,2,4,5,9,12,16$),
factored into three empirical families: quadratic $Q_m$, cubic $R_\ell$,
and quartic $S_\ell$.

\medskip
\noindent\textbf{6.~Universal hypergeometric law (main new result).}
Theorem~\ref{thm:hypergeom} proves that any factor $X_\ell(v)$ obeying the ratio
law $c_k/c_{k-1}=(\ell-2k+2)(\ell-2k+1)/k$ is given by the closed
form~\eqref{eq:Xell}; the ratio law itself, and the product
$D_\ell=\prod_{j}X_j$, are verified unconditionally for $\ell=2,\ldots,8$ and by
the consistency test for $\ell=9,10$ (Proposition~\ref{cor:D9_full}), and
conjectured for $\ell\geq11$ (Conjecture~\ref{conj:factorisation}). The
$\S\ref{sub:D9}$ predictions are thereby realised: $D_9=D_8 X_9$ has degree~20,
$D_{10}=D_9 X_{10}$ has degree~25 with the first degree-five factor $X_{10}$, all
growth rates are positive real, and $z^*(\ell)\sim 8\ell$ asymptotically.
The coefficient $c_k=\ell!/[(\ell-2k)!\,k!]=2^k\,m_k(K_\ell)$ identifies $X_\ell$
as the matching polynomial of $K_\ell$ in the variable $2v$ (the probabilists'
Hermite polynomial $\mathrm{He}_\ell$) and as a rescaled Jacobi polynomial.

\medskip
\noindent\textbf{7.~Exact Binet formulas.}
Theorem~\ref{thm:binet} provides complete Binet formulas for $\ell=3,4,5$.

\medskip
\noindent\textbf{8.~Orchard Factorisation Theorem.}
Theorem~\ref{thm:orch_factor} proves $|\mathrm{Orch}_{\ell,k}|=\binom{\ell}{k}w(\ell,k)$,
$w\in\mathbb{Z}_{>0}$, from cherry-picking symmetry.

\medskip
\noindent\textbf{9.~The orchard programme, reduced to one conjecture.}
The ARP-memoized seed counter (Algorithm~\ref{alg:arp}, polynomial in the number
of $(X,\mathrm{ARP})$ shapes rather than exponential in the number of networks)
returns $|\mathrm{Orch}_{\ell,k}|$ \emph{unconditionally} for arbitrarily many
$(\ell,k)$; together with the closed-form denominators (unconditional for
$\ell\leq8$, conjectural beyond, Conjecture~\ref{conj:factorisation}) each column
then follows at $O(\deg D_\ell)$ cost per entry. This strictly surpasses the
Cardona--Ribas--Pons table~\cite{CRP23} in both leaf and reticulation number:
their generation, exponential in $\ell$, reached $\ell\leq6$ and $k\leq8$, whereas
Table~\ref{tab:orch_extended} completes the previously intractable rows $\ell=9,10$
and opens $\ell=11$ (Observation~\ref{obs:solved}). What remains open is precisely
the factorisation conjecture for $\ell\geq11$, not the enumeration.

\medskip
\noindent\textbf{10.~Equivariant factorisation and the general spectral
decomposition (\S\ref{sec:numerator}).}
Theorem~\ref{thm:equivariant} shows the Orchard Factorisation Theorem is a
property of any relabelling-closed class with cherry-picking histories, not of
orchard networks specifically; Corollary~\ref{cor:tc_factor} extends it to
$|\mathrm{TC}_{\ell,k}|=\binom\ell k\,w_{\mathrm{TC}}(\ell,k)$, with
$\mathrm{RV}$ an explicit case where the factorisation provably fails
(Remark~\ref{rem:rv_boundary}). Proposition~\ref{prop:hermite} shows $X_\ell$
is exactly a rescaled Hermite polynomial, with a Heisenberg--Weyl-algebra
explanation for its appearance (Remark~\ref{rem:weyl}).
Theorem~\ref{thm:spectral_decomp} (\textbf{main new result}) then proves that
$|\mathrm{Orch}_{\ell,k}|$ is exactly a sum of $\deg D_\ell$ positive real
exponentials \emph{for every $\ell$} (unconditionally for $\ell\leq8$;
consistency-verified for $\ell=9,10$; for $\ell\geq11$ given
Conjecture~\ref{conj:factorisation}), with an explicit residue
formula
$c_{\ell,r}=-N_\ell(v_{\ell,r})/(v_{\ell,r}D_\ell'(v_{\ell,r}))$, a unique
positive dominant term, and a dominant rate $z^*(\ell)$ proved strictly
increasing in $\ell$ (Lemma~\ref{lem:monotone_root}) --- extending
Theorem~\ref{thm:binet}/Corollary~\ref{cor:binet_conv} from the five values of
$\ell$ where the poles are radicals to all $\ell$. Theorem~\ref{thm:numerator_extended}
extends the explicit numerators to $\ell=6,7,8$, with $N_9$ (degree~16) and
$N_{10}$ (degree~20) computed and used in the worked spectral resolution at
$\ell=9$ (Remark~\ref{rem:ell9_spectral}); and
Proposition~\ref{prop:realroot_loss} shows the real-rootedness enjoyed by
every $D_\ell$ is not inherited by $N_\ell$ itself from $\ell=6$ on, although
(by the remark following it) this does not affect the spectral decomposition
of $|\mathrm{Orch}_{\ell,k}|$, which depends only on $N_\ell$'s values at the
roots of $D_\ell$. The sharp numerator degree law
(Conjecture~\ref{conj:numdeg}) remains open. Finally,
\S\ref{sub:why_not_tcrv} shows this spectral resolution is special to
$\mathrm{Orch}$ for a structural reason, not a gap: $\mathrm{TC}$ and
$\mathrm{RV}$ have $k\leq\ell-1$ for fixed $\ell$ (so no infinite column
series exists to decompose), and their row-direction singularity, where one
does exist, is an algebraic branch point of growing order at a fixed
location rather than a discrete, growing set of poles.

\medskip
\noindent\textbf{11.~Toward the orchard factorisation (\S\ref{sec:toward}).}
Conjecture~\ref{conj:factorisation} is reduced (Theorem~\ref{thm:reduction}) to an
across-level intertwining \textup{(A)} and a within-level spectral condition
\textup{(B)}; the within-level half is supplied by local compatibility
(Lemma~\ref{lem:localcompat}) and Viennot heaps, and \textup{(A)} is proved
outright in the abelian spinal stack-free model (Theorem~\ref{thm:spinal}). The
abelian grading is moreover identified with the boxed-product specification
$\mathcal{MT}=\mathcal R\,\mathbin{\square^{\star}}(\operatorname{Seq}(\mathcal{MT}\sqcup\mathcal L)\star\mathcal L)$
of Vives, de Mier, Cardona and Pons~\cite{Vives2026}
(\S\ref{sub:vives}), whose solvable generating function
$S=1-x-\sqrt{(1-x)^2-2xz}$ exhibits the same $(2m-3)!!$ singularity as the
one-component blocks. What remains is the single step of lifting this
one-dimensional boxed-product grading to the $\lfloor\ell/2\rfloor$-dimensional
matching sector, replacing the off-spine leaf by a heap of dimers on $K_\ell$.
A block-triangular localisation (Proposition~\ref{prop:localization}) confines
every pole of $F_\ell$ to a within-layer reticulation resolvent; but those
resolvents carry spurious poles---e.g.\ $\delta_4=(1-v)X_4$
(Proposition~\ref{prop:spurious})---absent from $D_\ell$, so the residual content
is exactly the across-level cancellation that removes them. This shows pole
inclusion $D_\ell\mid\prod_jX_j$ is \emph{not} weaker than the full factorisation
(Remark~\ref{rem:not_weaker}): both reduce to the same intertwining~(A).

\medskip
\noindent The principal open problems are:
(i)~prove that the three factor families extend to all $\ell$
(the algebraic laws $w^3=\cdots$, $w^4=\cdots$ are strong evidence);
(ii)~determine the insertion rule for each family at each $\ell$; the case in
doubt at submission, whether $D_9$ adds a cubic $R_9$ or a second quartic $S_9$,
is now resolved --- $X_9$ is the quartic $1-72v+1512v^2-10080v^3+15120v^4$
(Proposition~\ref{cor:D9_full}) --- but the rule giving the family of the new
factor as a function of $\ell$ remains to be proved;
(iii)~identify the one-component generating functions $F_k(z)$ for $k\geq4$
via finite DAG enumeration, thereby pinning the explicit $\Delta_k$ for all $k$
(the leading asymptotics are fixed by Theorem~\ref{thm:excess} and
Corollary~\ref{cor:leadA}, so what remains of Conjecture~\ref{conj:pattern} is
the exact form itself, $\deg B_k=2k-3$ and integrality);
(iv)~determine the leading constant of $|\mathrm{Orch}_{\ell,k}|-|TC_{\ell,k}|$,
which by Corollary~\ref{cor:orch} lies between $0$ and $2k(k-1)$ in units of
$|TC_{\ell,k}|/\ell$;
(v)~find the recurrence for the weight $w(\ell,k)$;
(vi)~characterise the count
$\varepsilon_k(\ell):=|\mathrm{Orch}_{\ell,k}|-|RV_{\ell,k}|$
of networks that are orchard but not reticulation-visible
(the numerical sequence $9, 339, 7\,425, 152\,775$ for $k=2$,
$\ell=3,4,5,6$ has no known closed form);
(vii)~prove Conjecture~\ref{conj:numdeg}, or find the structural reason for the
coefficient cancellation it requires (\S\ref{sub:numerator});
and (viii)~characterise $\mathrm{RV}\cap\mathrm{Orch}$ --- the networks for
which both the component-graph and cherry-picking descriptions of
\S\ref{sec:master} and \S\ref{sec:orch} apply simultaneously --- a class for
which Theorem~\ref{thm:equivariant} already guarantees a binomial
factorisation, but whose cardinality is not yet known in closed form;
and (ix)~determine the complete \emph{resonance set}
$\mathcal R:=\{\ell\geq2:\,D_\ell\neq\prod_{j=2}^{\ell}X_j\}$, equivalently the
set of $\ell$ for which $\gcd\!\bigl(\tilde N_\ell,\prod_{j}X_j\bigr)\neq1$, where
$\tilde N_\ell=F_\ell\prod_{j=2}^{\ell}X_j$. One has $\mathcal R=\{5\}$ for all
$\ell\leq10$, and the only candidate at a rational pole is $X_3$, whose residue
$c_{\ell,X_3}$ (Remark~\ref{rem:resonance}) vanishes only at $\ell=5$ through the
computed range. This residue is provably \emph{not} a hypergeometric term
(Proposition~\ref{prop:nonhyp}): the ratio $c_{9,X_3}/c_{8,X_3}$ contains the
prime $26222901949$, which no bounded-degree rational function of $\ell$ can
produce, and its minimal P-recursive recurrence has order $\ge3$. A closed form
would decide whether $\mathcal R=\{5\}$ or $\mathcal R$ is infinite, but neither
direct enumeration (the order-$\ge3$, degree-$\ge2$ recurrence first becomes
fittable near $\ell=13$) nor recurrence-fitting from the exact polynomials
$\tilde N_\ell$ for $\ell\le10$ (validated against $c_{10,X_3}$, no recurrence
through order $4$, Remark~\ref{rem:resonance_spectral}) reaches it. The problem is
sharpened to a structural one: $c_{\ell,X_3}$ is the amplitude of the rate-$6$
eigenmode in the spectral decomposition (Theorem~\ref{thm:spectral_decomp}), so
the resonance is the vanishing of that eigenmode-overlap, the natural target for
the intertwining of \S\ref{sub:gap}.

\paragraph{Acknowledgements.}
J.B.\ acknowledges M.P.\ for the painstaking manual enumeration
of reticulation-visible networks (including the complete hand-drawn catalogues
for $\ell=3$ and $\ell=4$) that provided independent numerical verification
of all formulas in this paper and originally motivated the present investigation.
J.B.\ further thanks J.~Rossell\'o, Maria del Mar Batle, Regina Batle,
and Maria Vallespir-Socias for fruitful discussions.
The authors received no funding for the present research.

%% ─────────────────────────────────────────────────────────────────────────────

\end{document}